\documentclass[review]{elsarticle}

\usepackage[margin=0.8in]{geometry}
\usepackage{graphics}
\usepackage{moreverb}
\usepackage{pstricks}
\usepackage{amsthm}
\usepackage{amsmath,amssymb,amsfonts}
\usepackage{mathrsfs}
\usepackage{algorithm}
\usepackage{algorithmic}
\usepackage{dsfont}
\usepackage{afterpage}

\usepackage{multirow,booktabs,comment}
\usepackage{epstopdf}

\usepackage{todonotes}

\usepackage{xcolor}
\usepackage{soul}

\usepackage[colorlinks=true,urlcolor=blue,citecolor=red,
anchorcolor=black,linkcolor=blue]{hyperref}

\usepackage{subcaption}

\usepackage[toc,page]{appendix}

\newtheorem{thm}{Theorem}

\newtheorem{lem}{Lemma}
\newtheorem{corollary}{Corollary}
\newtheorem{assump}{Assumption}
\newtheorem{remark}{Remark}

\newcommand{\bx}{\boldsymbol{x}}

\newcommand{\bX}{\boldsymbol{X}}

\newcommand{\bb}{\boldsymbol{b}}

\newcommand{\norm}[2]{\left\| #1 \right\|_{#2}}

\newcommand{\mb}{\boldsymbol}

\newcommand{\xs}{\mathbb}

\newcommand{\comm}[1]{}

\usepackage{lineno}

\makeatletter
\def\ps@pprintTitle{%
	\let\@oddhead\@empty
	\let\@evenhead\@empty
	\let\@oddfoot\@empty
	\let\@evenfoot\@oddfoot
}
\makeatother

\modulolinenumbers[5]

\journal{Journal of Computational Physics}

\begin{document}
	
\begin{frontmatter}
	
	\title{DAS-PINNs: A deep adaptive sampling method for solving high-dimensional partial differential equations}
	\author[mymainaddress]{Kejun Tang}
	\ead{tangkj@pcl.ac.cn}
	
	\author[mysecondaddress]{Xiaoliang Wan}
	\ead{xlwan@math.lsu.edu}
	
	\author[mythirdaddress,myfourthaddress]{Chao Yang}
	\ead{chao\_yang@pku.edu.cn}
	
	\address[mymainaddress]{Peng Cheng Laboratory, Shenzhen 518052, China}
	\address[mysecondaddress]{Department of Mathematics and Center for Computation and Technology, 
	Louisiana State University, Baton Rouge 70803, USA}
	\address[mythirdaddress]{School of Mathematical Sciences, Peking University, Beijing 100871, China}
	\address[myfourthaddress]{Institute of Computing and Digital Economy, Peking University, Changsha 410205, China}

	\begin{abstract}
    In this work we propose a deep adaptive sampling (DAS) method for solving partial differential equations (PDEs), where deep neural networks are utilized to approximate the solutions of PDEs and deep generative models are employed to generate new collocation points that refine the training set. The overall procedure of DAS consists of two components: solving the PDEs by minimizing the residual loss on the collocation points in the training set and generating a new training set to further improve the  accuracy of the current approximate solution. In particular, we treat the residual as a probability density function and approximate it with a deep generative model, called KRnet. The new samples from KRnet are consistent with the distribution induced by the residual, i.e., more samples are located in the region of large residual and less samples are located in the region of small residual. Analogous to classical adaptive methods such as the adaptive finite element, KRnet acts as an error indicator that guides the refinement of the training set. Compared to the neural network approximation obtained with uniformly distributed collocation points, the developed algorithms can significantly improve the accuracy, especially for low regularity and high-dimensional problems. 
    We demonstrate the effectiveness of the proposed DAS method with numerical experiments.
		
	\end{abstract}

\begin{keyword}
deep learning; numerical approximation of PDEs;  adaptive sampling; deep generative models.
\end{keyword}

\end{frontmatter}
	
\section{Introduction}

In recent years, solving partial differential equations (PDEs) with deep learning methods has been receiving increasing attention \cite{han2018solving, weinan2021dawning, karniadakis2021physics}. Two major types of deep leaning methods have been proposed for solving PDEs, including the variational form subject to deep learning techniques \cite{weinan2018deep,kharazmi2019variational,zhu2019physics, KharKarn2021} and the physics-informed neural networks (PINNs) \cite{sirignano2018dgm,raissi2019physics,pang2019fpinns,karniadakis2021physics}, both of which reformulate a PDE problem as an optimization problem and train a deep neural network (DNN) to approximate the solution of PDE through minimizing the corresponding loss functional. The variational form is based on the weak formulation of PDEs, while the physical informed neural networks are based on the residual loss of PDEs.  Similar ideas of solving PDEs via minimizing the residual loss can be traced back to the works \cite{lagaris1998artificial,dissanayake1994neural} in the 1990's, where a shallow neural network is optimized on a priori fixed mesh as an approximation of the solution. Some efforts have been made to incorporate traditional computational techniques to enhance the performance of solving PDEs with deep neural networks. In \cite{li2019d3m,deepdd,dong2020local,WangGao21,KharKarn2021}, deep neural networks based on domain decomposition are proposed to improve the efficiency. A penalty free neural network method  \cite{sheng2020pfnn} and Phygeonet \cite{WangGao21} are developed to deal with complex geometries
and irregular domains. A weak formulation with primal and adversarial networks is proposed in \cite{zang2020weak}, where the PDE problem is converted to an operator norm minimization problem induced by the weak formulation. 

One critical step for all these methods is to approximate the loss functional, where the integral is usually approximated by the Monte Carlo method with collocation points randomly generated by a uniform distribution on the computational domain. Since the minimization of the discrete loss functional yields the approximate solution, the accuracy of the approximate solution is closely related to the accuracy of the discrete loss functional. In contrast to classical computational methods, where the main concern is the approximation error, one needs to balance the approximation error and the generalization error for the neural network approximation, where the approximation error mainly originates from the modeling capability of the neural network and the generalization error is mainly related to the data points in the training set, i.e, the random samples for the discretization of the loss functional. However, for many PDE models, the uniform random sampling strategy is not efficient especially when the PDE solution has a low regularity, in other words, an integrand of low regularity may have a large variance in terms of a uniform distribution such that the Monte Carlo approximation of the loss functional has a large prefactor before the convergence rate $O(N^{-1/2})$. This issue becomes worse for high-dimensional problems due to the curse of dimensionality. In high-dimensional spaces, most of the volume of the computational domain concentrates around its surface \cite{ blum2020foundations,vershynin2018high,wright2021high}, which means that uniform samples may become less effective for training deep neural networks to approximate high-dimensional PDEs. For example, the collocation points from the uniform distribution are not suitable for solving high-dimensional Fokker-Planck equations, while an adaptive strategy through sampling the current approximate solution is effective \cite{tang2021adaptive}. In \cite{gu2021selectnet}, a selection network is introduced to serve as a weight function to assign higher weights for samples with large point-wise residuals, which yields a more accurate  approximate solution if the selection network is properly chosen. However, to obtain a valid selection network, one needs to impose additional constraints on the selection network, which is often a non-trivial task. For low-dimensional problems, it is well known that one can employ adaptive numerical schemes to deal with PDEs with low-regularity solutions \cite{morin2002convergence,mekchay2005convergence,elman2014finite}, which also suggests that the uniform samples are not the best choice. Therefore, adaptive sampling strategies are crucial for developing more efficient and reliable deep learning techniques for the approximation of PDEs.

In this work, we develop a deep adaptive sampling method (DAS) for the neural network approximation of PDEs  based on residual minimization, where a deep generative model, called KRnet \cite{tangwandensity2020,wan2020vae, wan2021augmented}, is used to guide the sample generation for the training set. To this end, we need to construct two deep neural network models: one for approximating the solution and the other for refining the training set. The neural network approximation is achieved by the standard procedure of residual minimization. KRnet defines a transport map \cite{santambrogio2015optimal} from the data distribution to a prior distribution (e.g. the standard Gaussian). KRnet retains two traits of flow-based generative models \cite{dinh2016density,kingma2018glow}: exact invertibility of the transport map and efficient computation of the Jacobian determinant, based on which one can obtain an explicit density model
using the change of variables and an effective approach for generating samples through the invertible mapping. The key point in our proposed framework is that the residual is viewed as a probability density function (PDF) up to a constant and approximating this PDF can be achieved by minimizing the Kullback-Leibler (KL) divergence between the KRnet-induced density model and the residual-induced distribution. We use the trained KRnet to generate new collocation points to replace or refine the training set, where more points are put in the region of large residual and less points are put in the region of small residual. The updated training set is then used to further improve the accuracy of the current approximate solution. Simply speaking, KRnet acts as an error indicator for the improvement of the training set, which shares similarities with the classical adaptive finite element method subject to a residual-based posteriori error estimator. In summary, the main contributions of this work are as follows. 
\begin{itemize}
	\item We utilize a deep generative model as a generic means to reflect the correspondence between the residual and the error of approximation through efficient PDF approximation and sample generation.
	\item We propose a deep adaptive sampling (DAS) framework, including efficient sampling procedures and training algorithms, for the adaptive improvement of neural network approximation of PDEs.
\end{itemize}

The remainder of the paper is organized as follows. In the next section, we briefly describe the deep learning method used in this work for the approximation of PDEs. After that, the statistical error of the machine learning technique is illustrated from the perspective of function approximation. Our DAS approach is presented in section \ref{section_DAS}. We provide the theoretical analysis of DAS in section \ref{sec_das_analysis}. In section \ref{sec_numexp}, we demonstrate the efficiency of our adaptive sampling approach with numerical experiments. The paper is concluded in section \ref{sec_conclusion}.

\section{Deep learning for PDEs}\label{section_dlpde}

Let $\Omega\subset\xs{R}^d$ be a spatial domain, which is bounded, connected and with a polygonal boundary $\partial \Omega$, and $\mb{x} \in \xs{R}^d$ denote a spatial variable.  The PDE problem is stated as: find $u(\mb{x}) \in F: \xs{R}^d \mapsto \xs{R}$ where $F$ is a proper function space defined on $\Omega$, such that
\begin{equation} \label{eq_pde}
	\begin{aligned}
		\mathcal{L} u(\mb{x}) &= s(\mb{x}),  \quad \forall \mb{x} \in \Omega,\\
		\mathfrak{b} u(\mb{x}) &= g(\mb{x}), \quad \forall \mb{x} \in \partial \Omega,
	\end{aligned}
\end{equation}
where $\mathcal{L}$ is the partial differential operator, $\mathfrak{b}$ is the boundary operator, $s(\mb{x})$ is the source function, and $g(\mb{x})$ represents the boundary conditions.

Let $u(\mb{x};\Theta)$ be a neural network with parameters $\Theta$. In the framework of PINNs, the goal is to use $u(\mb{x};\Theta)$ to approximate the solution $u(\mb{x})$ through optimizing a loss functional defined as \cite{sirignano2018dgm,raissi2019physics}
\begin{equation} \label{eq_resloss}
	J \left( u(\mb{x};\Theta) \right) = \norm{r(\mb{x};\Theta)}{2,\Omega}^2 + \gamma \norm{b(\mb{x};\Theta)}{2, \partial \Omega}^2=J_r(u(\mb{x};\Theta)) + \gamma J_b(u(\mb{x};\Theta)),
\end{equation} 
where $r(\mb{x};\Theta) = \mathcal{L} u(\mb{x};\Theta) - s(\mb{x})$, and $b(\mb{x};\Theta) = \mathfrak{b} u(\mb{x};\Theta) - g(\mb{x})$ measure how well $u(\mb{x};\Theta)$ satisfies the partial differential equations and the boundary conditions, respectively, and $\gamma>0$ is a penalty parameter. Here, $\norm{u(\mb{x})}{2,\Omega}^2 = \int_{\Omega} |u(\mb{x})|^2 d\mb{x}$ and $\norm{u(\mb{x})}{2,\partial \Omega}^2 =\int_{\partial \Omega} |u(\mb{x})|^2 d\mb{x}$. The loss functional \eqref{eq_resloss} is usually discretized numerically before the optimization with respect to $\Theta$ is addressed. In practice, one often chooses two sets of uniformly distributed collocation points $\mathsf{S}_{\Omega} = \{\mb{x}_{\Omega}^{(i)} \}_{i=1}^{N_r}$ and $\mathsf{S}_{\partial \Omega} = \{\mb{x}_{\partial \Omega}^{(i)} \}_{i=1}^{N_b}$ respectively for the discretization of the two terms in the objective functional \eqref{eq_resloss},
leading to the following empirical loss
\begin{equation} \label{eq_discrete_loss}
	J_N \left( u(\mb{x};\Theta) \right) =\norm{r(\mb{x};\Theta)}{N_r, \mathsf{S}_{\Omega}}^2 + \hat{\gamma}\norm{b(\mb{x};\Theta)}{N_b, \mathsf{S}_{\partial \Omega}}^2,
\end{equation}
where $\hat{\gamma}>0$, and
$$\norm{u(\mb{x})}{N_r, \mathsf{S}_{\Omega}} = \left( \frac{1}{N_r} \sum_{i=1}^{N_r} u^2(\mb{x}_{\Omega}^{(i)}) \right)^{\frac{1}{2}},\quad \norm{u(\mb{x})}{N_b, \mathsf{S}_{\partial \Omega}} =  \left( \frac{1}{N_b} \sum_{i=1}^{N_b} u^2(\mb{x}_{\partial \Omega}^{(i)})  \right)^{\frac{1}{2}}.$$
Note that in the definition of $J_N$ we do not take into account the constants $|\Omega|=\int_{\Omega}d\bx$ and $|\partial\Omega|=\int_{\partial\Omega}d\bx$ and the ratio induced by these two constants can be dealt with by choosing $\hat{\gamma}=\frac{\gamma |\partial\Omega|}{|\Omega|}$ such that $J_N(u)$ is a Monte Carlo approximation of $J(u)$ up to a constant scaling factor $|\Omega|$.
We then seek an approximate solution by minimizing the empirical loss \eqref{eq_discrete_loss}, i.e.,
\begin{equation} \label{eq_opt_dlpde}
	\min_{\Theta}J_N(u(\mb{x};\Theta)),
\end{equation}
which can be solved by stochastic gradient-based methods \cite{bottou2018optimization, kingma2017adam}.

Recently, some prior error estimates of neural-network-based methods for solving PDEs are established. Combining the analysis techniques of the least square finite element method \cite{bochev2016least} with the universal approximation property of neural networks \cite{cybenko1989approximation,hornik1989multilayer,hornik1991approximation,leshno1993multilayer}, Shin et. al. propose an abstract framework for the error estimation of physical informed neural networks \cite{shin2020error}. Lu et. al. derive a prior estimate of the generalization error for the deep Ritz method with two-layer neural networks \cite{lu2021priori}. Suppose that $u(\mb{x},\Theta_N^*)$ is the minimizer of the empirical loss $J_N(u(\mb{x}; \Theta))$ and $u(\mb{x};\Theta^*)$ is the minimizer of $J(u(\mb{x};\Theta))$, i.e.,
\begin{equation*}
	\begin{aligned}
		u(\mb{x};\Theta^*) &= \arg \min_{\Theta} J(u(\mb{x};\Theta)), \\
		u(\mb{x};\Theta_N^*) &= \arg \min_{\Theta} J_N(u(\mb{x};\Theta)).
	\end{aligned}
\end{equation*}
We have
\begin{equation}
	u(\mb{x};\Theta_N^*) - u(\mb{x}) = u(\mb{x},\Theta_N^*) - u(\mb{x};\Theta^*) + u(\mb{x};\Theta^*) - u(\mb{x}),
\end{equation}
i.e.,
\begin{equation}
	\mathbb{E} \left( \norm{u(\mb{x};\Theta_N^*) - u(\mb{x})}{\Omega} \right) \leq \mathbb{E} \left( \norm{u(\mb{x},\Theta_N^*) - u(\mb{x};\Theta^*)}{\Omega} \right) + \norm{u(\mb{x};\Theta^*) - u(\mb{x})}{\Omega},
\end{equation}
where $\mathbb{E}$ indicates the expectation and the norm $\norm{\cdot}{\Omega}$ corresponds to the function space $F$ for $u(\mb{x};\Theta)$. The first term describes the statistical error from discretizing the loss functional with the Monte Carlo approximation, and the second term is the approximation error of minimizing the loss functional over the  hypothesis space. The approximation error depends on the capability of neural networks, while the statistical error depends on the definition of $\mathsf{S}_{\Omega}$ and $\mathsf{S}_{\partial \Omega}$. In this work, we focus on how to reduce the statistical error for problem \eqref{eq_opt_dlpde} and our algorithm can also be generalized to other formulations for the neural network approximation of PDEs. For simplicity, we focus on the integration of the residual $r(\mb{x};\Theta)$ and assume that the integral on the boundary is well approximated by a prescribed $\mathsf{S}_{\partial\Omega}$.

\section{Illustration of the statistical error}

We first use function approximation as an example to illustrate the statistical error of the machine learning technique.  Let $\mb{X}\in\mathbb{R}^d$ and $Y\in\mathbb{R}$ subject to a joint distribution $\rho_{\mb{X},\mb{Y}}$. Let $\hat{Y}=m(\mb{X})$ be a model and $y = h(\mb{x})$ be a function to be approximated. We know in the $L_2$ sense the optimal model is
\begin{equation}\label{eqn:fa_0}
	{m}^*(\mb{x}) = \arg \min_{{m}(\mb{x})}\left[L(Y,\hat{Y})=\int (y-{m}(\mb{x}))^2\rho_{\mb{X},Y}(\mb{x},y)d\mb{x}d y\right].    
\end{equation}
In reality, we usually do not know $\rho_{\mb{X},Y}$ and only have a set $\{(\mb{x}^{(i)},y^{(i)})\}_{i=1}^N$ of data  which can be regarded as samples of $\rho_{\mb{X},Y}$. For a certain hypothesis space $W$, we obtain a regression problem
\begin{equation}\label{eqn:fa_1}
	{m}_{\mb{w}^*}(\mb{x})=\arg \min_{{m}_{\mb{w}}\in W}\left[L_N(Y,\hat{Y})=\frac{1}{N}\sum_{i=1}^N(y^{(i)}-{m}_{\mb{w}}(\mb{x}^{(i)}))^2\right],
\end{equation}
where $L_N$ can be regarded as a Monte Carlo approximation of $L$ and the subscript ${\mb{w}}$ indicates the model parameters specified by $W$. If we let $\rho_{\mb{X},Y}(\mb{x},y)=\delta(y-h(\mb{x}))\rho(\mb{x})$, and assume $m(\bx)\in V$ with $V$ being a linear space, we then obtain the continuous least-squares method for function approximation
\begin{equation}\label{eqn:fa_2}
	{m}_V^*(\mb{x})=\arg \min_{{m}(\mb{x})\in V}\left[L_V(Y,\hat{Y})=\int({m}(\mb{x})-h(\mb{x}))^2\rho(\mb{x})d\mb{x}\right],
\end{equation}
where $m_V^*(\mb{x})$ is the best approximation of $h(\mb{x})$ located in $V$ subject to a weighted $L_2$ norm in terms of a probability density function $\rho(\mb{x})$. To approximate $h(\mb{x})$ with a machine learning technique, we consider
\begin{equation}\label{eqn:fa_3}
	m_{\hat{\mb{v}}^*}(\mb{x})=\arg \min_{m_{\hat{\mb{v}}} \in V}\left[L_{V,N}(Y,\hat{Y})=\frac{1}{N}\sum_{i=1}^N(m_{\hat{\mb{v}}} (\mb{x}^{(i)})-h(\mb{x}^{(i)}))^2\right],
\end{equation}
where $\{\mb{x}^{(i)}\}_{i=1}^N$ are samples of PDF $\rho(\mb{x})$, and $L_{V,N}$ is the Monte Carlo approximation of $L_V$. A classical choice for $V$ is the linear space spanned by polynomials, for which we derive the error estimate for $m_{\hat{\mb{v}}^*}(\mb{x})$ as follows.
\begin{lem}\label{lem:fn_stat_err}
	Let $h(\mb{x})\in C(D)$ be a continuous function defined on a compact domain $D\subset\mathbb{R}^d$ and $\rho(\mb{x})>0$ be a PDF on $D$. Let $V=\mathrm{span}\{q_i(\mb{x})\}_{i=1}^n$ with $q_i(\mb{x})$ being orthonormal polynomials in terms of $\rho(\mb{x})$. For any $\delta>0$ and with probability at least $1-2\delta$, we have for a sufficiently large $N$
	\[
	\|m_{\hat{\mb{v}}^*}(\mb{x})-h(\mb{x})\|_\rho\leq C\sqrt{\frac{\ln\delta^{-1}}{N}}+\|m^*_V(\mb{x})-h(\mb{x})\|_\rho,
	\]
	where $C$ is a constant, and $\|\cdot\|_{\rho}$ is the weighted $L_2$ norm in terms of $\rho(\mb{x})$. 
\end{lem}

The first term on the right-hand is the statistical error due to the random samples for the approximation of $L_{V}$ (see the proof in the appendix for more details) and its existence does not depend on the choice of $V$. When $N$ goes to infinity, the statistical error goes to zero and only the approximation error is left. In other words, when applying a machine learning technique to function approximation, we need to pay attention to both the hypothesis space $W$ and the choice of random samples $\{\mb{x}^{(i)}\}_{i=1}^N$, i.e., the training set, to obtain a trade-off between the statistical error and the approximation error. For low-dimensional problems, classical method such as finite element methods avoid the statistical error by using Gauss quadrature rules, which implies that machine learning techniques are in general less efficient than classical methods due to the existence of statistical error. On the other hand, for high-dimensional problems, classical methods may not be able to obtain a relatively small approximation error due to the curse of dimensionality and machine learning techniques may perform better by using a capable hypothesis space such as neural networks and an affordable sample size for a relatively small statistical error.  
A more general estimate about the statistical error needs the Rademacher complexity of the hypothesis space. In this work, we are interested in the reduction of the statistical error instead of its bound for a certain set of random samples.

\section{Deep adaptive sampling method}\label{section_DAS}

We now focus on the reduction of the statistical error when machine learning techniques are used to approximate PDEs. Our deep adaptive sampling method, i.e., the DAS method, will be established from the viewpoint of variance reduction since the statistical error is induced by the Monte Carlo approximation of the loss. Consider the term $J_r(u)$ in equation \eqref{eq_resloss}. It is easy to see that if $r^2(\mb{x})$ is a smooth function with a good regularity, the most effective way to reduce the error of $J_N(u)$ is to increase the sample size $N$. However, if  there exists low regularity, the situation might be different. For example, if the residual is strongly localized, the scenario can be regarded as a rare event. Assume that the residual $r^2(\mb{x})$ in equation \eqref{eq_resloss} has a similar behavior to an indicator function $1_I(\mb{x})$ where $I \subset \Omega$ and is much smaller than $\Omega$, i.e.,
\[
\zeta = \int_{\Omega}1_I(\mb{x})d\mb{x}\approx\int_{\Omega}r^2(\mb{x})d\mb{x}\ll 1,
\]
For simplicity, we assume that $|\Omega|=1$ when presenting the algorithm. To improve the approximation in $I$, we need to compute $\zeta$ accurately. Consider a Monte Carlo estimator of $\zeta$ in terms of uniform samples
\[
\hat{P}_{\mathsf{MC}}=\frac{1}{N}\sum_{i=1}^N 1_I(\mb{x}^{(i)}).
\]
The relative error of $\hat{P}_{\mathsf{MC}}$ is 
\[
\frac{\mathrm{Var}^{1/2}(\hat{P}_{\mathsf{MC}})}{\zeta}=N^{-1/2}((1-\zeta)/\zeta)^{1/2}\approx(\zeta N)^{-1/2}.
\]
We then need a sample size of $\mathit{O}(1/\zeta)$ to obtain a relative error of $\mathit{O}(1)$. This implies that a certain samples size $N$ quickly becomes less effective if the residual is strongly localized and such a problem can be worsen in the approximation of high-dimensional PDEs. To reduce the relative error, we need to choose more effective random samples instead of uniform samples. 

\subsection{Some ideas on variance reduction}\label{sec:is_idea}
We outline our basic ideas on variance reduction in this section and more details about the algorithm will be presented later. We first consider the importance sampling technique. For simplicity, we only consider $J_r(u(\mb{x};\Theta))$ in equation \eqref{eq_resloss}, which is the loss induced by the residual. We have
\begin{equation} \label{eq_imploss}
	\begin{aligned}
		J_r \left( u(\mb{x};\Theta) \right)=\mathbb{E}[r^2]
		= \int_{\Omega}  r^2(\mb{x};\Theta) d\mb{x}
		= \int_{\Omega}  \frac{r^2(\mb{x};\Theta)}{p(\mb{x})} p(\mb{x})d\mb{x} 
		\approx \frac{1}{N_r} \sum\limits_{i=1}^{N_r} 	\frac{r^2(\mb{x}_{\Omega}^{(i)};\Theta)}{p(\mb{x}_{\Omega}^{(i)})},
	\end{aligned}
\end{equation} 
where the set $\{\mb{x}_{\Omega}^{(i)}\}_{i=1}^{N_r}$ is generated with respect to the probability density function (PDF) $p(\mb{x})$ instead of a uniform distribution as in equation \eqref{eq_discrete_loss}. If the variance of $r^2(\mb{X})p^{-1}(\mb{X})$ in terms of $p(\mb{x})$ is smaller than the variance of $r^2(\mb{X})$ in terms of the uniform distribution, the accuracy of the Monte Carlo approximation will be improved for a fixed sample size $N_r$. The optimal choice for $p(\mb{x})$ is
\begin{equation}
	p^*(\mb{x}) = \frac{r^2(\mb{x};\Theta)}{\mu} ,
\end{equation}
where $\mu=\int_{\Omega}r^2(\mb{x};\Theta)d\mb{x}$. Although the optimal choice is useless in practice since $\mu$ is the quantity to be computed, it suggests that variance reduction can be achieved if $p(\bx)$ is close to the residual-induced distribution. 
\begin{lem}\label{lem:kl_is}
	Assume that $|\Omega|=1$ and $p(\bx)$ is a PDF satisfying
	\[
	D_{\mathsf{KL}}(p\|p^*)\leq \varepsilon<\infty,
	\]
	where $D_{\mathsf{KL}}$ indicates the Kullback-Leibler divergence. For any $0<a<\infty$, we have
	\begin{equation}\label{eqn:int_err_KL}
		\mathbb{E}\left|Q_p[r^2]-\mathbb{E}[r^2]\right|\leq aN_r^{-1/2}+2\|r^2/p\|_{p}\sqrt{\mathbb{P}(|r^2/p-\mu|>a;p)},
	\end{equation}
	where
	\[
	Q_p(r^2)=\frac{1}{N_r}\sum_{i=1}^{N_r}\frac{r^2(\mb{X}^{(i)})}{p(\mb{X}^{(i)})},
	\]
	and $\mb{X}^{(i)}\sim p(\mb{x})$ are i.i.d. random variables. $\|\cdot\|_{p}$ is the weighted $L_2$ norm in terms of $p$. The tail probability can be bounded as 
	\[
	\mathbb{P}(|r^2/p-\mu|>a;p)\leq \frac{\mu(2\varepsilon)^{1/2}}{a}.
	\]
\end{lem}

When $p(\mb{x})$ is close to $p^*(\mb{x})$, the ratio $r^2(\mb{x})/p(\mb{x})$ should be close to $\mu$. The second term on the right-hand side of inequality \eqref{eqn:int_err_KL} is related to the tail probability $\mathbb{P}(|r^2/p-\mu|>a;p)$ in terms of $p(\mb{x})$. It is seen that if the tail probability is small enough, a variance reduction can be achieved by choosing a small $a$. 	

Another option for variance reduction is to relax the definition of $J_r(u)$ as:
\begin{equation}\label{eqn:J_r_p}
	J_{r,p}(u(\mb{x};\Theta))=\int_{\Omega}r^2(\mb{x};\Theta)p(\mb{x})d\mb{x}\approx\frac{1}{N_r}\sum_{i=1}^{N_r} r^2(\mb{x}_{\Omega}^{(i)};\Theta),
\end{equation}
where the set $\{\mb{x}_{\Omega}^{(i)}\}_{i=1}^{N_r}$ is sampled from the PDF $p(\mb{x})$ and $p(\mb{x})>0$ on $\Omega$.  It is easy to see that the minimizer of $J_{r,p}(u(\mb{x};\Theta))$ is also the solution of problem \eqref{eq_pde} when the boundary conditions are satisfied. To reduce the error induced by the Monte Carlo approximation, we may adjust $p(\mb{x})$ such that the residual $r^2(\mb{x};\Theta)$ does not vary dramatically. This is similar to the classical adaptive finite element method, where mesh refinement/coarsening is supposed to make the approximation error nearly uniform. For our case, we may sample a PDF that is close to the residual-induced  distribution and add more samples from the region of large residual into the training set. 

\subsection{PDF approximation and sample generation}\label{sec:sample_generation}
Two options for variance reduction are considered in the previous section. To make both ideas practical, the key issue is to generate samples efficiently from a PDF $p(\bx)\approx\mu^{-1}r^2(\bx;\Theta)$ for a fixed $\Theta$. The first option is based on importance sampling, meaning that an explicit PDF model $p(\mb{x})$ is needed, The second option is based on refinement of the training set, which only needs samples from the region of large residual and does not need the likelihood of the samples. The second option has been employed in the recent literature to improve the neural network approximation, which is either ad hoc \cite{yu2022gradient} (only for low-dimensional problems) or based on traditional sampling strategies such as MCMC \cite{gao2021active}. To the best of our knowledge, there does not exist a generic algorithm that can be effectively adapted to both options. We intend to fill this gap in this work. 
It is seen from Lemma \ref{lem:kl_is} that the tail probability $\mathbb{P}(|r^2/p-\mu|>a;p)$ should be small enough for the effectiveness of $p(\mb{x})$, which means that $p(\mb{x})$ must be close enough to the distribution induced by $r^2(\mb{x})$. However, the approximation of PDF is a challenging task especially in high-dimensional spaces. Classical explicit PDF models such as the exponential family of distributions and Gaussian mixture models are in general not sufficient as an approximator for the PDF induced by $r^2(\mb{x})$. To alleviate this difficulty, we resort to deep generative modes. In particular, we employ a recently developed deep generative model called KRnet for both probability approximation and sample generation \cite{tang2021adaptive,tangwandensity2020}. KRnet is one type of normalizing flow \cite{kobyzev2020normalizing}, which provides an invertible transport map between a prior distribution and the target distribution. Unlike other types of deep generative models such as GAN \cite{goodfellow2014generative, arjovsky2017wasserstein, gulrajani2017improved} and VAE \cite{kingma2014auto}, normalizing flow provides an explicit likelihood. KRnet can be used to address both options for variance reduction while other deep generative models may also be employed if only the second option for variance reduction is considered. However, one computational issue faced by all deep generative models is that the distribution induced by $r^2(\mb{x})$ is usually defined on a compact domain while deep generative models are in general defined on the whole space. We subsequently address this issue without going into details about the structure of KRnet. 

Let $\mb{X} \in \xs{R}^d$ be a random vector associated with a given data set, and its PDF 
is denoted by $p_{\mb{X}}(\mb{x})$. The target is to estimate $p_{\mb{X}}(\bx)$ from data or to generate samples that are consistent with a given $p_{\mb{X}}(\mb{x})$. Let $\mb{Z} \in \xs{R}^{d}$ be a random vector associated with a PDF $p_{\mb{Z}}(\mb{z})$, where $p_{\mb{Z}}(\mb{z})$ is a prior distribution (e.g., Gaussian distributions). The \comm{goal of} flow-based generative modeling is to seek an invertible mapping $\mb{z} = f(\mb{x})$ \comm{where $f(\cdot)$ is a bijection: $f:\mb{x} \mapsto \mb{z}$} \cite{dinh2016density}. By the change of variables, we have the PDF of $\mb{X}=f^{-1}(\mb{Z})$ as 
\begin{equation}\label{eqn:pdf_model}
	p_{\mb{X}}(\mb{x})=p_{\mb{Z}}(f(\mb{x})) \left |\det\nabla_{\mb{x}} f \right|.
\end{equation}
Once the prior distribution $p_{\mb{Z}}(\mb{z})$ is specified, equation \eqref{eqn:pdf_model} provides an explicit density model for $\mb{X}$. The inverse of $f(\cdot)$ provides a convenient way to sample $\mb{X}$ as $\mb{X}=f^{-1}(\mb{Z})$. 
The basic idea of KRnet is to define the structure of $f(\mb{x})$ in terms of the Knothe-Rosenblatt rearrangement. 
Let $\mu_{\mb{Z}}$ and $\mu_{\mb{X}}$ be the probability measures of two random variables $\mb{Z},\mb{X}\in\xs{R}^d$ respectively. A mapping $\mathcal{T}$: $\mb{Z} \mapsto \mb{X}$ is called a transport map such that $\mathcal{T}_{\#} \mu_{\mb{Z}} = \mu_{\mb{X}}$, where $\mathcal{T}_{\#} \mu_{\mb{Z}}$ is the push-forward of $\mu_{\mb{Z}}$ such that $\mu_{\mb{X}}(B) = \mu_{\mb{Z}}(\mathcal{T}^{-1}(B))$ for every Borel set $B$ \cite{carlier2010knothe}. The transport map $\mathcal{T}$ given by the Knothe-Rosenblatt rearrangement \cite{carlier2010knothe, santambrogio2015optimal} has a lower-triangular structure 
\begin{equation}
	\mb{z} = \mathcal{T}^{-1}(\mb{x}) = \left[ 
	\begin{array}{l}
		\mathcal{T}_1(x_1) \\
		\mathcal{T}_2(x_1, x_2) \\
		\vdots \\
		\mathcal{T}_{d}(x_1, \ldots, x_d)
	\end{array}
	\right].
\end{equation}
Simply speaking, KRnet integrates the triangular structure of the K-R rearrangement into the definition of the invertible transport map. More details can be found in \cite{tangwandensity2020,tang2021adaptive,wan2021augmented}. 

Let $f_{\mathsf{KRnet}}(\cdot;\Theta_f)$ indicate the invertible transport map induced by KRnet, where $\Theta_f$ includes the model parameters. An explicit PDF model $p_{\mathsf{KRnet}}(\bx;\Theta_f)$ can be obtained by letting $f=f_{\mathsf{KRnet}}$ in equation \eqref{eqn:pdf_model}, i.e.,
	\begin{equation}\label{eq_krpdf}
	p_{\mathsf{KRnet}}(\mb{x};\Theta_f)=p_{\mb{Z}}(f_{\mathsf{KRnet}}(\mb{x})) \left |\det\nabla_{\mb{x}} f_{\mathsf{KRnet}} \right|.
\end{equation}
The samples of $p_{\mathsf{KRnet}}(\bx;\Theta_f)$ is given as $\mb{X}=f^{-1}_{\mathsf{KRnet}}(\mb{Z})$ by sampling $\mb{Z}$. A common choice for the distribution of $\mb{Z}$ is the standard Gaussian distribution. Depending on the prior knowledge of the problem, a more general model such as Gaussian mixture model can also be used as the prior distribution. The KRnet $f_{\mathsf{KRnet}}$ does not have any constraint on the range of the mapped data, meaning that both $\mb{X}$ and $\mb{Z}$ are defined on $\mathbb{R}^d$. Let $\mb{Z}$ be Gaussian. Due to the invertibility, $|\det\nabla_{\mb{x}}f_{\mathsf{KRnet}}|>0$ for any $\mb{x}\in\mathbb{R}^d$, which implies that $p_{\mathsf{KRnet}}(\mb{x};\Theta_f)>0$ for any $\mb{x}\in\mathbb{R}^d$. So $p_{\mathsf{KRnet}}(\mb{x};\Theta_f)$ is not consistent with the distribution induced by $r^2(\mb{x})$, which is equal to zero on $\mathbb{R}^d\backslash \Omega$. To deal with this issue, we propose the following strategy. 

Without loss of generality, we can assume that $\Omega=[-1/2,1/2]^d$. Let $B=(-(1/2+\delta/2),1/2+\delta/2)^d$ with  $0<\delta<\infty$ such that $\Omega \subset B$. For each dimension of $\bx$, we define the following logarithmic mapping
\begin{equation*}
	y=\ell(x)=\frac{s}{2}\log\frac{2x+(1+\delta)}{(1+\delta)-2x},\quad x=\ell^{-1}(y)=\frac{1+\delta}{2}\frac{e^{2y/s}-1}{e^{2y/s}+1},
\end{equation*} 
with $s>0$ being a scale parameter, which defines a one-to-one correspondence between $x \in (-(1/2+\delta/2),1/2+\delta/2)$ and $y \in (-\infty,+\infty)$. Let $\boldsymbol{\ell}(\bx):B \mapsto \mathbb{R}^d$ be a $d$-dimensional mapping such that 
\[
\ell_i(x_i)=\ell(x_i),\quad i=1,\ldots,d.
\]
Then the following invertible mapping
\begin{equation}\label{eqn_boundKR}
	\mb{z} = f_{\mathsf{KRnet}} \circ\boldsymbol{\ell}(\bx)
\end{equation}
defines a PDF
\begin{equation} \label{eqn_KRres}
	\hat{p}_{\mathsf{KRnet}}(\bx;\Theta_f) =p_{\mathsf{KRnet}}(\boldsymbol{\ell}(\bx);\Theta_f)|\nabla_{\bx}\boldsymbol{\ell}(\bx)|,
\end{equation}
where the support of $\hat{p}_{\mathsf{KRnet}}(\bx;\Theta_f)$ is $B$.

We now consider a modification of $r^2(\bx;\Theta)$. Define a cutoff function as
\begin{equation*}
	h(\bx)=\left\{
	\begin{array}{rl}
		1,&\bx \in \Omega,\\
		\prod_{i=1}^dh_\delta(x_i),&\bx \in B\backslash \Omega,
	\end{array}
	\right.
\end{equation*}
where $h_{\delta}(x)$ is a piecewise linear function 
\begin{equation*}
	h_{\delta}(x)=\left\{
	\begin{array}{rl}
		1,&x\in[-1/2,1/2],\\
		\delta^{-1}(2x+1+\delta),&x\in(-1/2-\delta/2,-1/2),\\
		\delta^{-1}(1+\delta-2x),&x\in(1/2,1/2+\delta/2),\\
		0,&x\in (-\infty,-1/2-\delta/2]\cup[1/2+\delta/2,\infty).
	\end{array}
	\right.
\end{equation*} 
We consider a modified PDF for any $\Theta$
\begin{equation}
	\hat{r}_{\bX}(\bx) \propto r^2(\bx;\Theta)h(\bx).
\end{equation}
Note that both $\hat{r}_{\bX}(\bx)$ and $\hat{p}_{\mathsf{KRnet}}(\bx;\Theta_f)$ have the support $B$. We then solve the following optimization problem
\begin{equation}\label{eqn:KL_opt}
	\Theta_f^*=\arg \min_{\Theta_f} D_{\mathsf{KL}}(\hat{r}_{\bX}(\bx)\|\hat{p}_{\mathsf{KRnet}}(\bx;\Theta_f)),
\end{equation}
where $D_{\mathsf{KL}}(\cdot\|\cdot)$ indicates the Kullback-Leibler (KL) divergence between two distributions. We finally use 
\begin{equation}
	p_{\bX}(\bx)\propto\hat{p}_{\mathsf{KRnet}}(\bx;\Theta_f^*)1_{\Omega}(\bx)
\end{equation}
as an approximation of the PDF induced by $r^2(\bx;\Theta)$. If $\delta\ll 1$, most the samples $\boldsymbol{\ell}^{-1}\circ f^{-1}_{\mathsf{KRnet}}(\mb{z}^{(i)};\Theta_f^*)$ will be located in $\Omega$. Since $\hat{r}_{\bX}(\bx)\propto r^2(\bx;\Theta)$ on $\Omega$, $p_{\bX}(\bx)$ approximates the $r^2(\bx;\Theta)$-induced PDF well when $\delta$ is small. In our numerical experiments, we set $\delta = 0.01$ and $s = 2$.

We now look at the approximation of $\Theta_f^*$. The KL divergence in the optimization problem \eqref{eqn:KL_opt} can be written as
\begin{equation}\label{eqn:KL_r2p}
	D_{\mathsf{KL}}(\hat{r}_{\bX}(\bx)\|\hat{p}_{\mathsf{KRnet}}(\bx;\Theta_f))=
	\int_{B} \hat{r}_{\bX} \log \hat{r}_{\bX}d\mb{x} - \int_{B} \hat{r}_{\bX} \log \hat{p}_{\mathsf{KRnet}} d\mb{x}.
\end{equation}
The first term on the right-hand side corresponds to the differential entropy of $\hat{r}_{\bX}$, which does not affect the optimization with respect to $\Theta_f$. So minimizing the KL divergence is equivalent to minimizing the cross entropy between $\hat{r}_{\bX}$ and $\hat{p}_{\mathsf{KRnet}}$ \cite{de2005tutorial, rubinstein2013cross}:
\begin{equation}
	H(\hat{r}_{\bX},\hat{p}_{\mathsf{KRnet}})=- \int_{B} \hat{r}_{\bX} \log \hat{p}_{\mathsf{KRnet}} d\mb{x}.
\end{equation}
Since the samples from $\hat{r}_{\bX}$ are not available, we approximate the cross entropy using the importance sampling technique:
\begin{equation}\label{eqn:ce_approx}
	H(\hat{r}_{\bX},\hat{p}_{\mathsf{KRnet}}) \approx -\frac{1}{N_r}\sum_{i=1}^{N_r} \frac{\hat{r}_{\bX}(\mb{x}_{ B}^{(i)})}{\hat{p}_{\mathsf{KRnet}}(\mb{x}_{B}^{(i)};\hat{\Theta}_f)}\log \hat{p}_{\mathsf{KRnet}}(\mb{x}_{B}^{(i)};\Theta_f),
\end{equation}
where $\hat{p}(\mb{x};\hat{\Theta}_f)$ is a PDF model with known parameter $\hat{\Theta}_f$ and its samples $\{\mb{x}_{B}^{(i)}\}_{i=1}^{N_r}$ can be generated efficiently as 
\begin{equation}\label{eqn:z_to_x}
	\mb{x}_B^{(i)}=\boldsymbol{\ell}^{-1} \circ f^{-1}_{\mathsf{KRnet}}(\mb{z}^{(i)};\hat{\Theta}_f)
\end{equation}
with $\mb{z}^{(i)}$ being sampled from the prior distribution. We then minimize the discretized cross entropy \eqref{eqn:ce_approx} to obtain an approximation of $\Theta_f^*$. The choice for $\hat{\Theta}_f$ will be specified in section \ref{sec_das} when the adaptive sampling method is defined.

\begin{remark}\label{rmk:kl}
	An alternative approach for the approximation of $\hat{r}_{\bX}(\bx)$ is to minimize the following KL divergence:
	\begin{equation}\label{eqn:KL_p2r}
		D_{\mathsf{KL}}(\hat{p}_{\mathsf{KRnet}}(\bx;\Theta_f)\|\hat{r}_{\bX}(\bx))=\int_{B} \hat{p}_{\mathsf{KRnet}} \log \hat{p}_{\mathsf{KRnet}}d\mb{x} - \int_{B}  \hat{p}_{\mathsf{KRnet}}\log \hat{r}_{\bX} d\mb{x},
	\end{equation}
	which can be approximated by samples from $\hat{p}_{\mathsf{KRnet}}(\bx;\Theta_f)$. Note that the KL divergence is asymmetric. Minimizing the KL divergence \eqref{eqn:KL_r2p} is not equivalent to the minimization of the KL divergence \eqref{eqn:KL_p2r}, although both minimizers will be achieved at $\hat{p}_{\mathsf{KRnet}}(\bx;\Theta_f)=\hat{r}_{\bX}(\bx)$ if $\hat{r}_{\bX}$ can be reached exactly by a certain parameter $\Theta_f$.
\end{remark}

\subsection{Adaptive sampling procedure}\label{sec_das}
We are now ready to present our algorithms. In this work, we mainly focus on the adaptivity of $\mathsf{S}_{\Omega}$ for simplicity. The key step of our adaptivity strategy is to improve the effectiveness of the random samples in the training set $\mathsf{S}_{\Omega}$, and we provide two algorithms corresponding to the two options discussed in section \ref{sec:is_idea}.
\begin{enumerate}
	\item[I.] We replace all the collocation points in the current training set using new samples from the probability measure for importance sampling. This corresponds to equation \eqref{eq_imploss}.
	\item[II.] We gradually add more collocation points to the current training set. This corresponds to equation \eqref{eqn:J_r_p}, where the new samples are mainly from the region of large residual.
\end{enumerate}

We first present strategy I.  Let $\mathsf{S}_{\Omega, 0} = \{\mb{x}_{\Omega, 0}^{(i)} \}_{i=1}^{N_r}$ and $\mathsf{S}_{\partial\Omega,0}$ be two sets of collocation points that are uniformly sampled from $\Omega$ and $\partial \Omega$ respectively. Using $\mathsf{S}_{\Omega,0}$ and $\mathsf{S}_{\partial\Omega,0}$, we minimize the empirical loss \eqref{eq_discrete_loss} to obtain $u(\mb{x};\Theta_N^{*,(1)})$. With $u(\mb{x};\Theta_N^{*,(1)})$, we minimize the cross entropy \eqref{eqn:ce_approx} to get $\hat{p}_{\mathsf{KRnet}}(\bx;\Theta_f^{*, (1)})$, where we simply use uniform samples for importance sampling. To refine the training set, a new set $\mathsf{S}_{\Omega, 1} = \{\mb{x}_{\Omega, 1}^{(i)} \}_{i=1}^{N_r}$ is generated by $\boldsymbol{\ell}^{-1} \circ f^{-1}_{\mathsf{KRnet}}(\mb{z}^{(i)};\Theta_f^{*,(1)})$ (see equation \eqref{eqn_boundKR}). Then we continue to update the approximate solution $u(\mb{x};\Theta_N^{*,(1)})$ using $\mathsf{S}_{\Omega, 1}$ as the training set. In general, we use $\mathsf{S}_{\Omega, k} = \{\mb{x}_{\Omega, k}^{(i)} \}_{i=1}^{N_r}$ to obtain $u(\mb{x};\Theta_N^{*,(k+1)})$ as
\begin{equation*}
	\Theta_N^{*,(k+1)}= \arg \min_{\Theta} J^{\mathsf{IS}}_{N}(u(\mb{x};\Theta)),
\end{equation*}
where $u(\bx;\Theta)$ is initialized as $u(\bx;\Theta_N^{*,(k)})$ and $J^{\mathsf{IS}}_N$ is defined as
\begin{equation} \label{eqn_ISdiscreteloss}
	J^{\mathsf{IS}}_N(u(\mb{x};\Theta))=\frac{1}{N_r} \sum\limits_{i=1}^{N_r} \frac{r^2(\mb{x}_{\Omega,k}^{(i)};\Theta)}{\hat{p}_{\mathsf{KRnet}}(\mb{x}_{\Omega, k}^{(i)};\Theta_f^{*,(k)})} + \frac{1}{N_b} \sum\limits_{i = 1}^{N_b} b^2(\mb{x}_{\partial \Omega,k}^{(i)};\Theta).  
\end{equation}
Starting with $\hat{p}_{\mathsf{KRnet}}(\bx;\Theta_f^{*,(k)})$, the density model $\hat{p}_{\mathsf{KRnet}}(\bx;\Theta_f)$ is updated as 
\begin{equation}\label{eqn_ISminKR}
	\Theta_f^{*,(k+1)} = \arg \min_{\Theta_f} -\frac{1}{N_r}\sum_{i=1}^{N_r} \frac{r^2(\mb{x}_{B, k}^{(i)};\Theta_N^{*,(k+1)})h(\mb{x}_{B, k}^{(i)})}{\hat{p}_{\mathsf{KRnet}}(\mb{x}_{B, k}^{(i)};\Theta_f^{*,(k)})} \log \hat{p}_{\mathsf{KRnet}}(\mb{x}_{B, k}^{(i)};\Theta_f),
\end{equation}
where we let $\hat{\Theta}_f=\Theta_f^{*,(k)}$ in equation \eqref{eqn:ce_approx}, i.e., the previous PDF model $\hat{p}_{\mathsf{KRnet}}(\bx;\Theta_f^{*,(k)})$ is used for importance sampling when computing the cross entropy. A new set $\mathsf{S}_{\Omega, k+1} = \{\mb{x}_{\Omega, k+1}^{(i)} \}_{i=1}^{N_r}$ of collocation points is then generated. As detailed in section \ref{sec:sample_generation}, the support of data points generated by KRnet is $B = (-(1/2+\delta/2), 1/2+\delta/2)^d$, while the computation domain is $\Omega = [-1/2,1/2]^d$. So we need to deal with the collocation points located in $B\backslash\Omega$. Instead of neglecting these points, we project them onto $\partial\Omega$. We define an entry-wise projection operator $\mathcal{P}(\mb{x}): B \mapsto \Omega$ as
\begin{equation} \label{eqn_proj}
	\mathcal{P}(x_i) = \begin{cases} -1/2, \quad  \mathrm{if} \  x_i < -1/2 ,\\
		x_i,  \quad \mathrm{if}  \  -1/2 \leq x_i \leq 1/2 , \\
		1/2, \quad  \mathrm{if} \ x_i > 1/2.
	\end{cases} i = 1, \ldots, d.
\end{equation}
For a sequence of i.i.d. samples $\mb{z}^{(j)}$ generated from the standard Gaussian with $j=1,2\ldots$, we compute $\mb{x}_B^{(j)}=\boldsymbol{\ell}^{-1} \circ f^{-1}_{\mathsf{KRnet}}(\mb{z}^{(j)})$. if $\mb{x}_B^{(j)}=\mathcal{P}(\mb{x}_B^{(j)})$, we assign $\mb{x}_B^{(j)}$ to $\mathsf{S}_{\Omega,k+1}$; otherwise, we add $\mathcal{P}(\mb{x}^{(j)}_B)$ to $\mathsf{S}_{\partial\Omega,k}$.
The updated training set $\mathsf{S}_{\Omega,k+1}$ and $\mathsf{S}_{\partial\Omega,k+1}$ will be used for the next training stage. This procedure is repeated until the stopping criterion is satisfied (see Algorithm \ref{alg_das_r}). Since the collocation points in $\mathsf{S}_{\Omega,k}$ will be completely replaced at the next stage, we call this type of deep adaptive sampling strategy DAS-R for short. The alternative approach given in Remark \ref{rmk:kl} can also be used to obtain $\hat{p}_{\mathsf{KRnet}}(\bx;\Theta_f^{*,(k)})$. 

We now look at strategy II.
Unlike DAS-R, the number of collocation points in the training set $\mathsf{S}_{\Omega}$ increases gradually. So we denote this type of deep adaptive sampling strategy by DAS-G for short.  Starting with an initial set of collocation points $\mathsf{S}_{\Omega, 0} = \{\mb{x}_{\Omega}^{(i)} \}_{i=1}^{n_r}$ (as well as $\mathsf{S}_{\partial \Omega, 0}$)  drawn from a uniform distribution defined on $\Omega$, we minimize the empirical loss \eqref{eq_discrete_loss}  on the training set $\mathsf{S}_{\Omega, 0}$ (as well as $\mathsf{S}_{\partial \Omega, 0}$) to obtain $u(\mb{x};\Theta_N^{*,(1)})$. Once we have $u(\mb{x};\Theta_N^{*,(1)})$ in hand, we can seek $\hat{p}_{\mathsf{KRnet}}(\mb{x};\Theta_f^{*,(1)})$ using the residual $r^2(\mb{x};\Theta_N^{*,(1)})$. We here use uniform samples to approximate and minimize the cross entropy \eqref{eqn:ce_approx}. Similar to DAS-R, a new set of collocation points $\mathsf{S}^{g}_{\Omega, 1} = \{\mb{x}_{\Omega, 1}^{g, (i)}\}_{i=1}^{n_r}$ is generated by $\hat{p}_{\mathsf{KRnet}}(\bx;\Theta_f^{*, (1)})$ while the main difference is that we update the training set as $\mathsf{S}_{\Omega, 1} = \mathsf{S}_{\Omega, 0} \cup \mathsf{S}^g_{\Omega, 1}$, in other words, $\mathsf{S}_{\Omega,0}$ is augmented rather than replaced by $\mathsf{S}_{\Omega,1}^g$. We continue to update  $u(\mb{x};\Theta)$ using $\Theta_N^{*, (1)}$ as the initial parameters and $\mathsf{S}_{\Omega, 1}$ as the training set, which yields a refined model $u(\mb{x};\Theta_N^{*, (2)})$. Staring from $k=2$, we seek $\hat{p}_{\mathsf{KRnet}}(\bx;\Theta_f^{*,(k)})$ using the approach given in Remark \ref{rmk:kl}.
We repeat the procedure to obtain an adaptive algorithm (see Algorithm \ref{alg_das_g}). 

Our two adaptive training methods are summarized in Algorithm \ref{alg_das_r} (DAS-R) and Algorithm \ref{alg_das_g} (DAS-G), where $N_{\rm adaptive}$ is a given number of maximum adaptivity iterations, $m$ is the batch size for stochastic gradient, and $N_e$ is the number of epochs for training $u(\mb{x};\Theta)$ and $\hat{p}_{\mathsf{KRnet}}(\mb{x};\Theta_f)$. The algorithms consist of three steps: solving PDE, training KRnet and refining the training set. 

\begin{remark}
	In both strategies, we use uniform samples to approximate and minimize the cross entropy to obtain $\hat{p}_{\mathsf{KRnet}}(\mb{x};\Theta_f^{*,(1)})$, after which either equation \eqref{eqn:ce_approx} or equation \eqref{eqn:KL_p2r} can be employed. The main reason of doing this is to use uniform samples to capture the modes if the residual-induced distribution is multimodal. The obtained PDF  $\hat{p}_{\mathsf{KRnet}}(\mb{x};\Theta_f^{*,(1)})$ provides a good initialization when equation \eqref{eqn:KL_p2r} is employed to seek $\hat{p}_{\mathsf{KRnet}}(\mb{x};\Theta_f^{*,(i)})$ with $i>1$. This choice is usually not necessary if the modes are not strongly localized or the location of the modes can be encoded into the prior distribution of KRnet through a Gaussian mixture model.
\end{remark}
\begin{algorithm}
	\caption{DAS-R for PDEs}
	\label{alg_das_r}
	\begin{algorithmic}[1]
		\REQUIRE Initial  $\hat{p}_{\mathsf{KRnet}}(\bx;\Theta_f^{(0)})$ , $u(\mb{x};\Theta^{(0)})$, maximum epoch number $N_e$, batch size $m$, initial training set $\mathsf{S}_{\Omega, 0} = \{\mb{x}_{\Omega, 0}^{(i)} \}_{i=1}^{N_r}$ and $\mathsf{S}_{\partial \Omega, 0} = \{\mb{x}_{\partial \Omega, 0}^{(i)} \}_{i=1}^{N_b}$.
		\FOR{$k = 0:N_{\rm adaptive}-1$}
		\STATE // Solve PDE
		\FOR {$i = 1:N_e$}
		\FOR {$j$ steps}
		\STATE Sample $m$ samples from $\mathsf{S}_{\Omega, k}$.
		\STATE Sample $m$ samples from $\mathsf{S}_{\partial \Omega, k}$.
		\STATE Update $u(\mb{x};\Theta)$ by descending the stochastic gradient of $J^{\mathsf{IS}}_N(u(\mb{x};\Theta))$ (see equation \eqref{eqn_ISdiscreteloss}).
		\ENDFOR
		\ENDFOR
		\STATE // Train KRnet
		\FOR {$i = 1:N_e$}
		\FOR {$j$ steps}
		\STATE Sample $m$ samples from $\mathsf{S}_{\Omega, k}$.
		\STATE Update $\hat{p}_{\mathsf{KRnet}}(\bx;\Theta_f)$ by descending the stochastic gradient of $H(\hat{r}_{\bX},\hat{p}_{\mathsf{KRnet}})$ (see equation \eqref{eqn:ce_approx}).
		\ENDFOR
		\ENDFOR
		\STATE // Refine training set
		\STATE Generate  $\mathsf{S}_{\Omega, k+1} \subset \Omega$ through $\hat{p}_{\mathsf{KRnet}}(\bx;\Theta_f^{*,(k+1)})$.
		\ENDFOR	
		\ENSURE $u(\mb{x};\Theta_N^*)$
	\end{algorithmic}
\end{algorithm}

\begin{algorithm}
	\caption{DAS-G for PDEs}
	\label{alg_das_g}
	\begin{algorithmic}[1]
		\REQUIRE Initial  $\hat{p}_{\mathsf{KRnet}}(\bx;\Theta_f^{(0)})$ , $u(\mb{x};\Theta^{(0)})$, maximum epoch number $N_e$, batch size $m$, initial training set $\mathsf{S}_{\Omega, 0} = \{\mb{x}_{\Omega, 0}^{(i)} \}_{i=1}^{n_r}$ and $\mathsf{S}_{\partial \Omega, 0} = \{\mb{x}_{\partial \Omega, 0}^{(i)} \}_{i=1}^{N_b}$.
		\FOR{$k = 0:N_{\rm adaptive}-1$}
		\STATE // Solve PDE
		\FOR {$i = 1:N_e$}
		\FOR {$j$ steps}
		\STATE Sample $m$ samples from $\mathsf{S}_{\Omega, k}$.
		\STATE Sample $m$ samples from $\mathsf{S}_{\partial \Omega, k}$.
		\STATE Update $u(\mb{x};\Theta)$ by descending the stochastic gradient of $J_N(u(\mb{x};\Theta))$ (see equation \eqref{eq_discrete_loss}).
		\ENDFOR
		\ENDFOR
		\STATE // Train KRnet
		\FOR {$i = 1:N_e$}
		\FOR {$j$ steps}
		\STATE Sample $m$ samples from $\mathsf{S}_{\Omega, k}$.
		\STATE Update $\hat{p}_{\mathsf{KRnet}}(\bx;\Theta_f)$ by descending the stochastic gradient of $H(\hat{r}_{\bX},\hat{p}_{\mathsf{KRnet}})$ (see equation \eqref{eqn:ce_approx}).
		\ENDFOR
		\ENDFOR
		\STATE // Refine training set
		\STATE Generate  $\mathsf{S}^g_{\Omega, k+1} \subset \Omega$ with size $n_r$ through $\hat{p}_{\mathsf{KRnet}}(\bx;\Theta_f^{*,(k+1)})$.
		\STATE $\mathsf{S}_{\Omega, k+1} = \mathsf{S}_{\Omega, k} \cup \mathsf{S}^g_{\Omega, k+1}$.
		\ENDFOR	
		\ENSURE $u(\mb{x};\Theta_N^*)$
	\end{algorithmic}
\end{algorithm}

\section{Analysis of DAS}\label{sec_das_analysis}

As discussed in section \ref{sec_das}, the key point of our DAS method is to achieve variance reduction for the discretization of the residual loss, based on which we expect to improve the accuracy of the approximate solution. Under certain conditions, we show that the expectation of error bound becomes smaller when the adaptive sampling strategy is employed. 

\begin{assump} \label{assump_norm_relations} \cite{bochev2016least}
	In problem \eqref{eq_pde}, we let $F =\mathscr{H}$ be a Hilbert space and $\mathcal{L}$ a linear operator. Assume that the differential operator $\mathcal{L}$ and the boundary operator $\mathfrak{b}$ satisfy 
	\begin{equation}
		C_1 \norm{v}{2, \Omega} \leq \norm{\mathcal{L} v}{2, \Omega} + \norm{\mathfrak{b} v}{2, \partial \Omega} \leq C_2 \norm{v}{2, \Omega} \quad  \forall  v \in \mathscr{H}
	\end{equation}
	where $\mathscr{H}$ is a Hilbert space defined on $\Omega$ and the positive constants $C_1$ and $C_2$ are independent of $v$. 
\end{assump}

The above condition is called the stability bound \cite{bochev2016least}, which is essential to the existence and uniqueness of problem \eqref{eq_pde}.
Except for this assumption, the following two assumptions for the relationship between $r(\mb{x};\Theta_N^{*, (k)})$ and $\hat{p}_{\mathsf{KRnet}}(\mb{x};\Theta_f^{*,(k)})$ are given.

\begin{assump} \label{assump_res_pdf}
	Assume that $\hat{p}_{\mathsf{KRnet}}(\mb{x};\Theta_f^{*,(k)})$ is the optimal candidate for the change of measure in equation \eqref{eq_imploss}
	\begin{equation}
		\hat{p}_{\mathsf{KRnet}}(\mb{x};\Theta_f^{*,(k)}) = c_k r^2(\mb{x};\Theta_N^{*, (k)})
	\end{equation}
	where $c_k = 1/ \int_{\Omega} r^2(\mb{x};\Theta_N^{*,(k)})d\mb{x}$ is the normalization constant.
\end{assump}

\begin{assump} \label{assump_bounded_var}
	Let
	\begin{equation*}
		R_k = \frac{1}{N_r}\sum\limits_{i=1}^{N_r} \frac{ r^2(\mb{x}_{\Omega}^{(i)};\Theta_{N}^{*,(k)})}{\hat{p}_{\mathsf{KRnet}}(\mb{x}_{\Omega}^{(i)}; \Theta_f^{*, (k-1)})}
	\end{equation*}
	be the discrete residual loss at the $k$-th stage, where each $\mb{x}_{\Omega}^{(i)}$ is drawn from $\hat{p}_{\mathsf{KRnet}}(\mb{x};\Theta_f^{*, (k-1)})$. Assume that $r^2(\mb{x}_{\Omega}^{(i)};\Theta_{N}^{*,(k)})/\hat{p}_{\mathsf{KRnet}}(\mb{x}_{\Omega}^{(i)}; \Theta_f^{*, (k-1)})  \in [\tau_1, \tau_2]$ almost surely for each $i = 1, 2, \ldots, N_r$.
\end{assump}

At each adaptivity stage, the error of the approximate solution is estimated as follows.
\begin{thm} \label{thm_error_bound}
	Let $u(\mb{x};\Theta_N^{*, (k)}) \in F$ be a solution of \eqref{eq_discrete_loss} where the collocation points are independently drawn from $\hat{p}_{\mathsf{KRnet}}(\mb{x};\Theta_f^{*,{(k-1)}})$. Suppose that Assumption \ref{assump_norm_relations} and Assumption \ref{assump_bounded_var} are satisfied. Given $0 < \varepsilon < 1$, the following error estimate holds 
	\begin{equation*}
		\norm{u(\mb{x};\Theta_N^{*, (k)}) - u(\mb{x})}{2, \Omega} \leq \sqrt{2} C_1^{-1} \left( R_k + \varepsilon + \norm{b(\mb{x};\Theta_N^{*,(k)})}{2, \partial \Omega}^2 \right)^{\frac{1}{2}}.
	\end{equation*}
	with probability at least $1 - \mathrm{exp}(-2N_r \varepsilon^2/(\tau_2-\tau_1)^2)$.
\end{thm} 

The approximate solutions at two adjacent adaptivity stages satisfy: 
\begin{corollary}\label{cor:R_k}
	Under the same conditions of Theorem \ref{thm_error_bound}, suppose that Assumption \ref{assump_res_pdf} is satisfied and the boundary loss $J_b(u)$ is zero,  then the following inequality holds
	\begin{equation*}
		\mathbb{E} (R_{k+1})  \leq \mathbb{E} (R_{k}).
	\end{equation*}
\end{corollary}

Theorem \ref{thm_error_bound} provides an error estimate for the approximate solution similar to the results in \cite{shin2020error}. Corollary \ref{cor:R_k} outlines the error behavior of a sequence of approximate solutions  induced by adaptivity. However, it is not quite straightforward to quantify the decay of the error due to adaptive refinement. For example, for DAS-R the reduction in variance is up to a tail probability as shown in Lemma \ref{lem:kl_is} and for DAS-G the loss is changed at each adaptivity stage due to the modification of the training set. These issues are left for future study.

\section{Numerical experiments}\label{sec_numexp}

In this section, we conduct some numerical experiments to demonstrate the effectiveness of the proposed DAS method, including two low-dimensional and low regularity test problems, one high-dimensional linear test problem, and one high-dimensional nonlinear test problem. Due to the curse of dimensionality, data are sparse in high-dimensional spaces \cite{vershynin2018high, blum2020foundations, wright2021high}, which implies that effective samples should be able to deal with localized information. We mainly use low-dimensional and low regularity test problems to demonstrate that the sampling strategy affects significantly the performance of neural network approximation if the residual is strongly localized.  For comparison, we also test the performance of a heuristic method based on residual refinement \cite{lu2021deepxde,yu2022gradient} (see section \ref{sec_numexp_hd_linear} and section \ref{sec_numexp_hd_nonlinear}) for high-dimensional problems, where the heuristic method searches uniform samples to find the points with large residuals and add them to the current training set. The heuristic method is similar to DAS-G, where the main difference is that the selection of new samples in DAS-G is completely guided by an optimization problem while the heuristic method relies partially on the user's intuition. All deep neural network models are trained by the ADAM method \cite{kingma2017adam}. The penalty parameter in equation \eqref{eq_discrete_loss} is set to $\hat{\gamma} = 1$. The activation function of $u(\mb{x};\Theta)$ is set to the hyperbolic tangent function. The activation function of KRnet is the rectified linear unit (ReLU) function since we only use the KRnet for density approximation.

\subsection{Low-dimensional and low regularity test problems} \label{sec_numexp_peak}
In this part, two-dimensional low regularity problems are considered, where the solution of the first one has a peak and the solution of the second one has two peaks. 

\subsubsection{Two-dimensional peak problem}
The following elliptic equation is considered
\begin{equation}
	\begin{aligned}
		- \Delta u(x_1, x_2) &= s(x_1,x_2) \quad \text{in} \ \Omega, \\
		u(x_1, x_2) &= g(x_1, x_2) \quad \text{on} \  \partial \Omega,
	\end{aligned}	
\end{equation}
where the computation domain is $\Omega = [-1,1]^2$. In order to quantify the error, we use the following reference solution
\begin{equation*}
	u(x_1, x_2) = \mathrm{exp} \left(  -1000 [(x_1 -r_c)^2 + (x_2 - r_c)^2 ]\right),
\end{equation*}
which has a peak at the point $(r_c, r_c)$ and decreases rapidly away from $(r_c, r_c)$. This test problem is often used to test the performance of adaptive finite element methods \cite{mitchell2013collection,morin2002convergence}. 

We choose a six-layer fully connected neural network $u(\mb{x};\Theta)$ with $32$ neurons to approximate the solution. For KRnet, we take $L = 6$ affine coupling layers, and two fully connected layers with $24$ neurons for each affine coupling layer. The number of epochs for training both $u(\mb{x};\Theta)$ and $p(\mb{x};\Theta_f)$ is set to $N_e = 3000$. The learning rate for ADAM optimizer is set to $0.0001$, and the batch size is set to $m = 500$. Here, we set $(r_c, r_c) = (0.5, 0.5)$. To assess the effectiveness of our DAS methods,  we generate a uniform meshgrid with size $256 \times 256$ in $[-1,1]^2$ and compute the mean square error on these grid points. 

In Figure \ref{fig:peak2d_error_comparison}, we plot the approximation errors given by different sampling strategies with respect to the sample size in the left plot and with respect to the number of epochs in the right plot. 
For each $|\mathsf{S}_{\Omega}|$, we take three runs with different random seeds for initialization and compute the mean error of the three runs as the final error. For DAS strategies, the numbers of adaptivity iterations are set to $N_{\rm{adaptive}} = 4, 6, 8, 10$ for $\vert \mathsf{S}_{\Omega} \vert = 2 \times 10^3, 3 \times 10^3, 4 \times 10^3, 5 \times 10^3$ respectively, and  $n_r = 500$ is set for the DAS-G strategy (see section \ref{sec_das}). For the uniform sampling strategy, the number of epochs is set to be the same as the total number of epochs of each DAS method. It is clear that  for this test problem the DAS methods (DAS-G and DAS-R) have a better performance than the uniform sampling strategy and DAS-R performs better than DAS-G. For the same sample size, both DAS-R and DAS-G yield a smaller error than the uniform sampling method. In terms of the number of epochs, the errors of DAS-R and DAS-G decay in a more consistent way than the uniform sampling method.

In Figure \ref{fig:peak2d_sol} we compare the exact solution, the DAS solutions given by $5 \times 10^3$ nonuniform samples and the approximate solution given by $5 \times 10^3$ uniform samples. It is seen that DAS methods are much more effective than the uniform sampling method to capture the information in the region of low regularity. Figure \ref{fig:peak2d_error_at_k} shows the error evolution of DAS-R at different adaptivity iteration steps. It is seen that the approximation error drops as the adaptivity iteration step $k$ increases, which is consistent with Corollary \ref{cor:R_k}, and the relaxation time for the optimization iterations reduces as well.  Figure \ref{fig:peak2d_dasr_resample} shows the evolution of the training sets ($\vert \mathsf{S}_{\Omega} \vert = 5 \times 10^3$) of DAS-R method with respect to adaptivity iterations $k = 1, 4, 7, 9$, where the initial training set $\mathsf{S}_{\Omega, 0}$ consists of uniform collocation points on $\Omega$ (see section \ref{sec_das}). It is seen that the largest density of $\mathsf{S}_{\Omega,1}$ for DAS-R is around $[0.5,0.5]$ since $\mathsf{S}_{\Omega,k}$ is consistent with the residual-induced distribution. However, we also expect that the tail of the residual-induced distribution becomes heavier as $k$ increases since the adaptivity tries to make the residual-induced distribution more uniform, which is illustrated by $\mathsf{S}_{\Omega,4}$, $\mathsf{S}_{\Omega,7}$ and $\mathsf{S}_{\Omega,9}$. 
\begin{figure}
	\center{
		\includegraphics[width=0.42\textwidth]{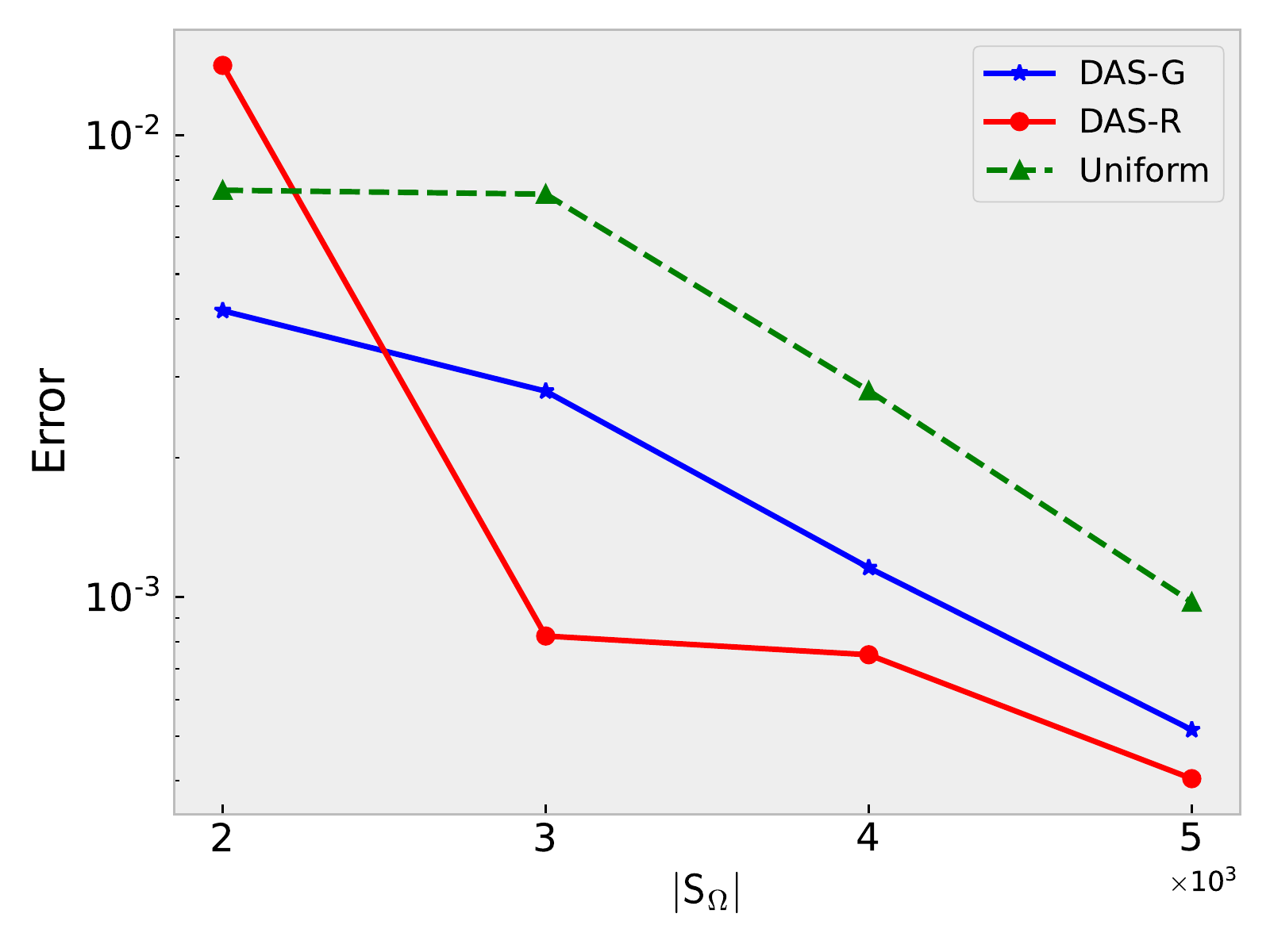}
		\includegraphics[width=0.42\textwidth]{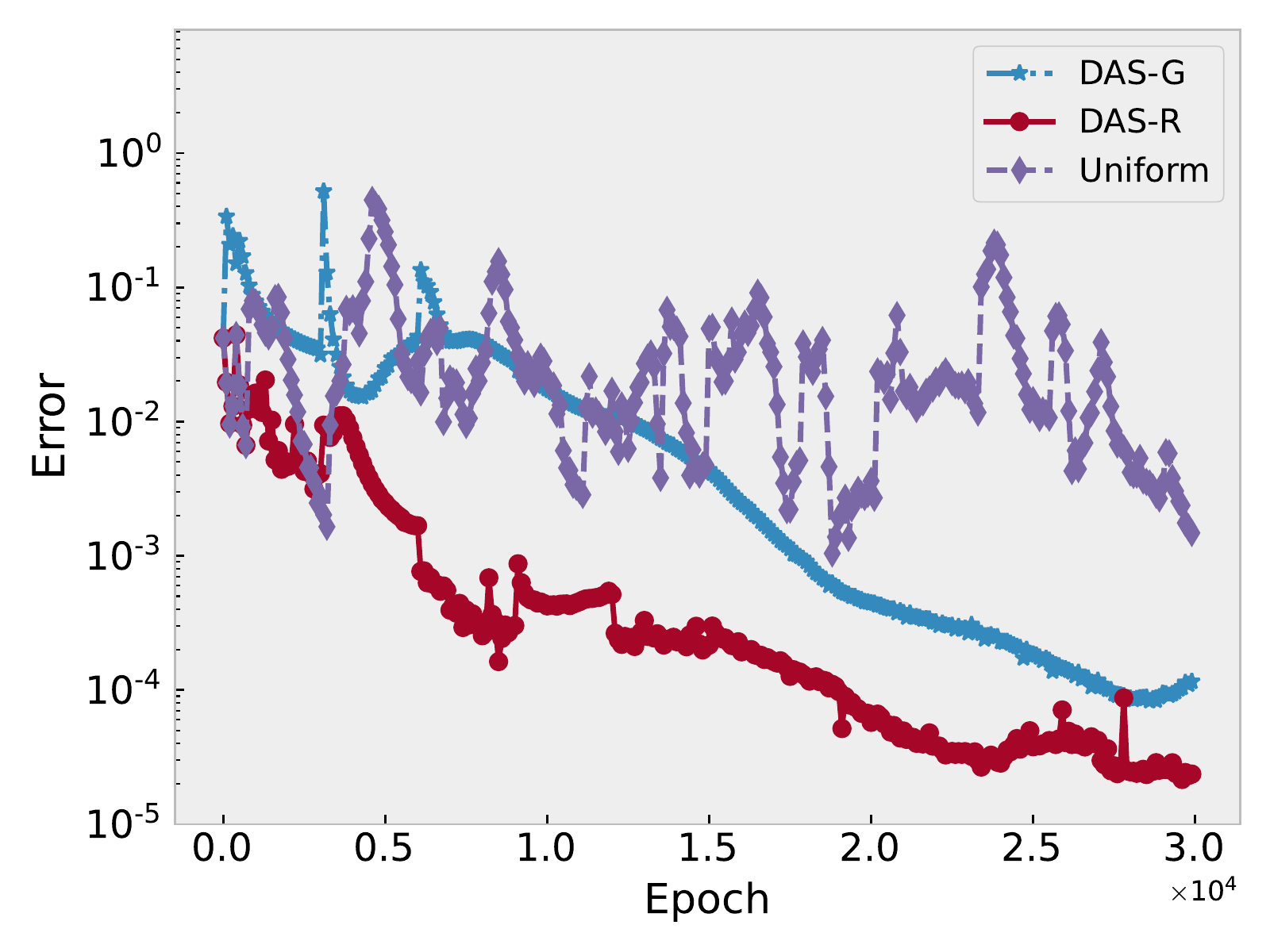}
	}
	\caption{Approximation errors for the two-dimensional peak test problem. Left: The error w.r.t sample size $\vert \mathsf{S}_{\Omega} \vert$; Right: The error w.r.t epoch for $\vert \mathsf{S}_{\Omega} \vert = 5 \times 10^3$. }\label{fig:peak2d_error_comparison}
\end{figure}

\begin{figure}[!ht]
	\centering
	\subfloat[][The exact solution. ]{\includegraphics[width=.38\textwidth]{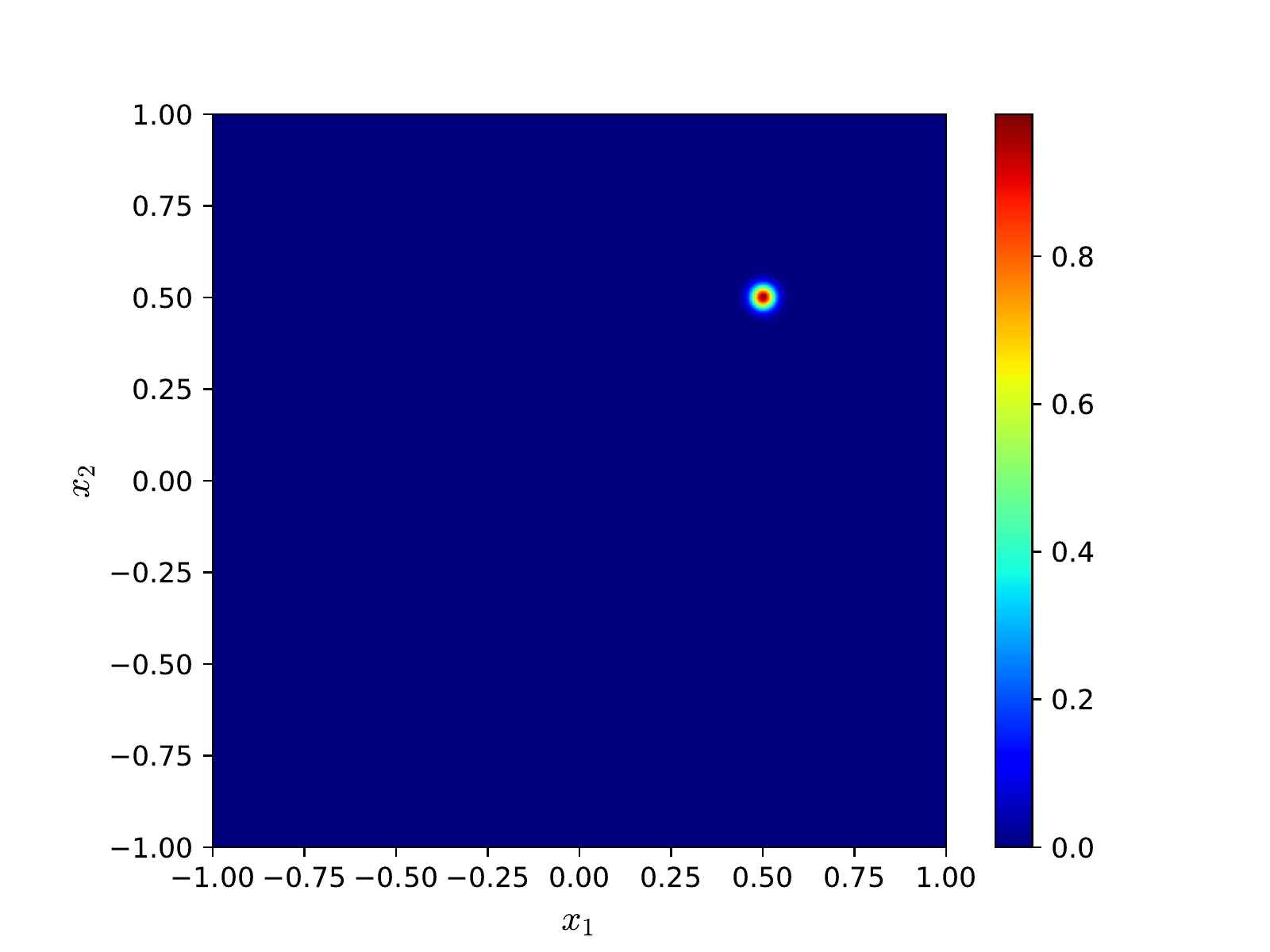}}\quad
	\subfloat[][DAS-R approximation. ]{\includegraphics[width=.38\textwidth]{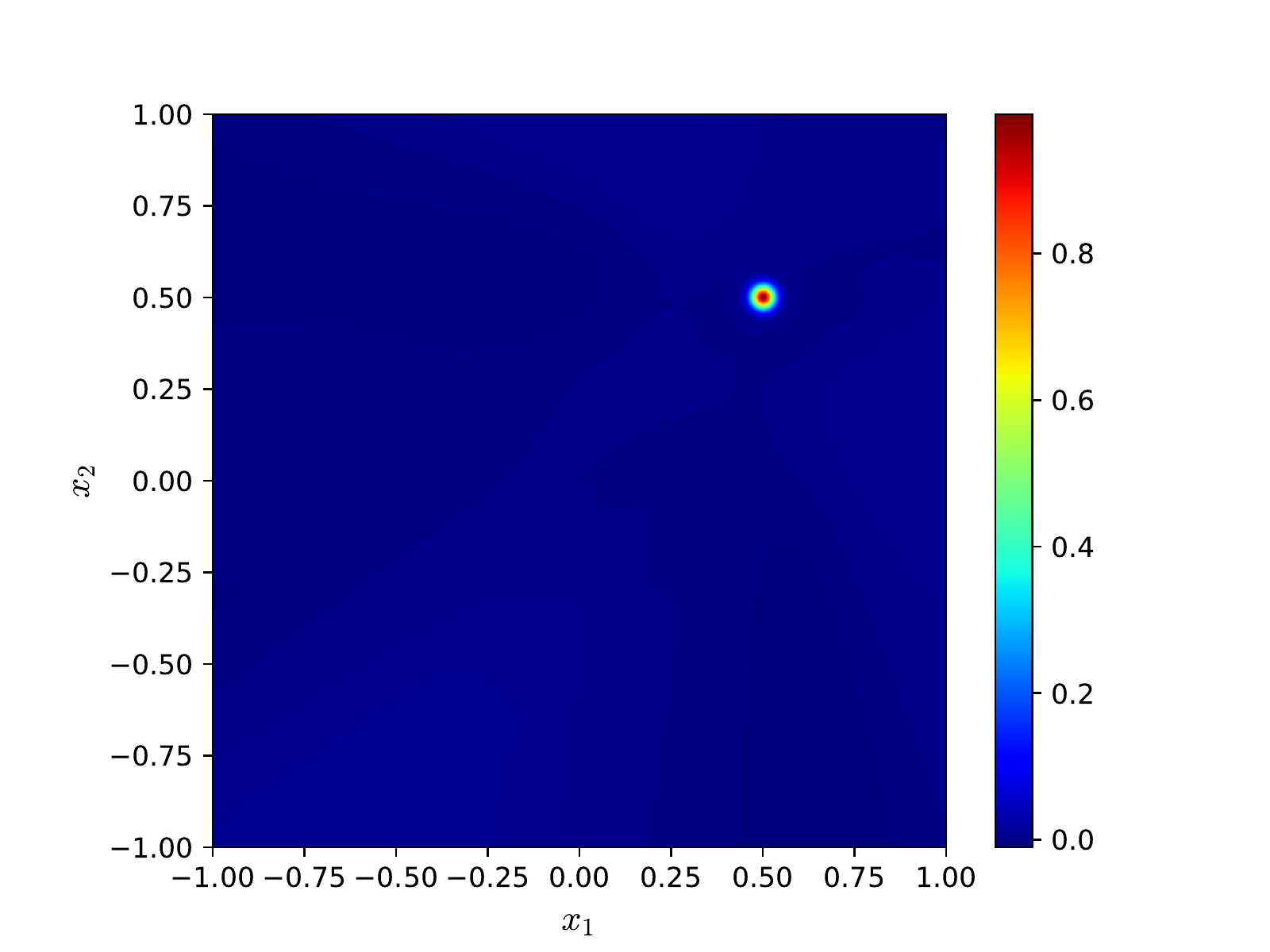}}\\
	\subfloat[][DAS-G approximation. ]{\includegraphics[width=.38\textwidth]{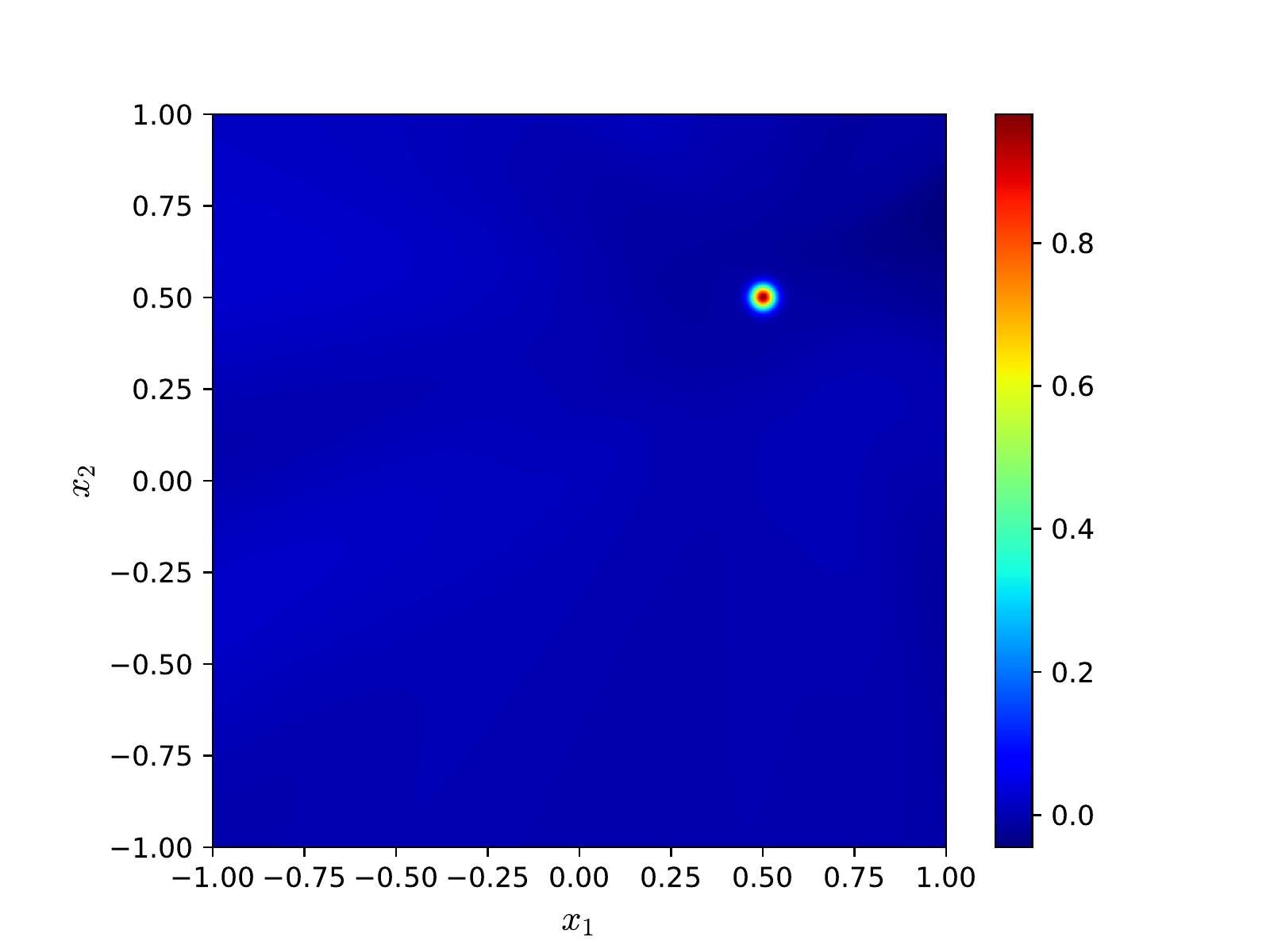}}\quad
	\subfloat[][Approximation using the uniform sampling strategy. ]{\includegraphics[width=.38\textwidth]{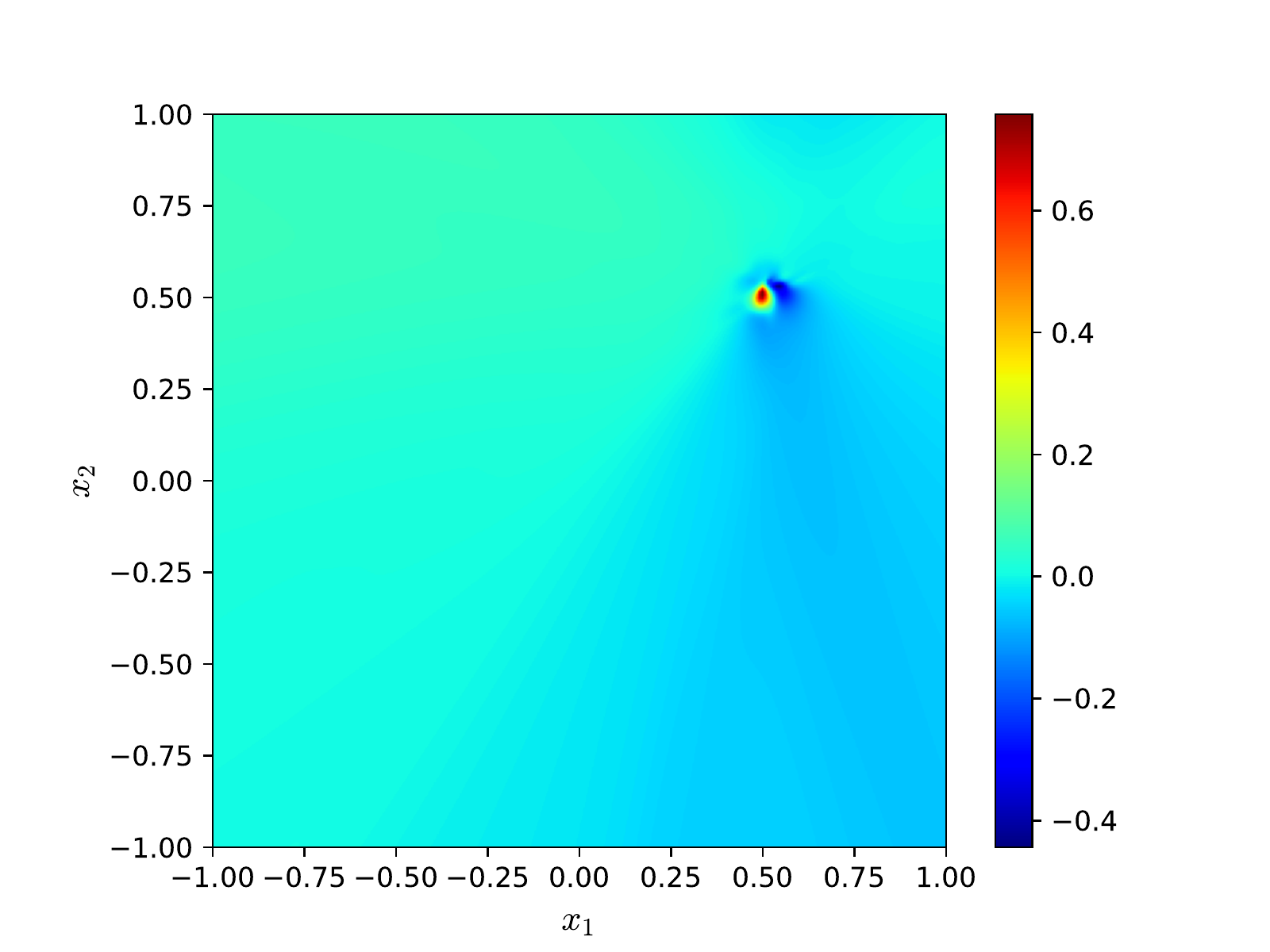}}
	\caption{Solutions, two-dimensional peak test problem. }
	\label{fig:peak2d_sol}
\end{figure}

\begin{figure}
	\center{
		\includegraphics[width=0.48\textwidth]{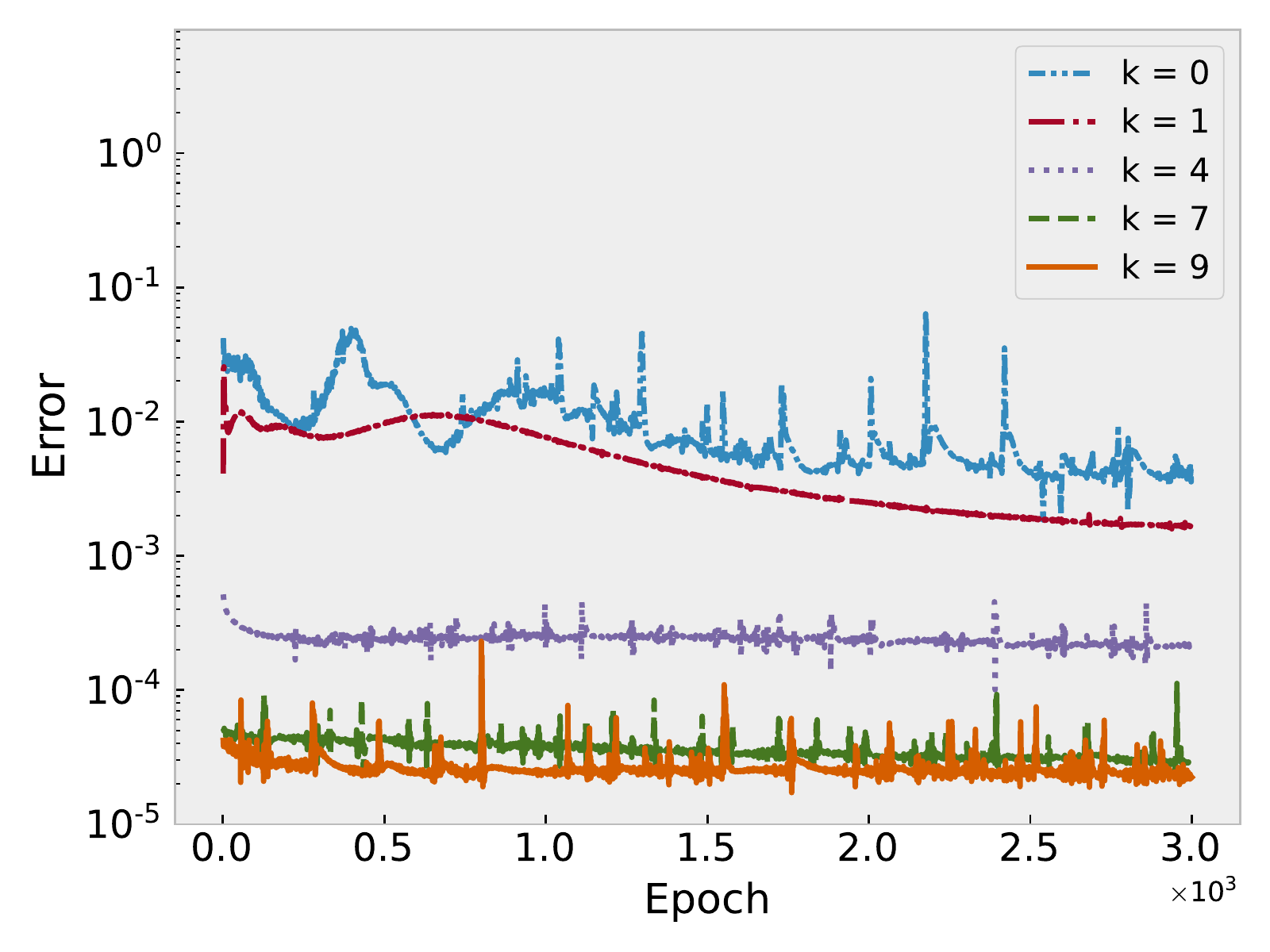}
	}
	\caption{The errors of DAS-R at certain adaptivity iteration steps for the two-dimensional peak test problem. $\vert \mathsf{S}_{\Omega} \vert = 5 \times 10^3$.}\label{fig:peak2d_error_at_k}
\end{figure}

\begin{figure}
	\center{
		\includegraphics[width=0.7\textwidth]{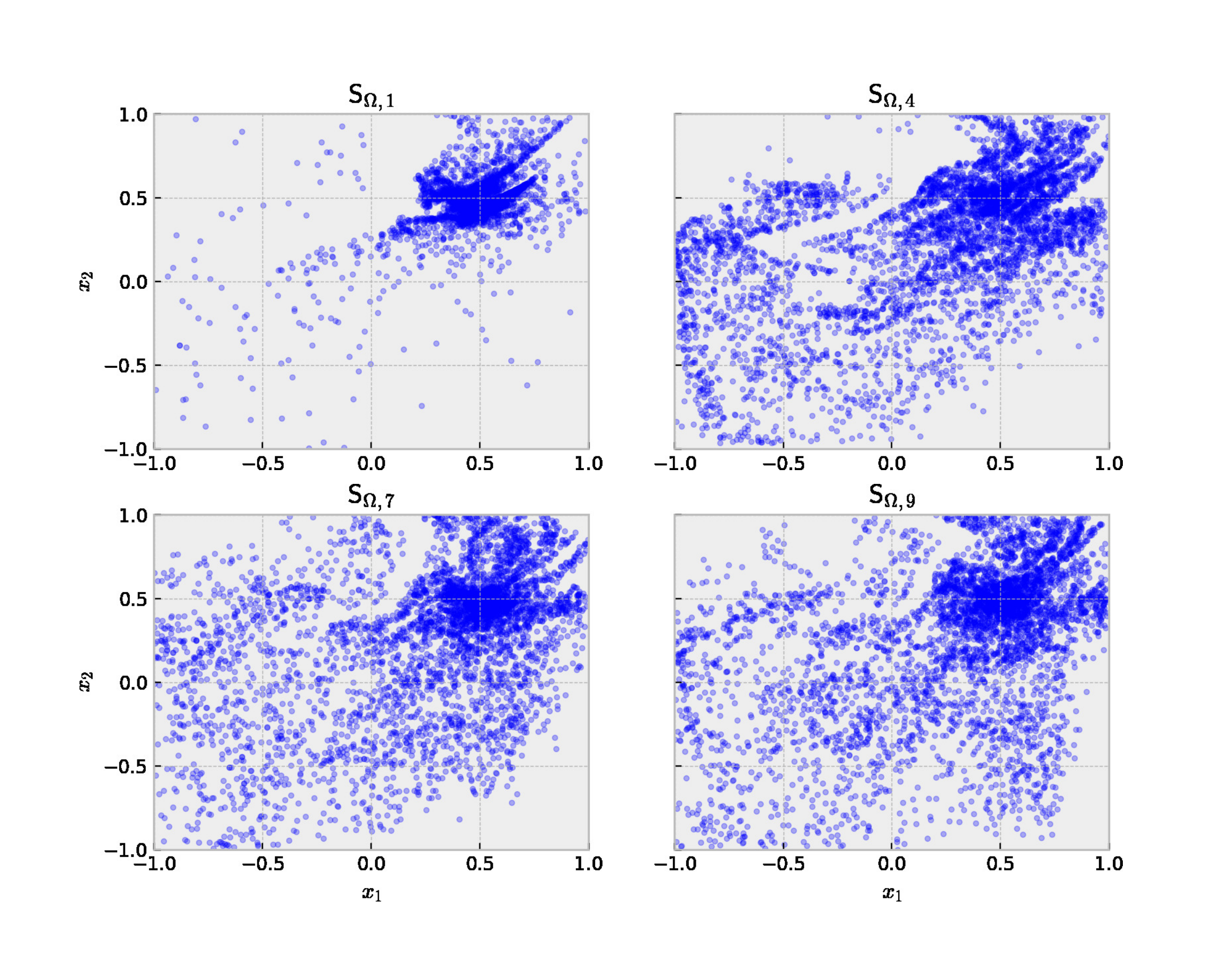}
	}
	\caption{The evolution of $\mathsf{S}_{\Omega, k}$ in DAS-R for the two-dimensional peak test problem.}\label{fig:peak2d_dasr_resample}
\end{figure}

\subsubsection{Two-dimensional test problem with two peaks}
In this test problem, we consider the following equation
\begin{equation}  \label{eq_two_peaks_pde}
	\begin{aligned}
		-\nabla \cdot \left [u(x_1, x_2) \nabla (x_1^2 + x_2^2) \right ] &+ \nabla^2 u(x_1, x_2) = s(x_1, x_2) \quad \text{in} \ \Omega, \\
		u(x_1, x_2) &= g(x_1, x_2) \quad \text{on} \  \partial \Omega,
	\end{aligned}	
\end{equation}
where the computation domain is $\Omega = [-1,1]^2$. The exact solution of \eqref{eq_two_peaks_pde} is chosen as
\begin{equation*}
	u(x_1, x_2) = \mathrm{e} ^{  -1000 [(x_1 -0.5)^2 + (x_2 - 0.5)^2 ]} + \mathrm{e}^ {  -1000 [(x_1 + 0.5)^2 + (x_2 + 0.5)^2 ]}, 
\end{equation*}
which has two peaks at the points $(0.5, 0.5)$ and $(-0.5, -0.5)$. Here, the Dirichlet boundary condition on $\partial \Omega$ is given by the exact solution.

We choose a six-layer fully connected neural network $u(\mb{x};\Theta)$ with $64$ neurons to approximate the solution of \eqref{eq_two_peaks_pde}. For KRnet, we take $L = 8$ affine coupling layers, and two fully connected layers with $48$ neurons for each affine coupling layer. The number of epochs for training both $u(\mb{x};\Theta)$ and $p(\mb{x};\Theta_f)$ is set to $N_e = 5000$. The learning rate for ADAM optimizer is set to $0.0001$, and the batch size is set to $m = 500$. Again, we generate a uniform meshgrid with size $256 \times 256$ in $[-1,1]^2$ and compute the mean square error on these grid points to assess the effectiveness of our DAS methods. 

Figure \ref{fig:2peaks2d_error_comparison} shows the approximation errors for this test problem, where the left one displays the errors with respect to the sample size $\vert \mathsf{S}_{\Omega} \vert$ for different sampling strategies, and the right one shows the error evolution of DAS-G at different adaptivity iteration steps. For each $\vert \mathsf{S}_{\Omega} \vert$, we again take three runs with different random seeds for initialization and compute the mean error of the three runs as the final error. For the DAS-G strategy, the numbers of adaptivity iterations is set to $N_{\rm{adaptive}} = 5$ (also for DAS-R), and the numbers of collocation points in $\mathsf{S}^g_{\Omega} (k = 1,2,3,4)$ is set to $n_r = 500, 1 \times 10^3, 1.5 \times 10^3, 2 \times 10^3$ for
$\vert \mathsf{S}_{\Omega} \vert = 2.5 \times 10^3, 5 \times 10^3, 7.5 \times 10^3,  10^4$ respectively. For the uniform sampling strategy, we train the model
with $2.5 \times 10^4 $ epochs to match the total number of epochs of DAS methods. From Figure \ref{fig:2peaks2d_error_comparison}, it is seen that for this test problem our DAS methods (DAS-G and DAS-R) have a better performance than the uniform sampling strategy and DAS-G performs better than
DAS-R. It is also seen that the error decreases as the adaptivity iteration step $k$ increases.

In Figure \ref{fig:2peaks2d_sol} we compare the exact solution, the DAS solutions given by $10^4$ nonuniform samples and the approximate solution given by $10^4$ uniform samples. It is seen that DAS methods are much more effective than the uniform sampling method to capture the information around the two peaks. 
Figure \ref{fig:2peaks2d_dasg_resample} shows the evolution of $\mathsf{S}^g_{\Omega, k}$ of DAS-G method with respect to adaptivity iterations $k = 1, 2, 3, 4$ ($ \vert \mathsf{S}^g_{\Omega, k} \vert = 2 \times 10^3$), where the initial training set $\mathsf{S}_{\Omega, 0}$ consists of uniform collocation points on $\Omega$ (see section \ref{sec_das}). $\mathsf{S}_{\Omega,1}^g$ shows that the error profile has two peaks. After the training set is augmented with $\mathsf{S}_{\Omega,1}$, the error profile becomes more flat as shown by the distribution of $\mathsf{S}_{\Omega,2}^g$. After the training set is augmented with $\mathsf{S}_{\Omega,2}$, the largest error is found again around the two peaks, and then the subsequent augmentation of the training set yields a more flat error profile. Such a pattern is repeated until no improvement can be reached. 

\begin{figure}
	\center{
		\includegraphics[width=0.42\textwidth]{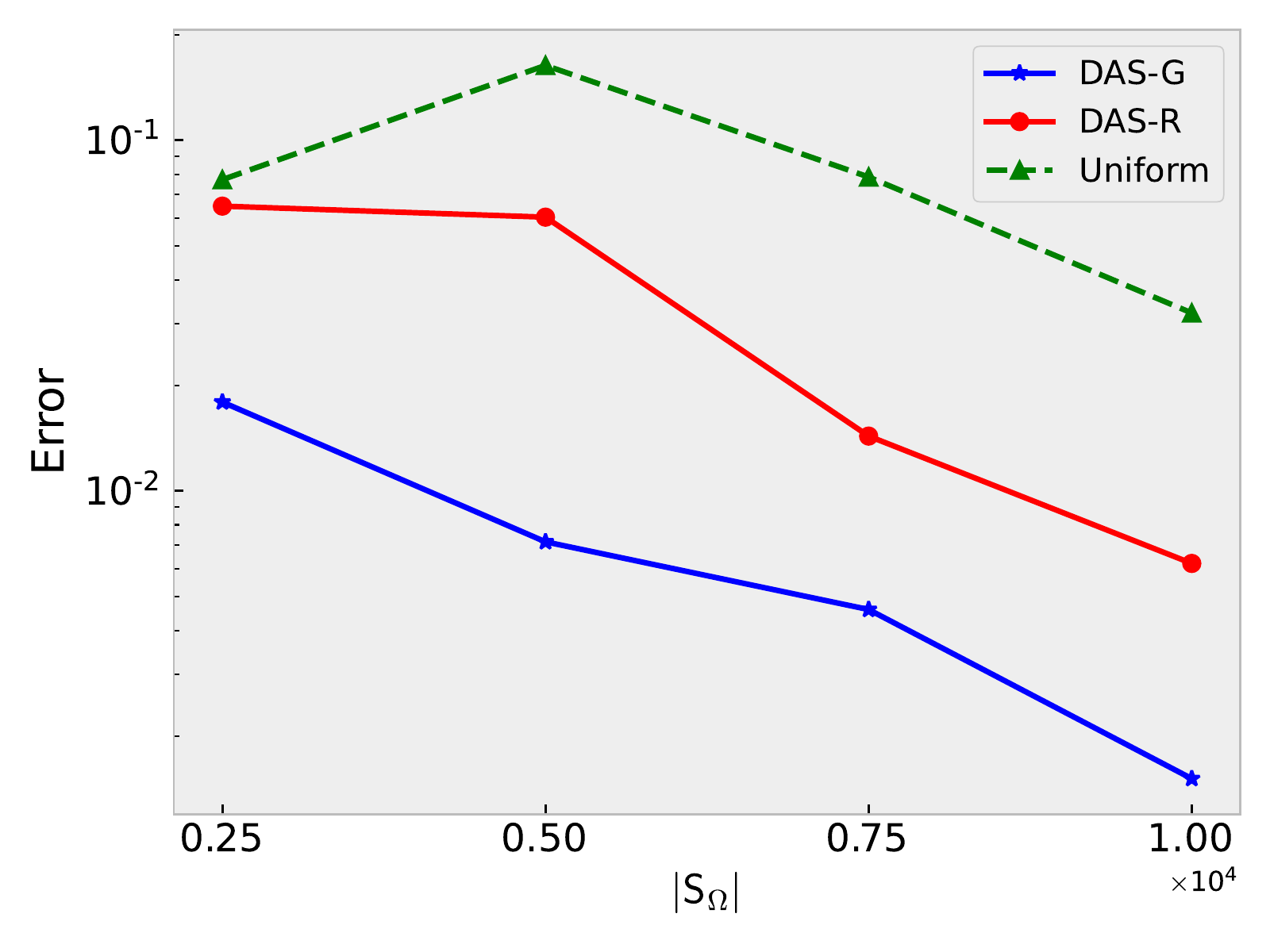}
		\includegraphics[width=0.42\textwidth]{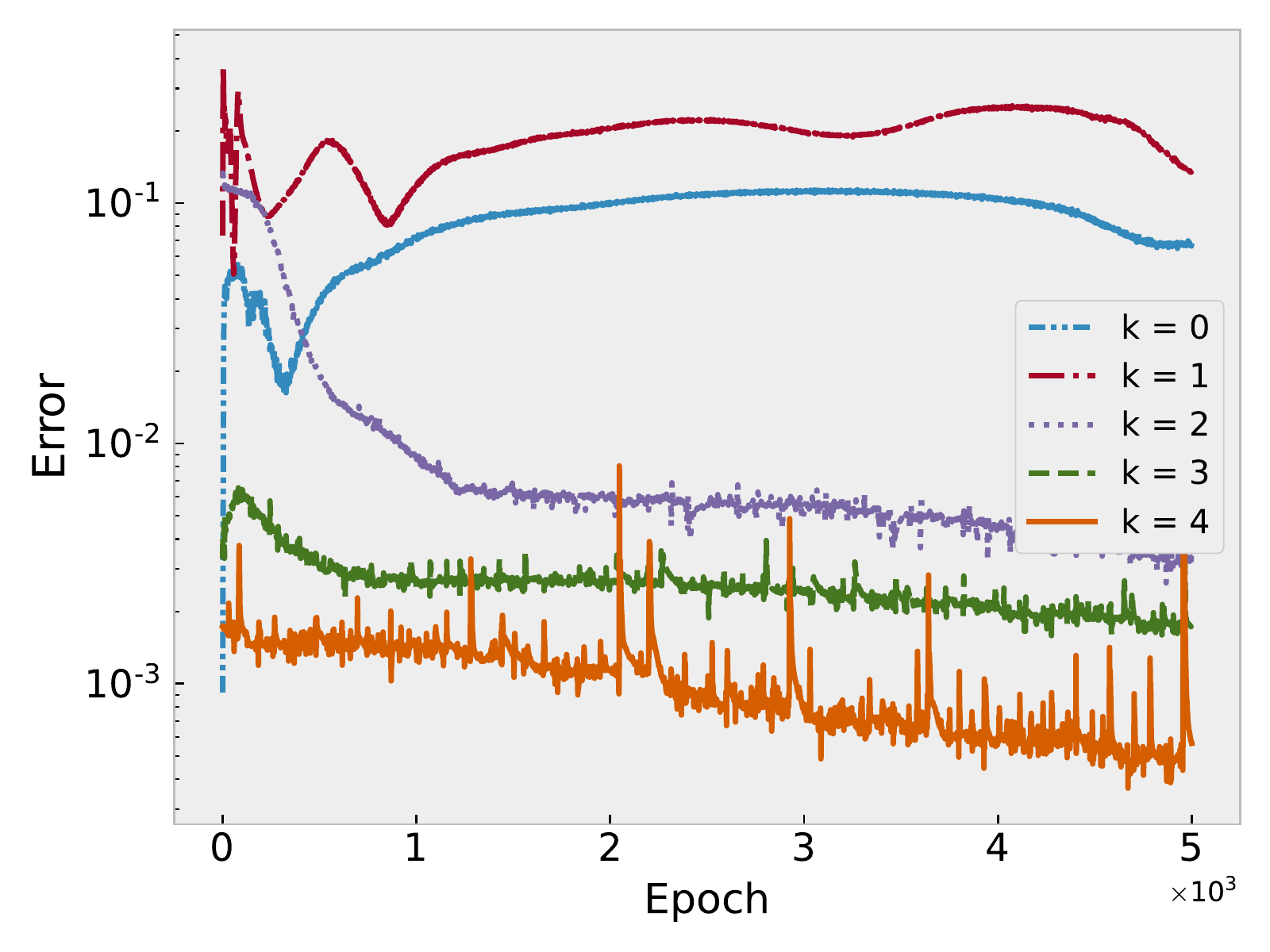}
	}
	\caption{Approximation errors for the two-dimensional test problem with two peaks. Left: The error w.r.t sample size $\vert \mathsf{S}_{\Omega} \vert$; Right: The errors of DAS-G at each adaptivity iteration steps. $\vert \mathsf{S}_{\Omega} \vert =  10^4$. }\label{fig:2peaks2d_error_comparison}
\end{figure}

\begin{figure}[!ht]
	\centering
	\subfloat[][The exact solution. ]{\includegraphics[width=.38\textwidth]{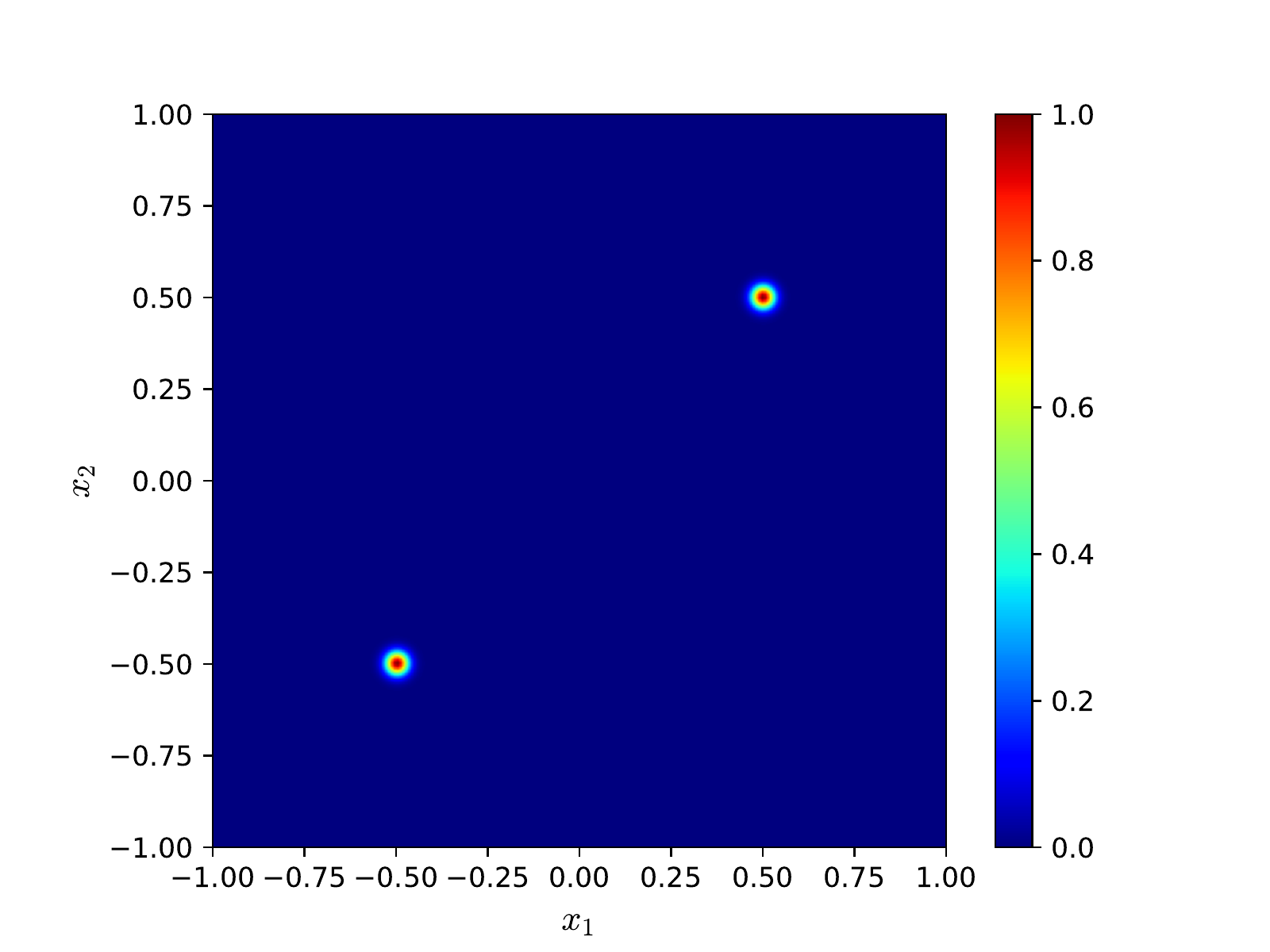}}\quad
	\subfloat[][DAS-R approximation. ]{\includegraphics[width=.38\textwidth]{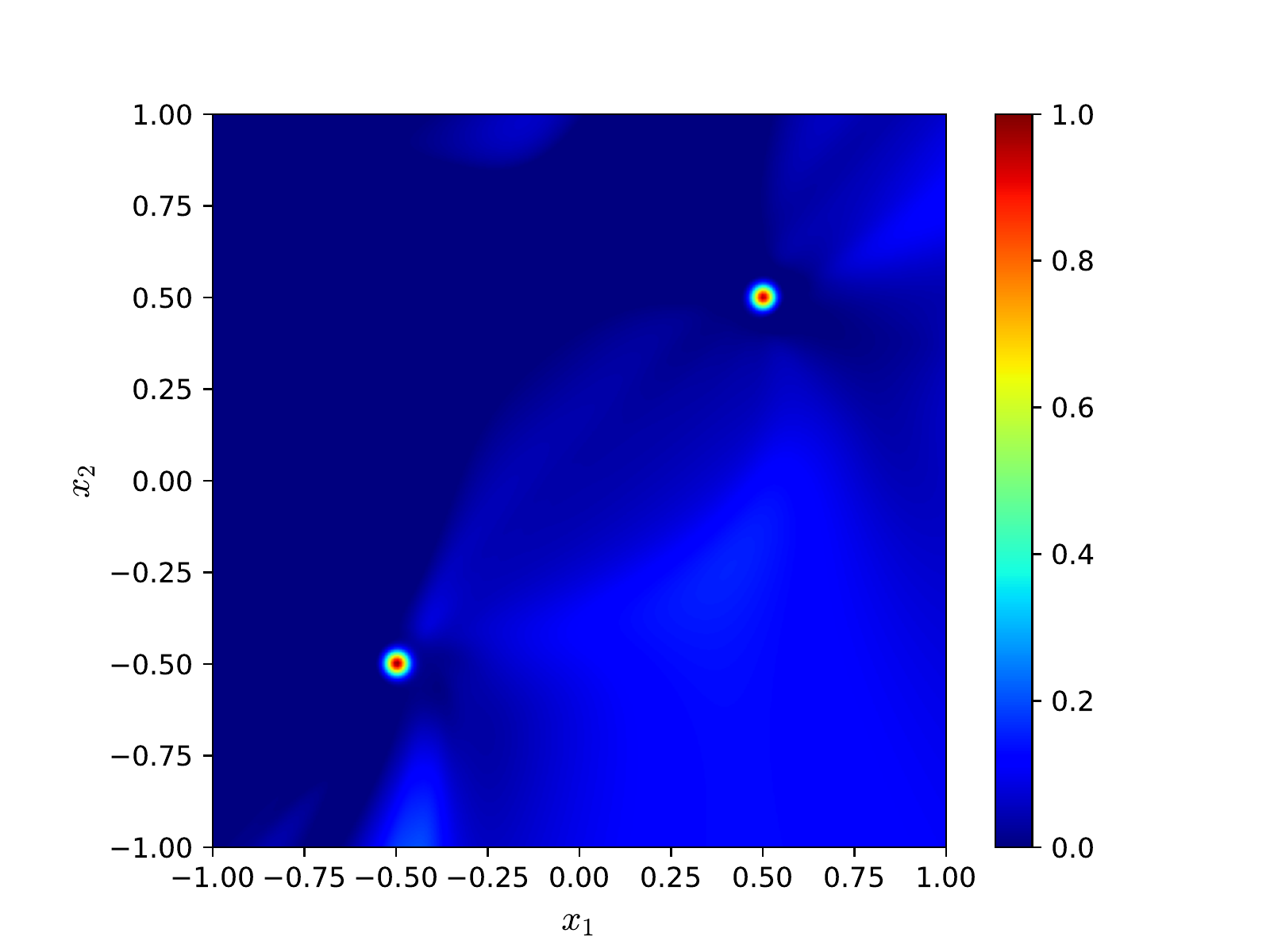}}\\
	\subfloat[][DAS-G approximation. ]{\includegraphics[width=.38\textwidth]{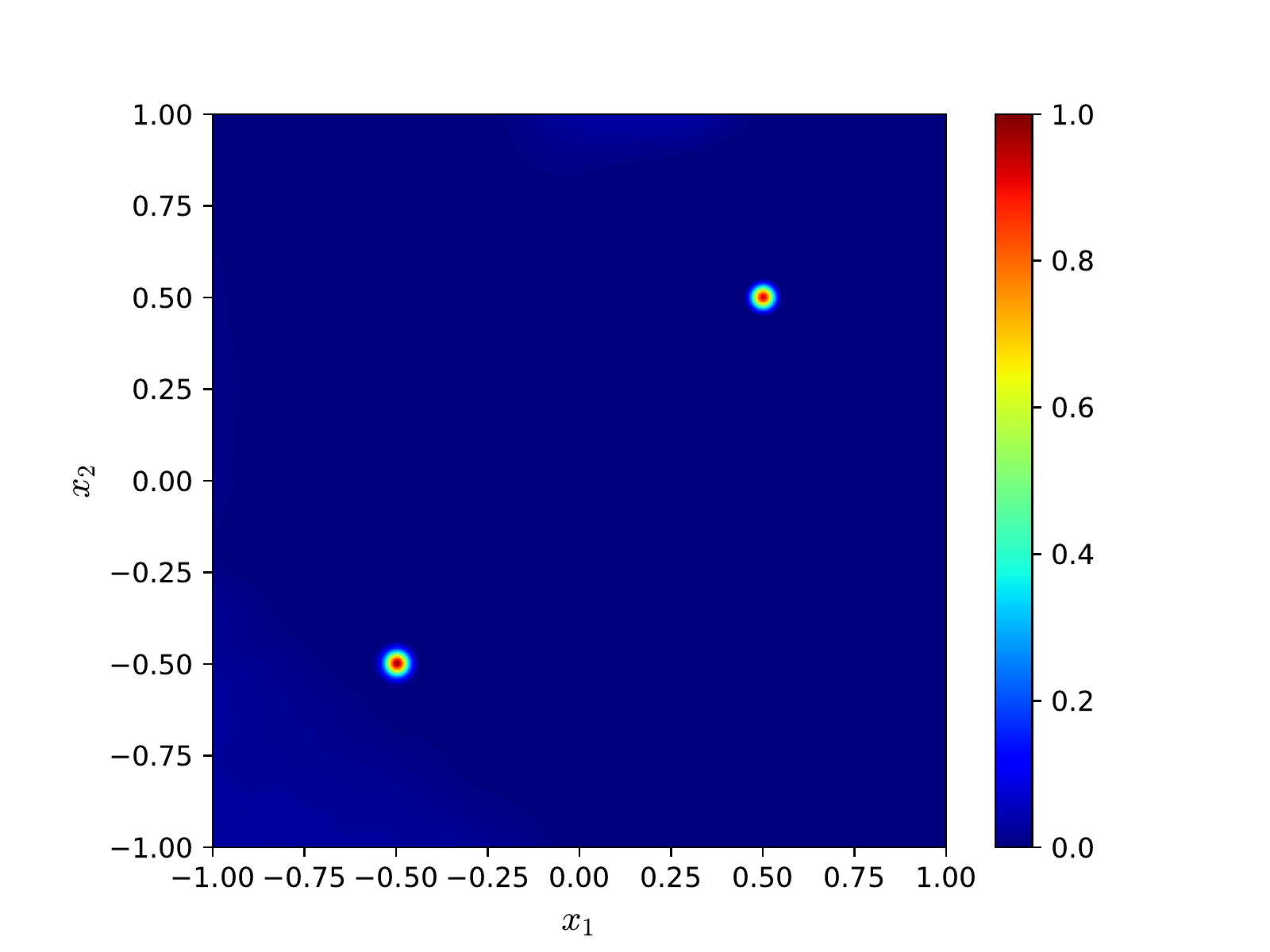}}\quad
	\subfloat[][Approximation using the uniform sampling strategy. ]{\includegraphics[width=.38\textwidth]{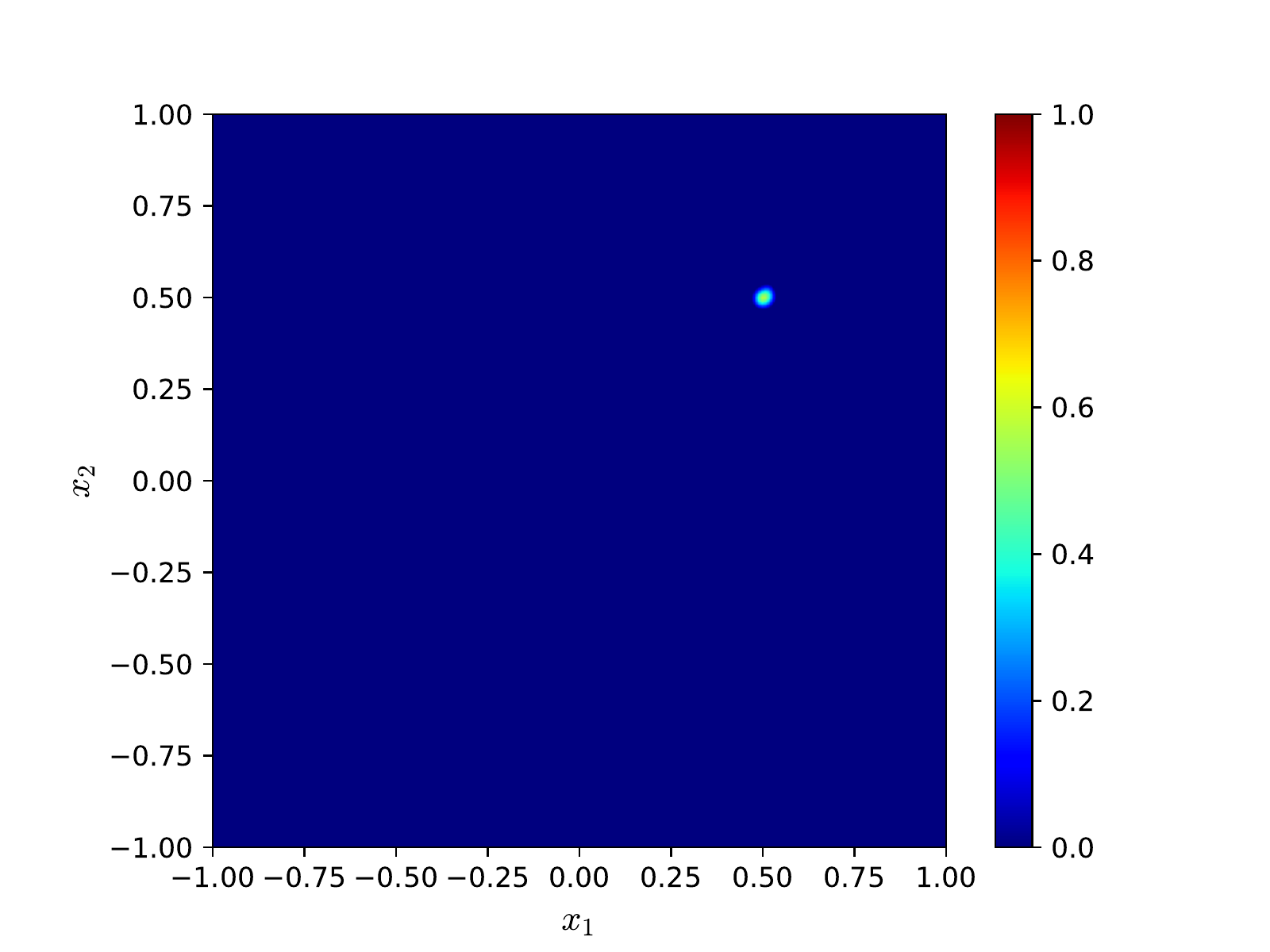}}
	\caption{Solutions, two-dimensional test problem with two peaks. }
	\label{fig:2peaks2d_sol}
\end{figure}

\begin{figure}
	\center{
		\includegraphics[width=0.7\textwidth]{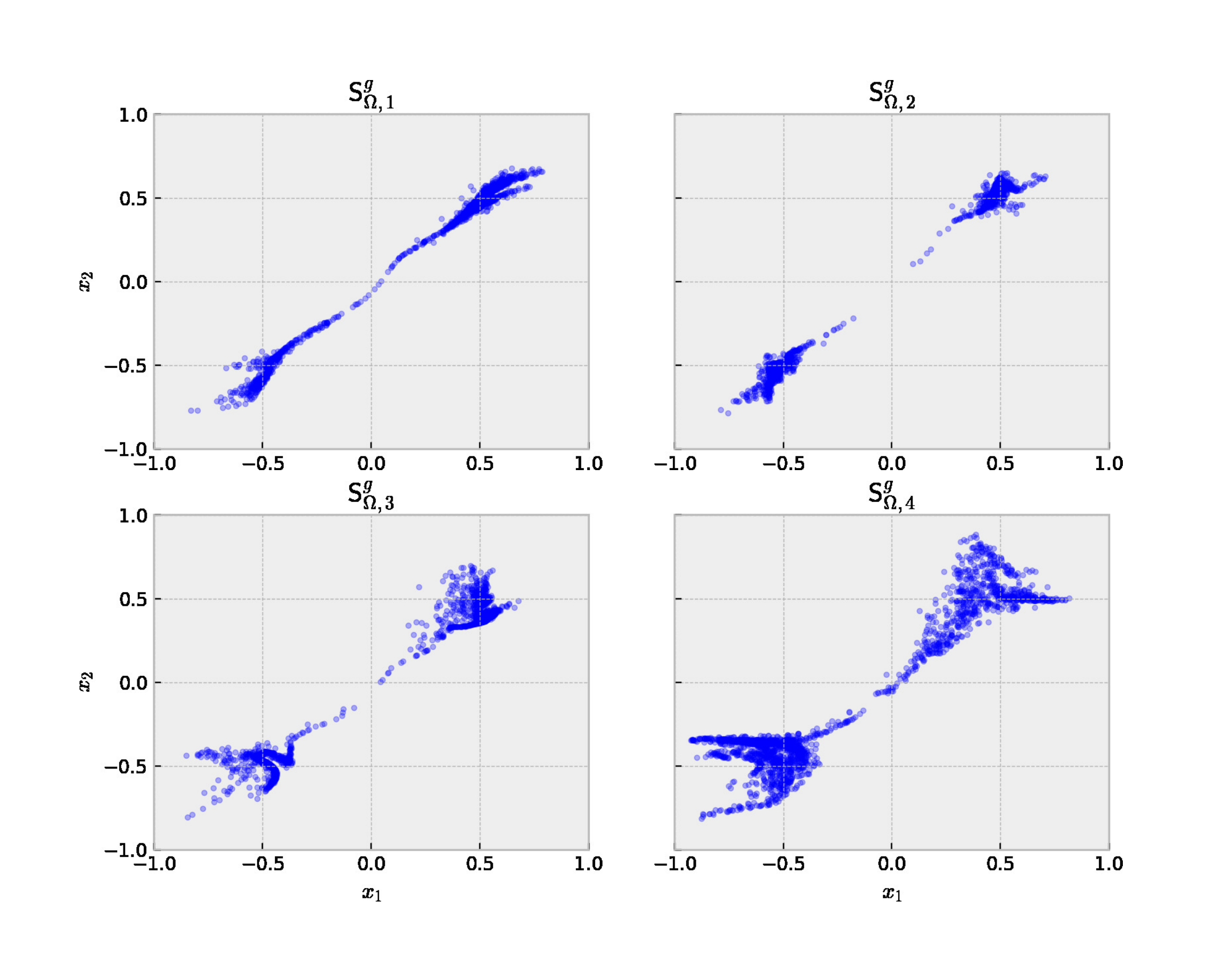}
	}
	\caption{The evolution of $\mathsf{S}^g_{\Omega, k}$ in DAS-G for the two-dimensional test problem with two peaks.}\label{fig:2peaks2d_dasg_resample}
\end{figure}

\subsection{High-dimensional linear test problems} \label{sec_numexp_hd_linear}
Next we consider the $d$-dimensional elliptic equation 
\begin{equation} \label{eq_highdim_linearpde}
	- \Delta u(\mb{x}) = s(\mb{x}), \quad  \mb{x}  \ \text{in} \ \Omega = [-1,1]^{d},
\end{equation}
with an exact solution
$$ u(\mb{x}) = \mathrm{e}^{-10 \norm{\mb{x}}{2}^2},$$ where the Dirichlet boundary condition on $\partial\Omega$ is given by the exact solution. We are interested in cases with a large $d>3$.
Note that the geometric properties of high-dimensional spaces are significantly different from our intuitions on low-dimensional ones, e.g., most of the volume of a high-dimensional cube is located around its corners \cite{vershynin2018high, blum2020foundations, wright2021high}. If we use uniform samples to generate $\mathsf{S}_{\Omega}$, most of the collocation points in $\mathsf{S}_{\Omega}$ are near the surface of the hypercube. Since the information of the exact solution is mainly from the neighborhood of the origin, most of the samples in $\mathsf{S}_{\Omega}$ may not contribute to training the neural network when $d$ is large enough. 

We choose a six-layer fully connected neural network $u(\mb{x};\Theta)$ with 64 neurons to approximate the solution. For KRnet, we set $K = 3$ and take $L = 6$ affine coupling layers, and two fully connected layers with $64$ neurons for each affine coupling layer. The number of epochs for training both $u(\mb{x};\Theta)$ and $p(\mb{x};\Theta_f)$ is set to $N_e = 3000$. The learning rate for ADAM optimizer is set to $0.0001$, and the batch size is set to $m = 5000$. The numbers of adaptivity iterations is set to $N_{\rm{adaptive}} = 5$. To measure the quality of approximation, we generate a tensor grid with $n_t^d$ points around the origin (in $[-0.1, 0.1]^d$) where $n_t$ is the number of nodes for each dimension. We define the relative error 
\begin{equation*} 
	\text{Relative error} = \frac{ \norm{ \mb{u}_{\mathsf{NN}} - \mb{u}}{2}}{\norm{\mb{u}}{2}},
\end{equation*}
where $\mb{u}_{\mathsf{NN}}$ and $\mb{u}$ denote two vectors whose elements are the function values of $u(\mb{x};\Theta)$ and $u(\mb{x})$ at the tensor grid respectively. 

We first investigate the relation between the error and the dimensionality $d$ when the uniform sampling strategy is employed. Figure \ref{fig:example_10d}(a) shows the relative errors in terms of a varying $d$ for a sample size $\vert \mathsf{S}_{\Omega} \vert = 2 \times 10^5$. To roughly match the number of grid points for different $d$, we set $n_t = 16, 6, 4, 3$ for $d = 4, 6, 8, 10$ respectively. It is seen that the relative error grows quickly to $\mathit{O}(1)$ as $d$ increases. However, as shown in Figure \ref{fig:example_10d}(b), all training losses are finally close to zero for $d = 4, 6, 8, 10$. This is consistent with the fact that in a high-dimensional space most of the uniform samples are located around the boundary, where the solution is close to zero. The optimizer is then in favor of the trivial solution since there are not sufficient samples to resolve the peak at the origin. This phenomenon demonstrates that the uniform sampling method may become less effective as $d$ increases and the convergence of the approximate solution is highly dependent on the choice of $\mathsf{S}_{\Omega}$ for a large $d$. 

Figure \ref{fig:exp10d_error_comparison} shows the relative errors for the uniform sampling strategy, the residual-based adaptive refinement (RAR) method proposed in \cite{lu2021deepxde}, DAS-R and DAS-G, where different numbers of samples $|\mathsf{S}_{\Omega}|$ are considered. For each $|\mathsf{S}_{\Omega}|$, we again take three runs with different random seeds for initialization and compute the mean error of the three runs as the final error. For the DAS-G strategy, the numbers of collocation points in $\mathsf{S}^g_{\Omega, k}$ ($k = 1,2,3,4$) are set to $n_r = 10^4, 2 \times 10^4, 3 \times 10^4, 4 \times 10^4$ for $|\mathsf{S}_{\Omega} |= 5 \times 10^4, 10^5, 1.5 \times 10^5, 2 \times 10^5$ respectively. For the uniform sampling strategy, we train the model with $1.5 \times 10^4$ epochs to match the total number of epochs of DAS methods. For the heuristic method RAR, the numbers of collocation points in $\mathsf{S}^g_{\Omega, k}$ ($k = 1,2,3,4$) are set to $n_r = 5 \times 10^3,  10^4, 1.5 \times 10^4, 2.5 \times 10^4$ for $|\mathsf{S}_{\Omega} |= 5 \times 10^4, 10^5, 1.5 \times 10^5, 2 \times 10^5$ respectively. From Figure \ref{fig:exp10d_error_comparison}, it can be seen that both DAS-G and DAS-R improve the accuracy significantly compared to the uniform sampling strategy and RAR. In addition, the error of DAS-G decreases slightly faster than that of DAS-R for this test problem. 
In Figure \ref{fig:exp10d_error_epoch} we compare the error evolution of different sampling strategies.  From the left plot of Figure \ref{fig:exp10d_error_epoch}, as the number of epochs increases, the errors of DAS-G and DAS-R decrease quickly, while the errors of RAR and the uniform sampling strategy do not decrease. This result suggests that for high-dimensional problems DAS methods are able to achieve a good approximation with a relatively small number of nonuniform samples while much more uniform samples are needed for the same accuracy.
The right plot of Figure \ref{fig:exp10d_error_epoch} shows the error of DAS-G at each adaptivity iteration step $k$. It is seen that the error drops dramatically after we refine the solution using $\mathsf{S}_{\Omega,1}$.

Figures  \ref{fig:exp10d_dasr_resample_67} and \ref{fig:exp10d_dasg_resample_67} show $3000$ samples from the training sets ($\vert \mathsf{S}_{\Omega} \vert = 2 \times 10^5$) DAS-R and DAS-G for the first four adaptivity iterations, where the components $x_6$ and $x_7$ are used for visualization. We have also checked the other components, and no significantly different results were found. For DAS-R, $3000$ samples are randomly chosen from $\mathsf{S}_{\Omega, k}$ ($k = 1,2,3,4$). For DAS-G, $3000$ samples for visualization are randomly selected from $\mathsf{S}^g_{\Omega, k}$ ($k = 1,2,3,4$). It is seen that the profile of $\mathsf{S}_{\Omega,k}$ is gradually flattened as $k$ increases, meaning the nonuniform samples are able to smooth the error profile which has a peak around the origin. As for DAS-G, the improvement takes a similar path. $\mathsf{S}_{\Omega,1}^g$ shows that the error profile has a peak around the origin. After the training set is augmented with $\mathsf{S}_{\Omega,1}$, the error profile becomes more flat as shown by the distribution of $\mathsf{S}_{\Omega,2}^g$. This is expected since more collocation points are added to the neighborhood of the origin which should reduce the error over there. Such a pattern is similar to what we have observed in Figure \ref{fig:2peaks2d_dasg_resample}. In Figure \ref{fig:exp10d_variance_comparison}, we compare the evolution of the variance of the residual for the training with $5 \times 10^4$ samples. We estimate the variance of residual using $59049$ grid points around the origin (these points are also used to compute the relative errors in the above discussion). It is clear that both DAS-R and DAS-G achieve the variance reduction significantly compared with RAR, which helps reduce the statistical error dramatically for a fixed sample size. Looking more closely, the variance of DAS-R has a transition between two consecutive adaptivity iterations, resulting in the oscillation of errors for DAS-R as observed in the left plot of Figure \ref{fig:exp10d_error_epoch}. From Figure \ref{fig:exp10d_variance_comparison}, it can be seen that DAS-G appears more robust than DAS-R for this test problem. We may adjust the communication pattern between the PDE model and the PDF model to reduce the  oscillations, which will be left for future study. 

\begin{figure}[ht]
	\centering
	\subfloat[][Error ]{\includegraphics[width=.42\textwidth]{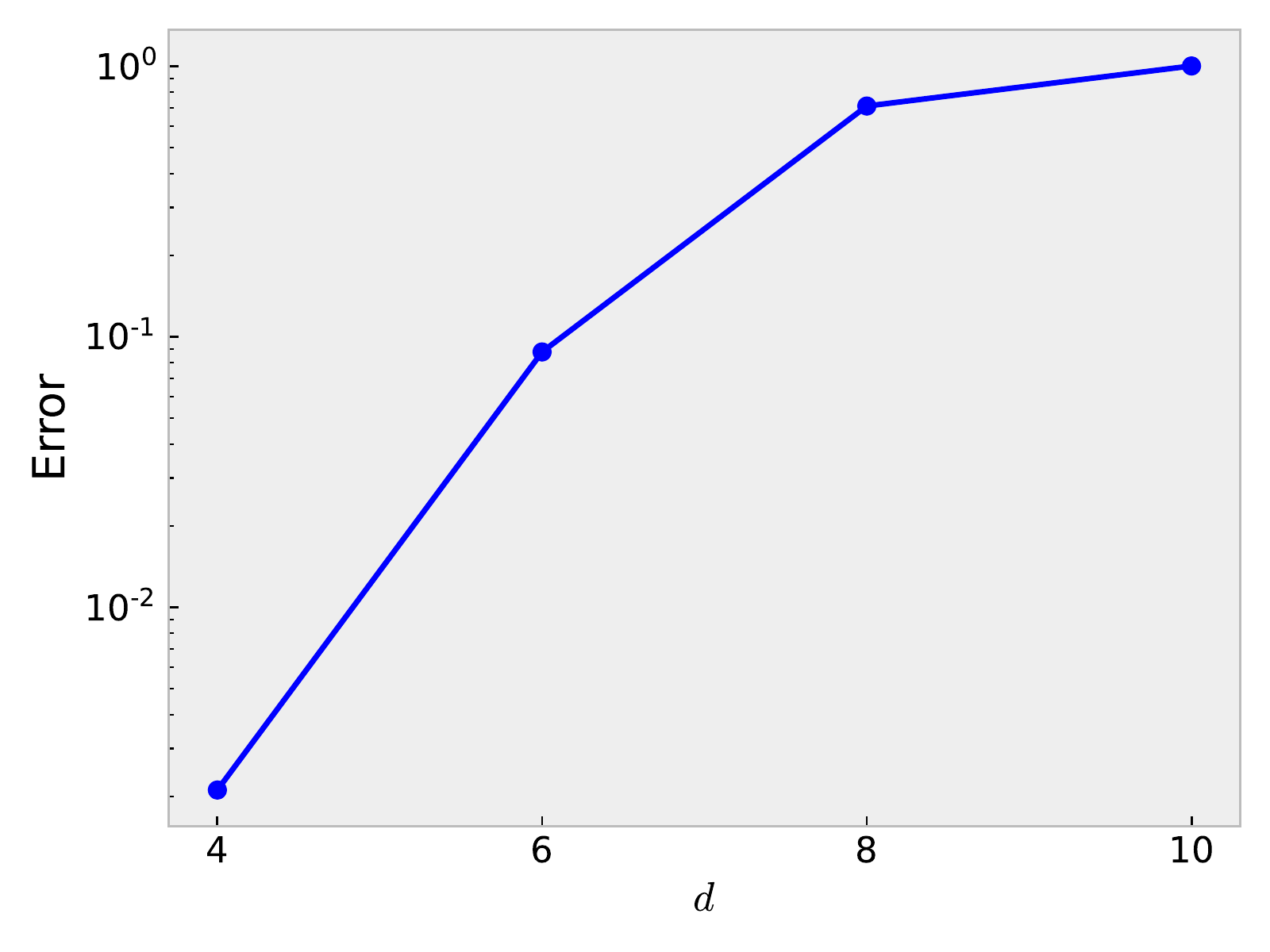}}\quad
	\subfloat[][Loss ]{\includegraphics[width=.42\textwidth]{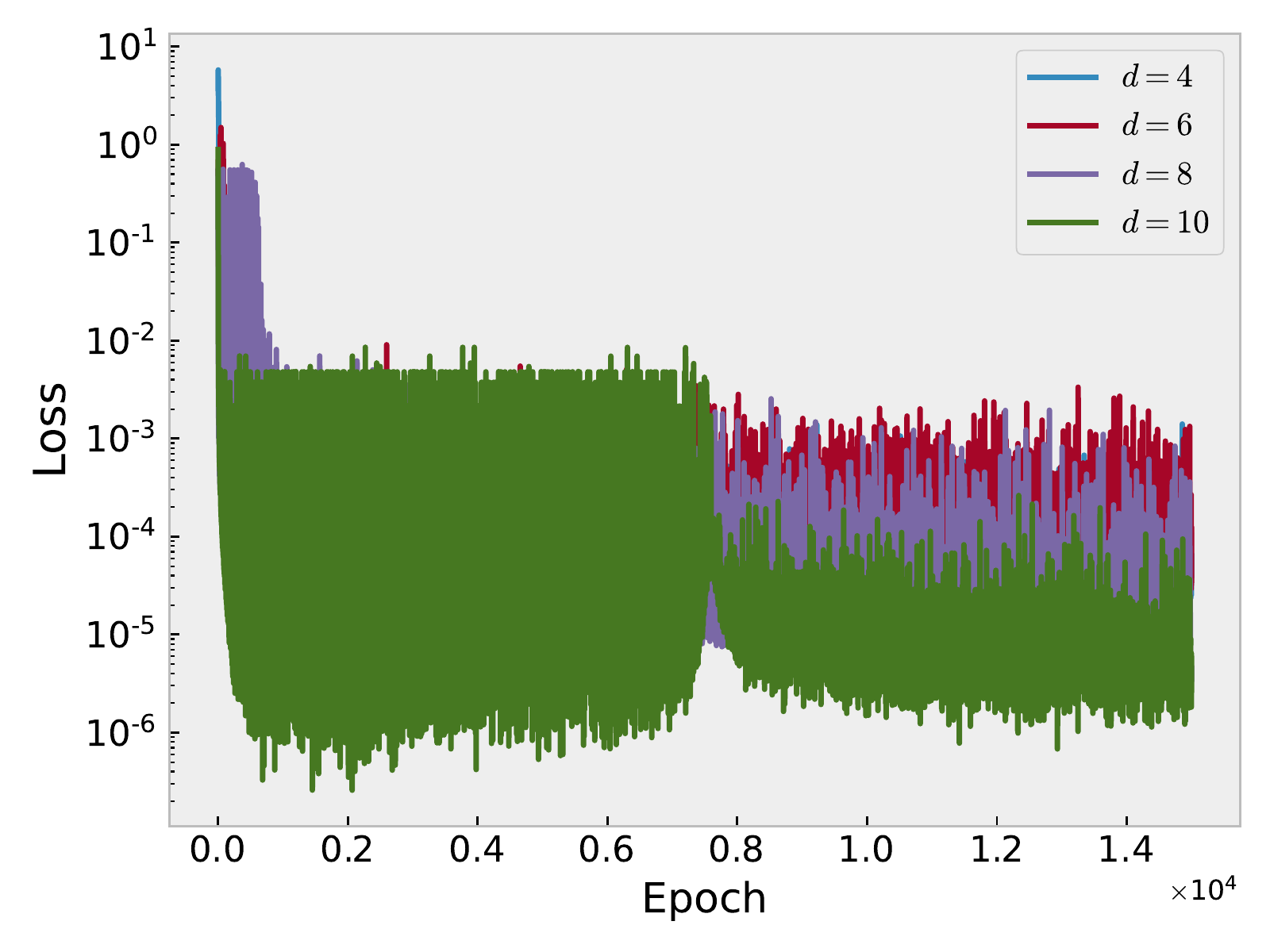}}
	\caption{The convergence behavior of high-dimensional PDEs with uniform sampling method. Loss is close to zero, but the error is still large for the ten-dimensional test problem. }
	\label{fig:example_10d}
\end{figure}

\begin{figure}
	\center{
		\includegraphics[width=0.48\textwidth]{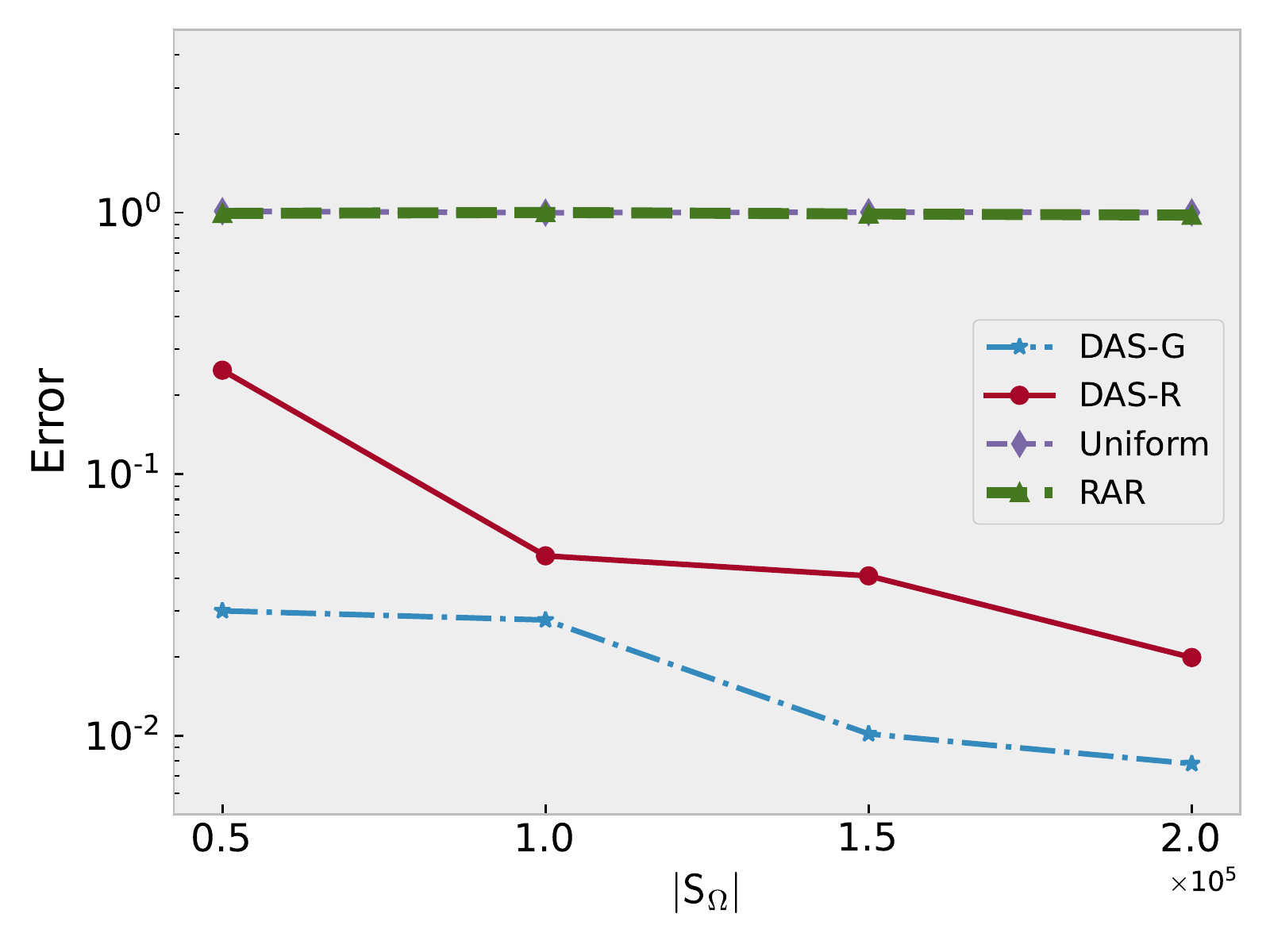}
	}
	\caption{The error w.r.t sample size $\vert \mathsf{S}_{\Omega} \vert$, ten-dimensional linear test problem.}\label{fig:exp10d_error_comparison}
\end{figure}

\begin{figure}
	\center{
		\includegraphics[width=0.42\textwidth]{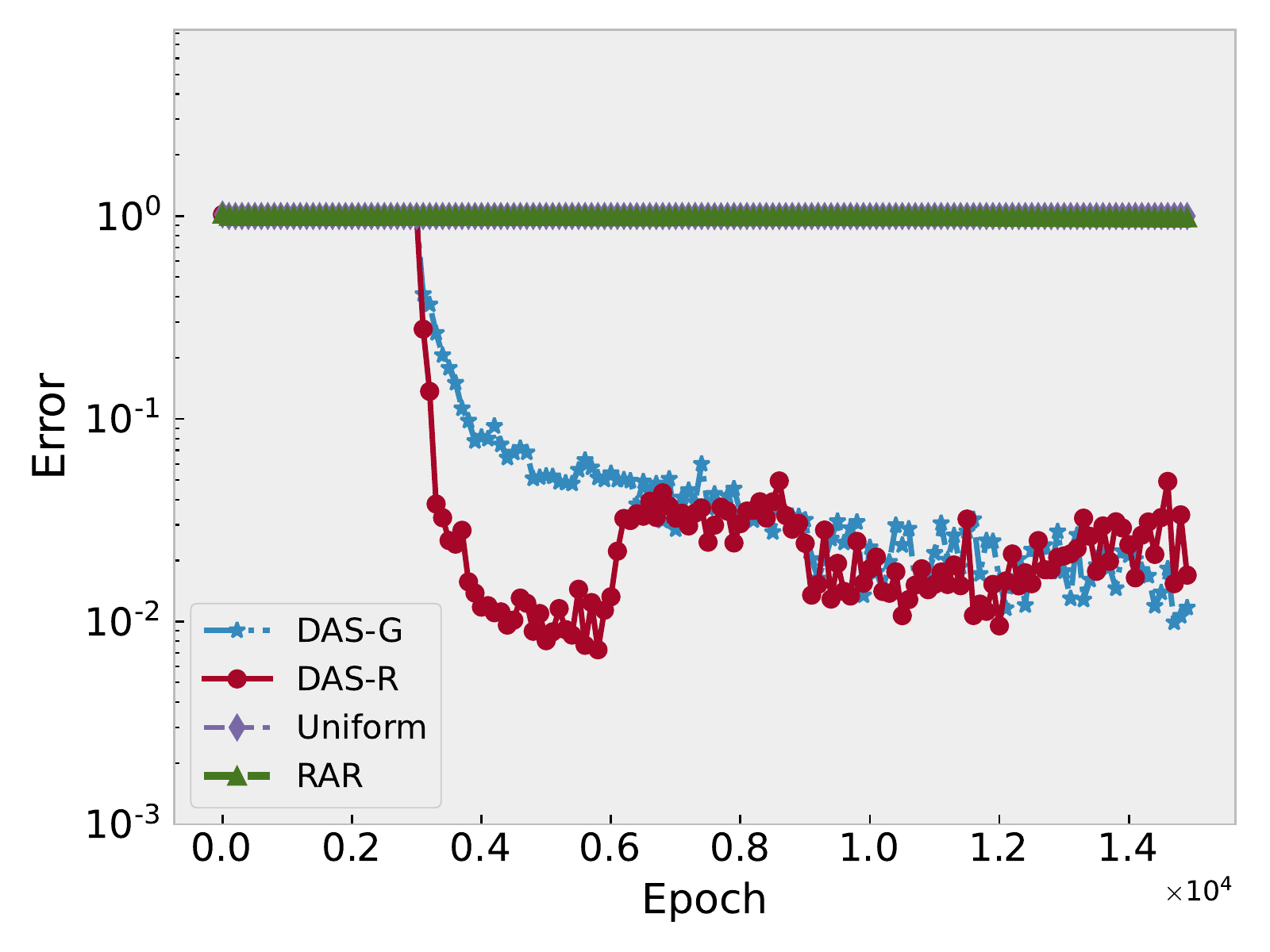}
		\includegraphics[width=0.46\textwidth]{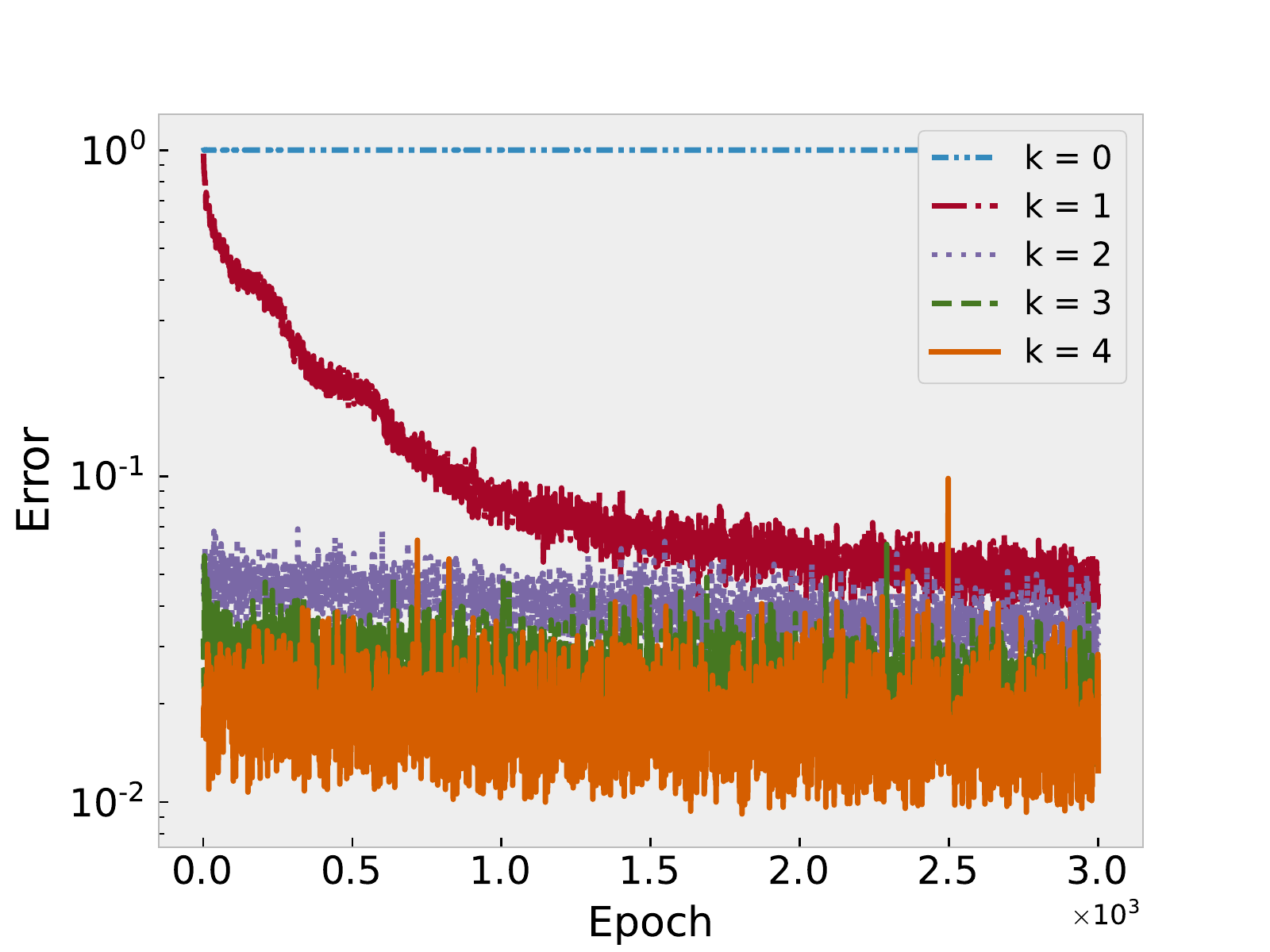}
	}
	\caption{The error evolution of different sampling strategies with $\vert \mathsf{S}_{\Omega} \vert = 2 \times 10^5$ and $d=10$ (high-dimensional linear test problem). Left: A comparison of DAS-G, DAS-R and the uniform sampling method; Right: The error evolution of DAS-G at different adaptivity iteration steps.}\label{fig:exp10d_error_epoch}
\end{figure}

\begin{figure}
	\center{
		\includegraphics[width=0.7\textwidth]{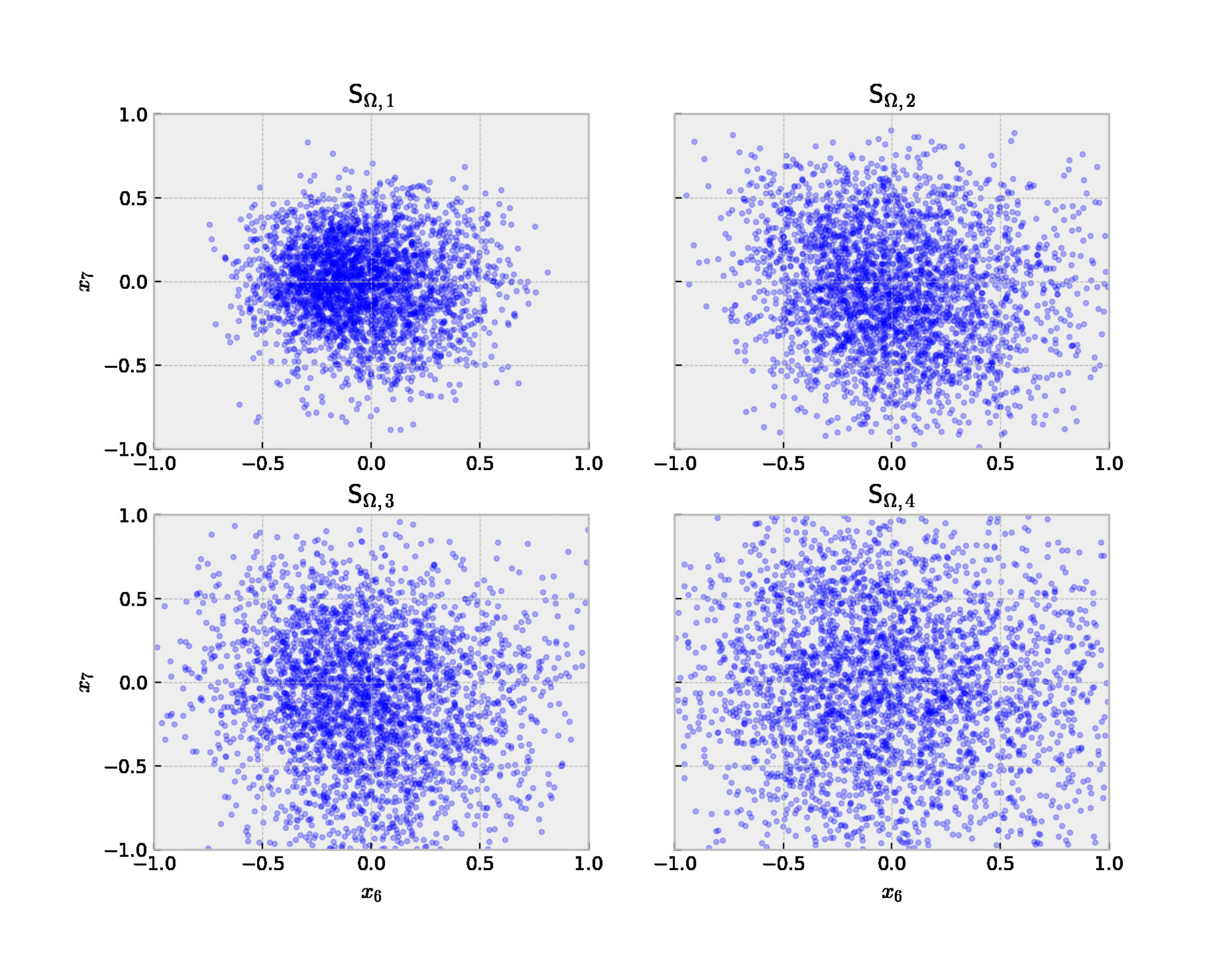}
	}
	\caption{The evolution of $\mathsf{S}_{\Omega, k}$ in DAS-R, ten-dimensional linear test problem.}\label{fig:exp10d_dasr_resample_67}
\end{figure}

\begin{figure}
	\center{
		\includegraphics[width=0.7\textwidth]{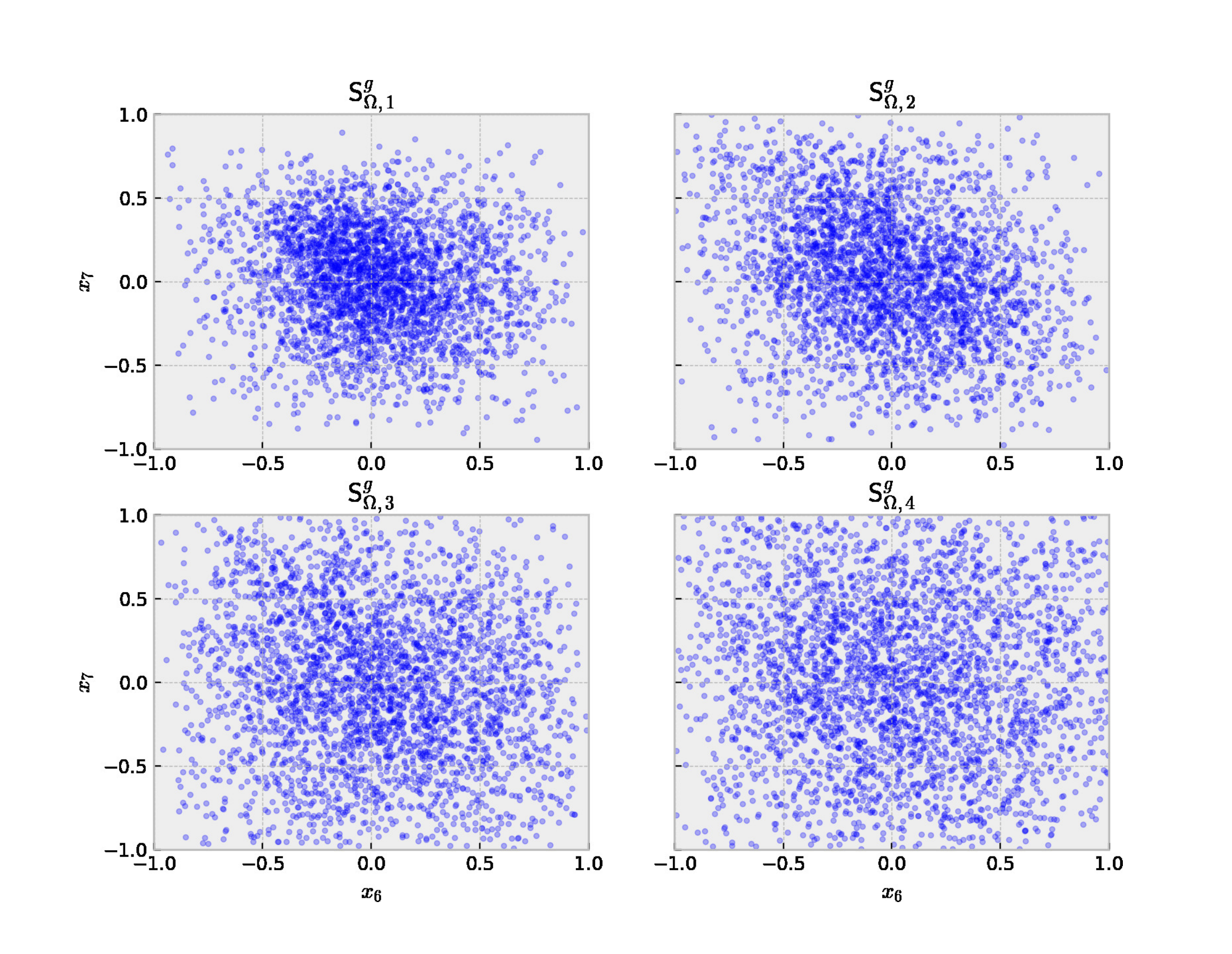}
	}
	\caption{The evolution of $\mathsf{S}^g_{\Omega, k}$ in DAS-G, ten-dimensional linear test problem.}\label{fig:exp10d_dasg_resample_67}
\end{figure}

\begin{figure}
	\center{
		\includegraphics[width=0.48\textwidth]{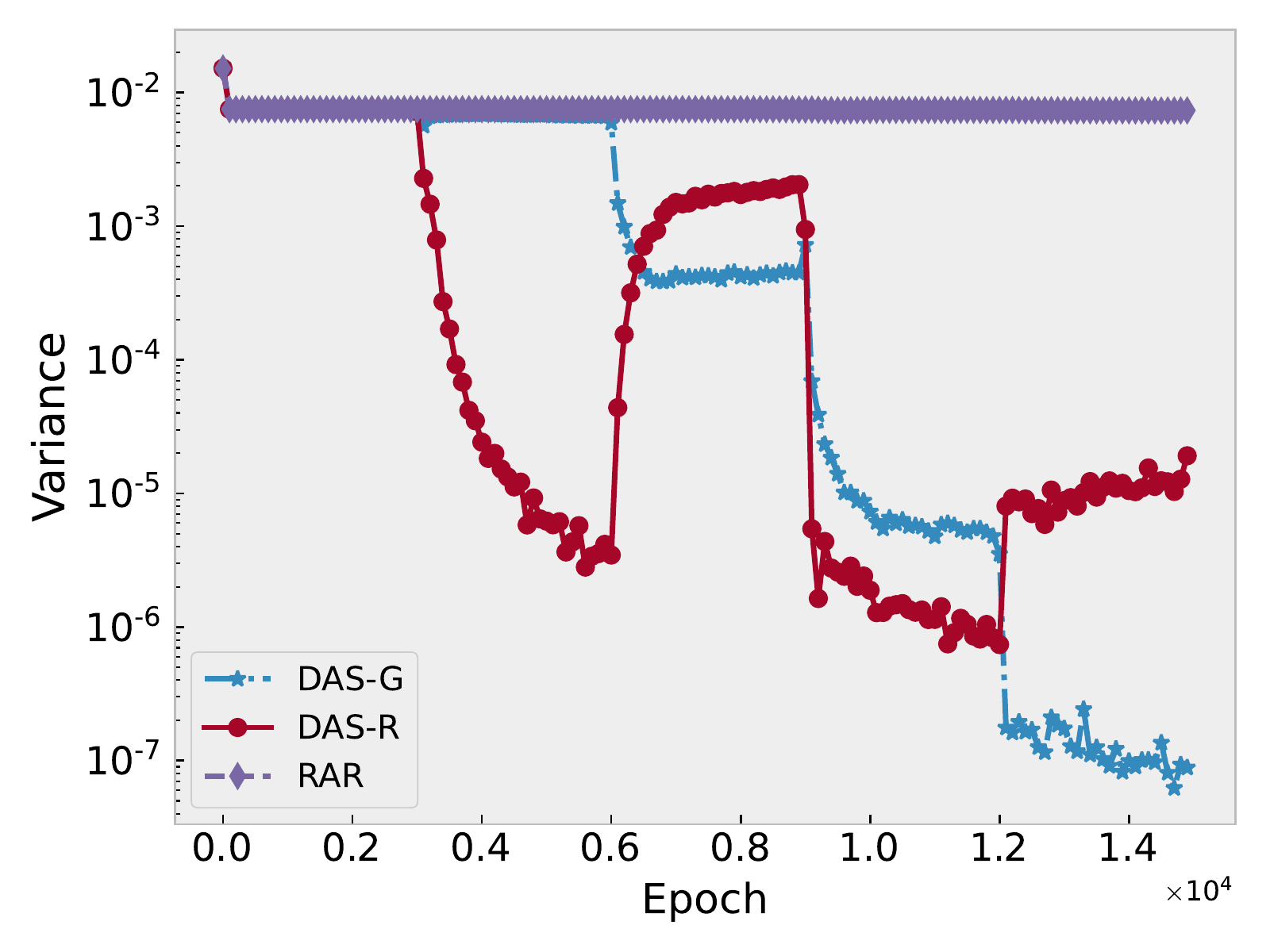}
	}
	\caption{The evolution for the variance of residual, ten-dimensional linear test problem.}\label{fig:exp10d_variance_comparison}
\end{figure}

\subsection{High-dimensional nonlinear test problem} \label{sec_numexp_hd_nonlinear}
In this part, the ten-dimensional nonlinear partial differential equation considered is 
\begin{equation}
	- \Delta u(\mb{x}) + u(\mb{x}) - u^3(\mb{x}) = s(\mb{x}), \quad  \mb{x}  \ \text{in} \ \Omega = [-1,1]^{10}.
\end{equation}
The exact solution is set to be the same as \eqref{eq_highdim_linearpde}, and the Dirichlet boundary condition on $\partial \Omega$ is given by the exact solution. The settings of $u(\mb{x};\Theta)$ and KRnet are the same as those in section \ref{sec_numexp_hd_linear}.

The observations are similar to those for the high-dimensional linear problem. Figure \ref{fig:nlexp10d_error_comparison} shows the relative errors for the uniform sampling strategy, RAR, DAS-R and DAS-G, where different numbers of samples $|\mathsf{S}_{\Omega}|$ are considered. Again, we take three runs with different random seeds for initialization and compute the mean error of the three runs as the final error for each $|\mathsf{S}_{\Omega}|$. For the DAS-G strategy, the numbers of collocation points in $\mathsf{S}^g_{\Omega, k}$ ($k = 1,2,3,4$) are set to the same as those in section \ref{sec_numexp_hd_linear}. For the uniform sampling strategy, we train the model with $1.5 \times 10^4$ epochs to match the total number of epochs of DAS methods. For the heuristic method RAR, the numbers of collocation points in $\mathsf{S}^g_{\Omega, k}$ ($k = 1,2,3,4$) are set to $n_r = 5 \times 10^3,  10^4, 1.5 \times 10^4, 2.5 \times 10^4$ for $|\mathsf{S}_{\Omega} |= 5 \times 10^4, 10^5, 1.5 \times 10^5, 2 \times 10^5$ respectively. Both DAS-G and DAS-R improve the accuracy significantly compared to the uniform sampling strategy and RAR. In Figure \ref{fig:nlexp10d_error_epoch} we compare the error evolution of different sampling strategies. Similar to the high-dimensional linear problem, the errors of DAS-G and DAS-R decrease quickly while the errors of the uniform sampling strategy and RAR do not decrease. The error behavior of DAS-G at each adaptivity iteration step $k$ is shown in the right plot of Figure \ref{fig:nlexp10d_error_epoch}. It is seen that the approximation is significantly improved when the adaptivity iteration step $k$ increases from $0$ to $1$. Figure  \ref{fig:nlexp10d_dasr_resample_67} and \ref{fig:nlexp10d_dasg_resample_67} show $3000$ samples from the training sets ($\vert \mathsf{S}_{\Omega} \vert = 2 \times 10^5$) DAS-R and DAS-G for the first four adaptivity iterations, where the components $x_6$ and $x_7$ are used for visualization. For DAS-R, $3000$ samples are randomly chosen from $\mathsf{S}_{\Omega, k}$ ($k = 1,2,3,4$). For DAS-G, $3000$ samples for visualization are randomly selected from $\mathsf{S}^g_{\Omega, k}$ ($k = 1,2,3,4$). Both DAS-R and DAS-G flatten the error profile through adaptive sampling as we have observed in Figures \ref{fig:exp10d_dasr_resample_67} and \ref{fig:exp10d_dasg_resample_67}. Figure \ref{fig:nlexp10d_variance_comparison} shows the evolution of the variance of the residual for DAS-R, DAS-G and RAR. The behavior is similar to that in Figure  \ref{fig:exp10d_variance_comparison}.

\begin{figure}
	\center{
		\includegraphics[width=0.48\textwidth]{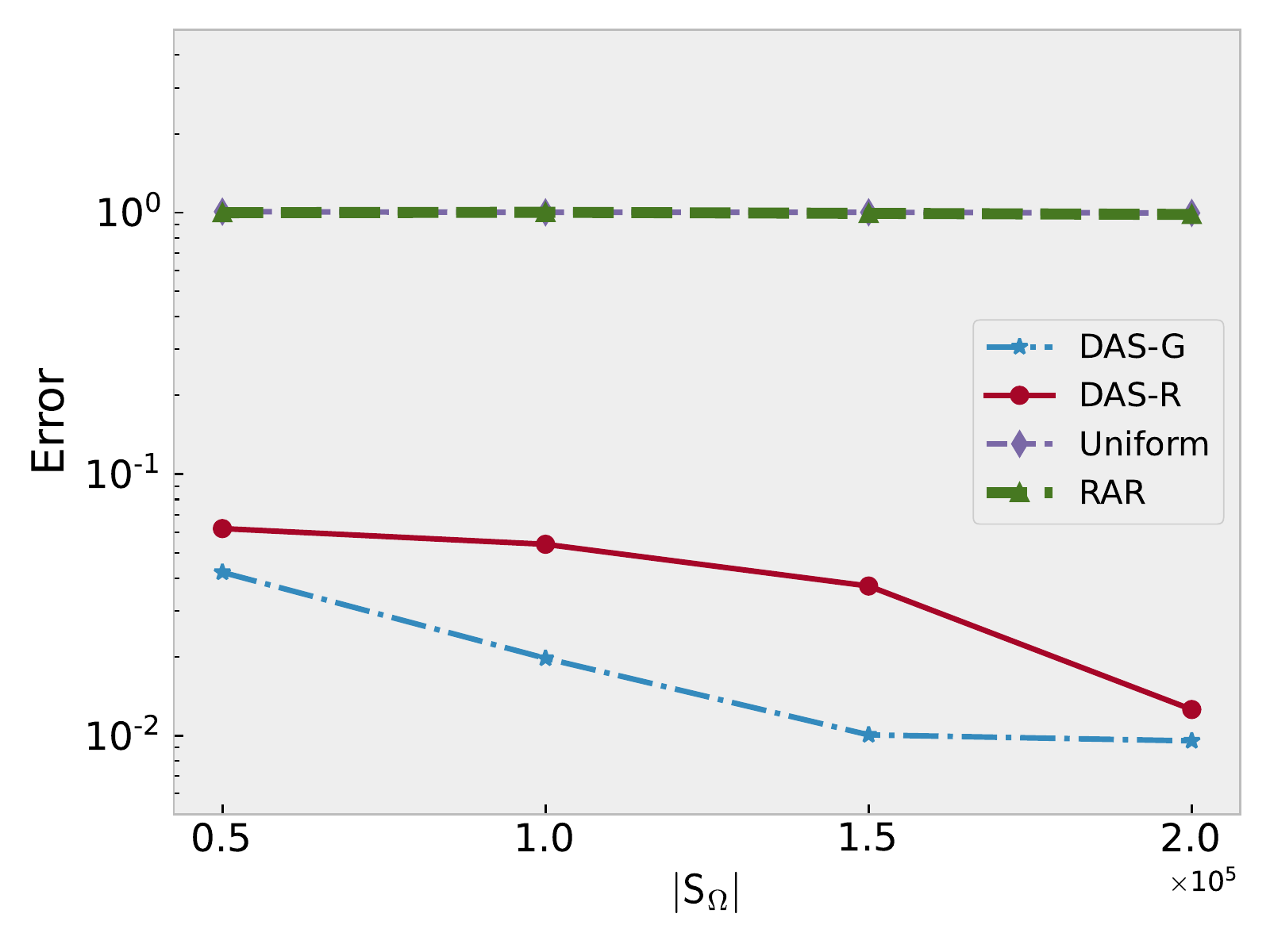}
	}
	\caption{The error w.r.t sample size $\vert \mathsf{S}_{\Omega} \vert$, ten-dimensional nonlinear test problem.}\label{fig:nlexp10d_error_comparison}
\end{figure}

\begin{figure}
	\center{
		\includegraphics[width=0.42\textwidth]{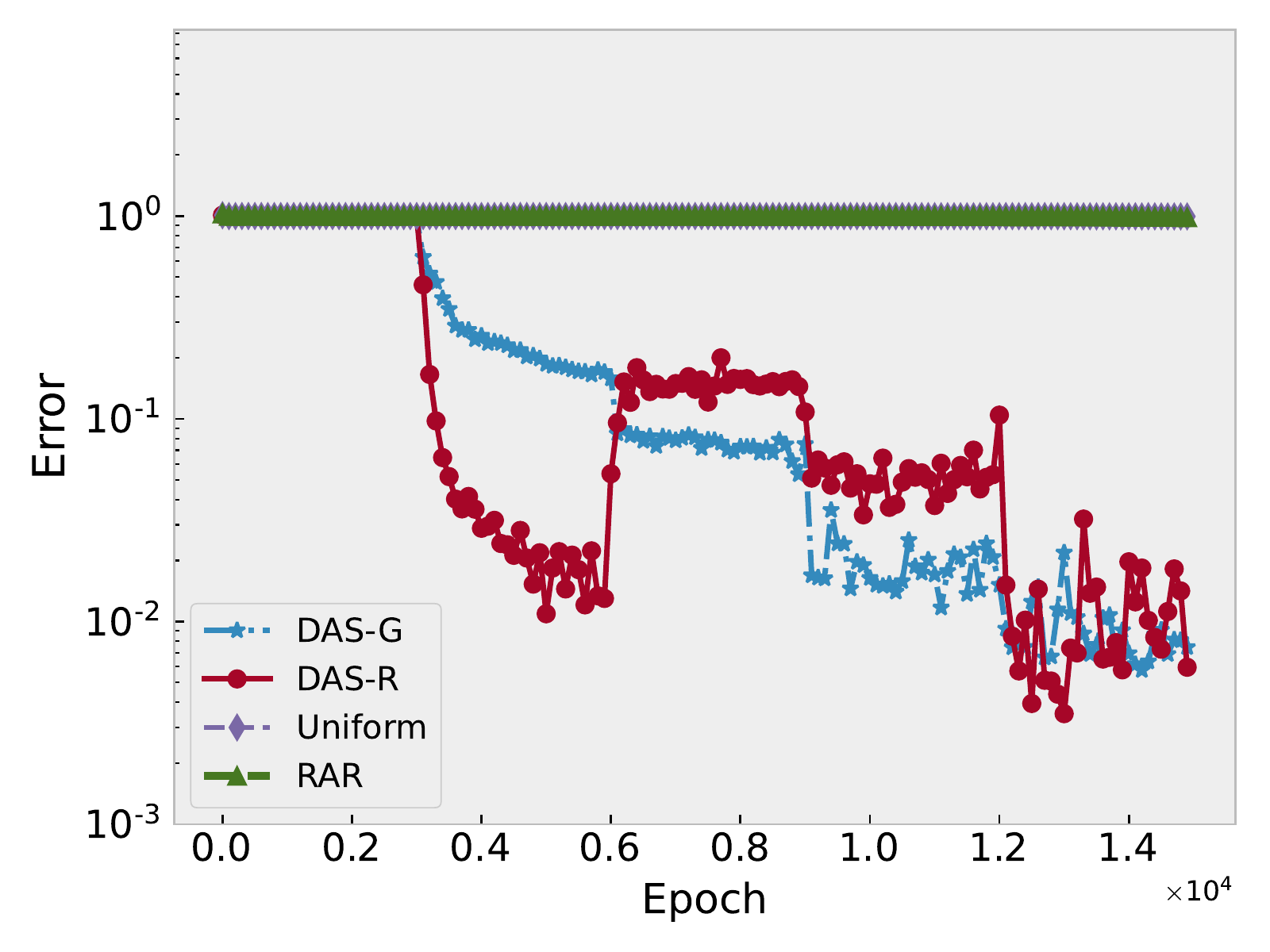}
		\includegraphics[width=0.46\textwidth]{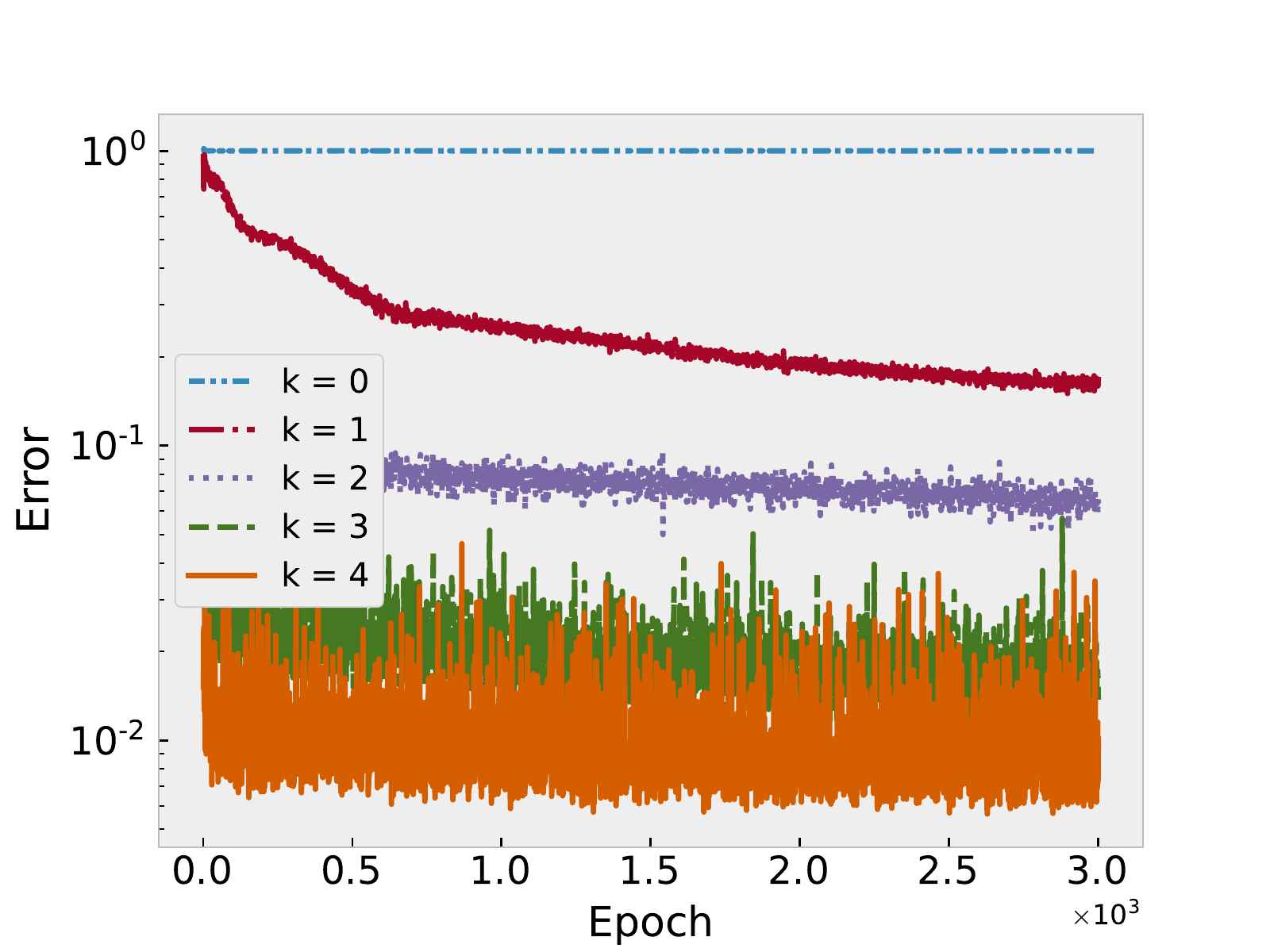}
	}
	\caption{The error evolution of different sampling strategies with $\vert \mathsf{S}_{\Omega} \vert = 2 \times 10^5$ and $d=10$ (high-dimensional nonlinear test problem). Left: A comparison of DAS-G, DAS-R and the uniform sampling method; Right: The error evolution of DAS-G at different adaptivity iteration steps.}\label{fig:nlexp10d_error_epoch}
\end{figure}

\begin{figure}
	\center{
		\includegraphics[width=0.7\textwidth]{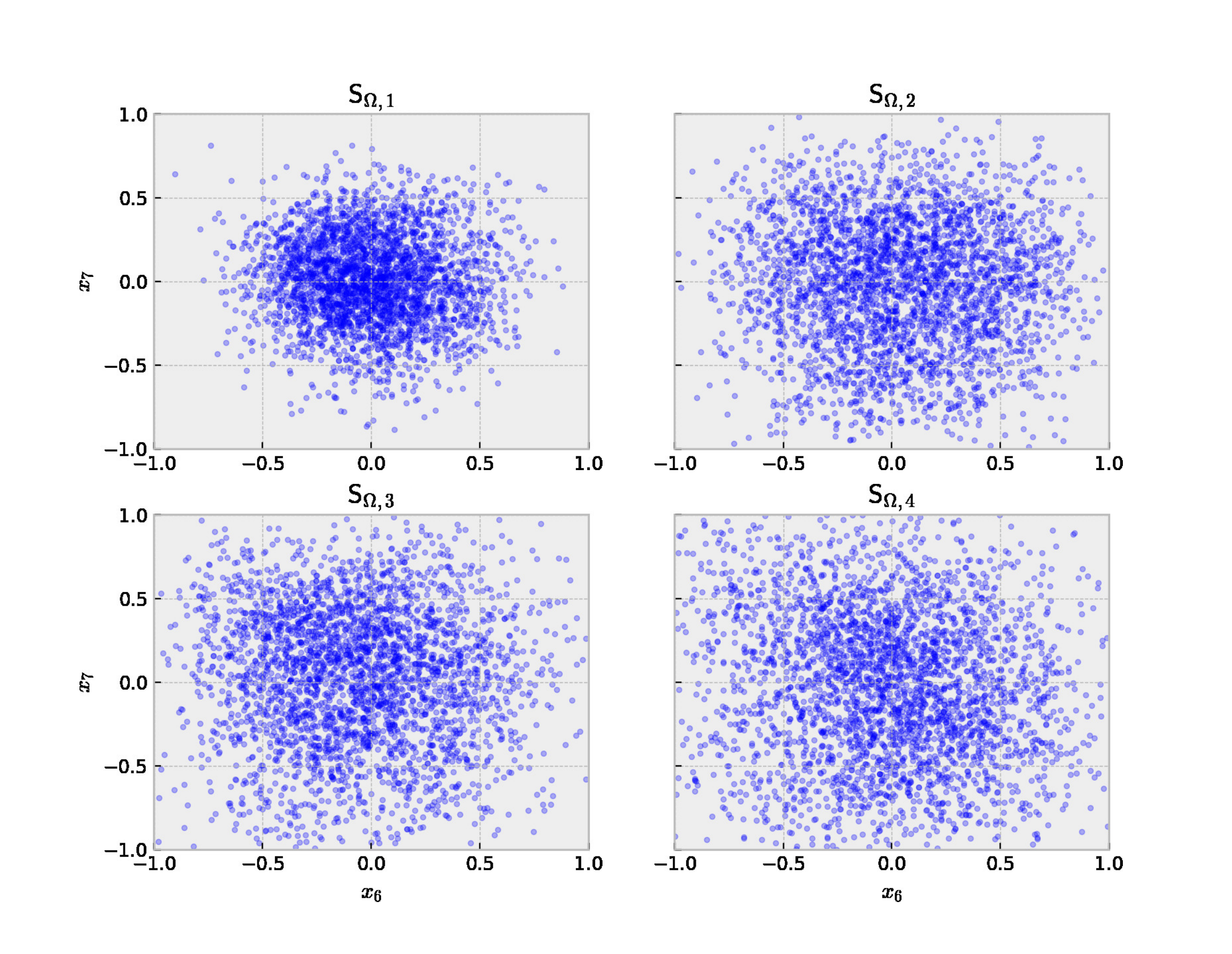}
	}
	\caption{The evolution of $\mathsf{S}_{\Omega, k}$ in DAS-R, ten-dimensional nonlinear test problem.}\label{fig:nlexp10d_dasr_resample_67}
\end{figure}

\begin{figure}
	\center{			\includegraphics[width=0.7\textwidth]{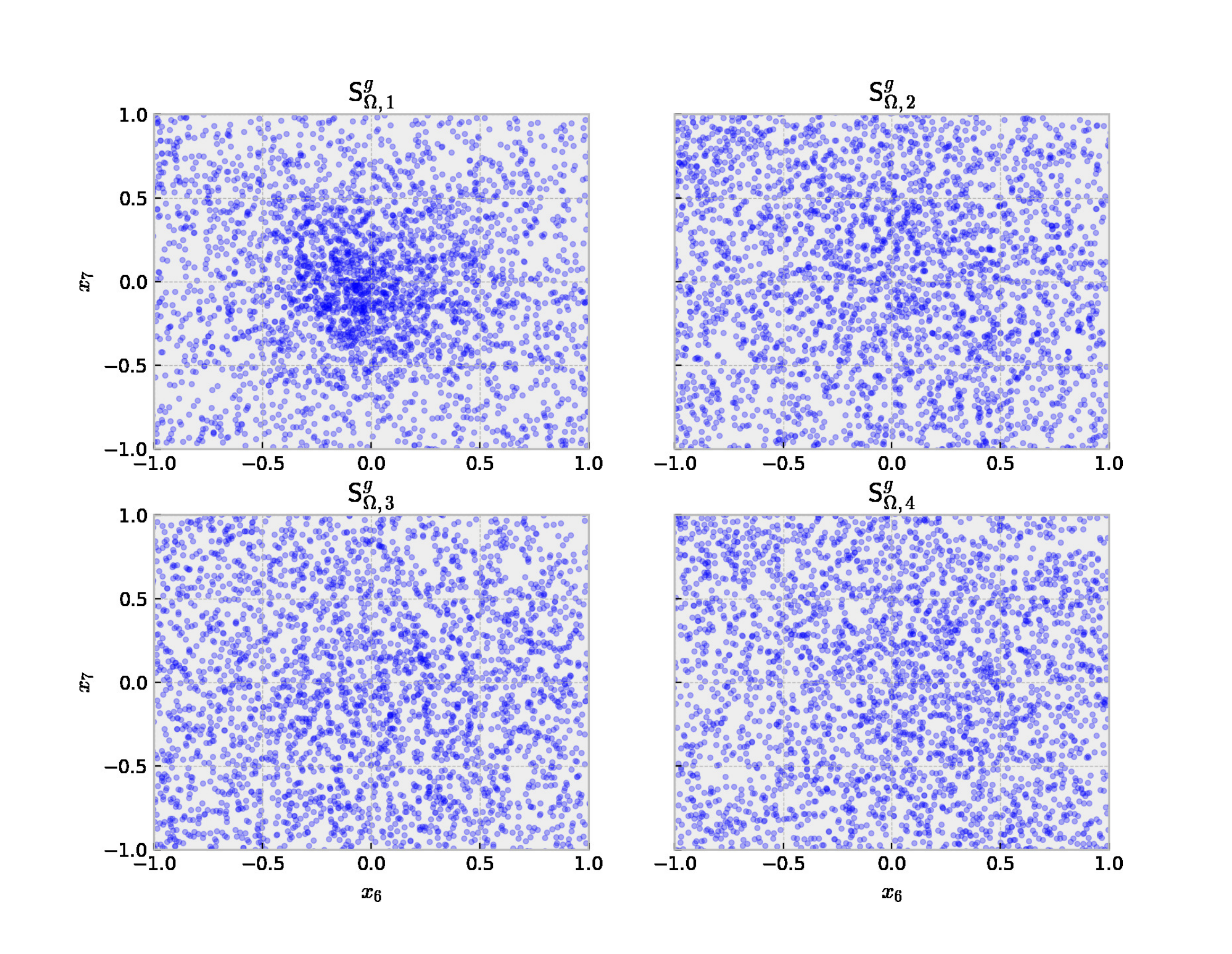}
	}
	\caption{The evolution of $\mathsf{S}^g_{\Omega, k}$ in DAS-G, ten-dimensional nonlinear test problem.}\label{fig:nlexp10d_dasg_resample_67}
\end{figure}

\begin{figure}
	\center{
		\includegraphics[width=0.48\textwidth]{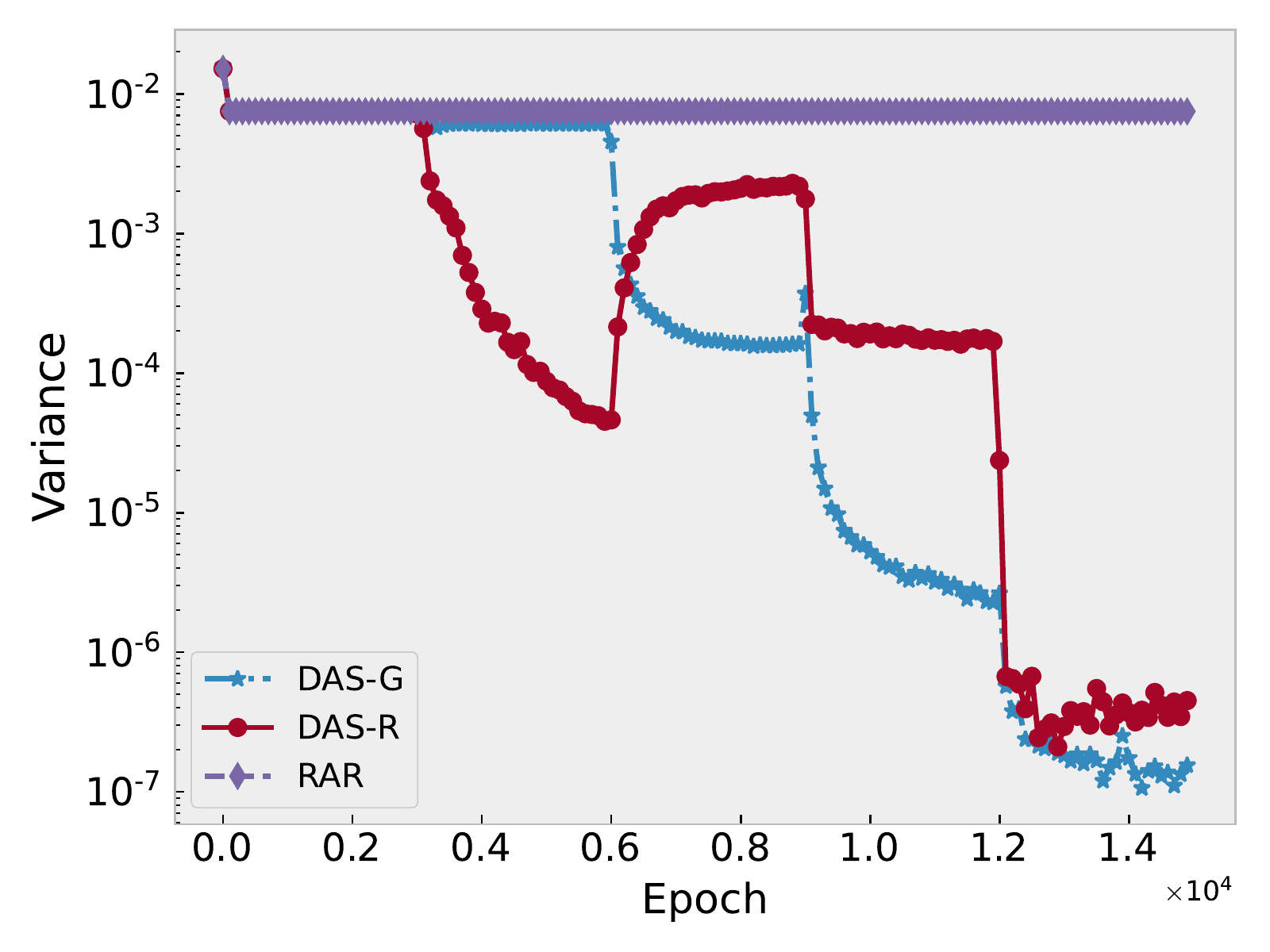}
	}
	\caption{The evolution for the variance of residual, ten-dimensional nonlinear test problem.}\label{fig:nlexp10d_variance_comparison}
\end{figure}

\section{Conclusion}\label{sec_conclusion}
	In this paper we have developed a deep adaptive sampling (DAS) method and coupled it with physics-informed neural networks (PINNs) to improve the neural network approximation of PDEs iteratively. The key idea of DAS is to employ a deep generative model to generate collocation points that are consistent with the distribution induced by an appropriate error indicator function. In this way, the training set is refined according to the regularity of the PDE solution, which follows the similar principle of  adaptive mesh refinement of classical numerical methods. Numerical experiments have shown that the DAS method is able to significantly improve the accuracy for the approximation of low regularity problems especially when the dimensionality is relatively large. The proposed DAS method provides a very general and flexible framework for an adaptive learning strategy.  There are several possible ways to further improve it. First, DAS consists of two DNN-based models: one model serves as an approximator for the PDE solution and the other one serves as an error indicator for the selection of collocation points. Both models can be chosen in terms of a certain criterion. In this work, we use a regular DNN for PDE approximation and KRnet for density approximation and sample generation.  Second, the underlying distribution for the training set can be problem dependent. In this work, we choose the residual-induced distribution. We may also use the gradient of the approximation solution to define an indicator distribution. In \cite{tang2021adaptive}, we employ KRnet to approximate the Fokker-Planck equation, where the collocation points are sampled from the approximate solution. Third, the DAS method is not limited to steady-state PDE problems. We may employ the DAS method on the space-time domain to refine the training set for the approximation of time-dependent problems. Last but not least, the current training process can also be improved. Although the current DAS methods work well enough to demonstrate the effectiveness of the algorithm, many questions remain open, e.g., what is the optimal way for the two deep models to communicate and what is the optimal sample size for $\mathsf{S}_{\Omega,k}^g$. Research on these issues will be reported in forthcoming papers. 
	
	\bigskip
	\textbf{Acknowledgments:}
	K. Tang has been supported by the China Postdoctoral 
	Science Foundation grant 2022M711730. X. Wan has been supported by NSF grant DMS-1913163.
	C. Yang has been supported by NSFC grant 12131002.
	%
	%
	\bibliography{tang}

\begin{thebibliography}{10}
\expandafter\ifx\csname url\endcsname\relax
  \def\url#1{\texttt{#1}}\fi
\expandafter\ifx\csname urlprefix\endcsname\relax\def\urlprefix{URL }\fi
\expandafter\ifx\csname href\endcsname\relax
  \def\href#1#2{#2} \def\path#1{#1}\fi

\bibitem{han2018solving}
J.~Han, A.~Jentzen, E.~Weinan, Solving high-dimensional partial differential
  equations using deep learning, Proceedings of the National Academy of
  Sciences 115~(34) (2018) 8505--8510.

\bibitem{weinan2021dawning}
E.~Weinan, The dawning of a new era in applied mathematics, Notices of the
  American Mathematical Society 68~(4) (2021) 565--571.

\bibitem{karniadakis2021physics}
G.~E. Karniadakis, I.~G. Kevrekidis, L.~Lu, P.~Perdikaris, S.~Wang, L.~Yang,
  Physics-informed machine learning, Nature Reviews Physics 3~(6) (2021)
  422--440.

\bibitem{weinan2018deep}
W.~E, B.~Yu, The deep {R}itz method: A deep learning-based numerical algorithm
  for solving variational problems, Communications in Mathematics and
  Statistics 6~(1) (2018) 1--12.

\bibitem{kharazmi2019variational}
E.~Kharazmi, Z.~Zhang, G.~E. Karniadakis, Variational physics-informed neural
  networks for solving partial differential equations, arXiv preprint
  arXiv:1912.00873.

\bibitem{zhu2019physics}
Y.~Zhu, N.~Zabaras, P.-S. Koutsourelakis, P.~Perdikaris, Physics-constrained
  deep learning for high-dimensional surrogate modeling and uncertainty
  quantification without labeled data, Journal of Computational Physics 394
  (2019) 56--81.

\bibitem{KharKarn2021}
E.~Kharazmi, Z.~Zhang, G.~E. Karniadakis, {hp-VPINNs}: Variational
  physics-informed neural networks with domain decomposition, Computer Methods
  in Applied Mechanics and Engineering 374 (2021) 113547.

\bibitem{sirignano2018dgm}
J.~Sirignano, K.~Spiliopoulos, D{GM}: A deep learning algorithm for solving
  partial differential equations, Journal of Computational Physics 375 (2018)
  1339--1364.

\bibitem{raissi2019physics}
M.~Raissi, P.~Perdikaris, G.~E. Karniadakis, Physics-informed neural networks:
  A deep learning framework for solving forward and inverse problems involving
  nonlinear partial differential equations, Journal of Computational Physics
  378 (2019) 686--707.

\bibitem{pang2019fpinns}
G.~Pang, L.~Lu, G.~E. Karniadakis, f{PINN}s: Fractional physics-informed neural
  networks, SIAM Journal on Scientific Computing 41~(4) (2019) A2603--A2626.

\bibitem{lagaris1998artificial}
I.~E. Lagaris, A.~Likas, D.~I. Fotiadis, Artificial neural networks for solving
  ordinary and partial differential equations, IEEE {T}ransactions on {N}eural
  {N}etworks 9~(5) (1998) 987--1000.

\bibitem{dissanayake1994neural}
M.~Dissanayake, N.~Phan-Thien, Neural-network-based approximations for solving
  partial differential equations, Communications in Numerical Methods in
  Engineering 10~(3) (1994) 195--201.

\bibitem{li2019d3m}
K.~Li, K.~Tang, T.~Wu, Q.~Liao, D{3M}: A deep domain decomposition method for
  partial differential equations, IEEE Access 8 (2020) 5283--5294.

\bibitem{deepdd}
W.~Li, X.~Xiang, Y.~Xu, Deep domain decomposition method: {E}lliptic problems,
  in: J.~Lu, R.~Ward (Eds.), Proceedings of The First Mathematical and
  Scientific Machine Learning Conference, Vol. 107 of Proceedings of Machine
  Learning Research, PMLR, Princeton University, Princeton, NJ, USA, 2020, pp.
  269--286.

\bibitem{dong2020local}
S.~Dong, Z.~Li, Local extreme learning machines and domain decomposition for
  solving linear and nonlinear partial differential equations, Computer Methods
  in Applied Mechanics and Engineering 387 (2021) 114129.

\bibitem{WangGao21}
H.~Gao, L.~Sun, J.-X. Wang, Phygeonet: Physics-informed geometry-adaptive
  convolutional neural networks for solving parameterized steady-state {PDEs}
  on irregular domain, Journal of Computational Physics 428 (2021) 110079.

\bibitem{sheng2020pfnn}
H.~Sheng, C.~Yang, P{FNN}: A penalty-free neural network method for solving a
  class of second-order boundary-value problems on complex geometries, Journal
  of Computational Physics (2020) 110085.

\bibitem{zang2020weak}
Y.~Zang, G.~Bao, X.~Ye, H.~Zhou, Weak adversarial networks for high-dimensional
  partial differential equations, Journal of Computational Physics 411 (2020)
  109409.

\bibitem{blum2020foundations}
A.~Blum, J.~Hopcroft, R.~Kannan, Foundations of data science, Cambridge
  University Press, 2020.

\bibitem{vershynin2018high}
R.~Vershynin, High-dimensional probability: An introduction with applications
  in data science, Vol.~47, Cambridge University Press, 2018.

\bibitem{wright2021high}
J.~Wright, Y.~Ma, High-dimensional data analysis with low-dimensional models:
  Principles, computation, and applications, Cambridge University Press, 2021.

\bibitem{tang2021adaptive}
K.~Tang, X.~Wan, Q.~Liao, Adaptive deep density approximation for
  {F}okker-{P}lanck equations, Journal of Computational Physics 457 (2022)
  111080.

\bibitem{gu2021selectnet}
Y.~Gu, H.~Yang, C.~Zhou, Selectnet: Self-paced learning for high-dimensional
  partial differential equations, Journal of Computational Physics 441 (2021)
  110444.

\bibitem{morin2002convergence}
P.~Morin, R.~H. Nochetto, K.~G. Siebert, Convergence of adaptive finite element
  methods, SIAM Review 44~(4) (2002) 631--658.

\bibitem{mekchay2005convergence}
K.~Mekchay, R.~H. Nochetto, Convergence of adaptive finite element methods for
  general second order linear elliptic pdes, SIAM Journal on Numerical Analysis
  43~(5) (2005) 1803--1827.

\bibitem{elman2014finite}
H.~C. Elman, D.~J. Silvester, A.~J. Wathen, Finite elements and fast iterative
  solvers: With applications in incompressible fluid dynamics, Oxford
  University Press, USA, 2014.

\bibitem{tangwandensity2020}
K.~Tang, X.~Wan, Q.~Liao, Deep density estimation via invertible
  block-triangular mapping, Theoretical \& Applied Mechanics Letters 10 (2020)
  143.

\bibitem{wan2020vae}
X.~Wan, S.~Wei, {VAE-KR}net and its applications to variational {B}ayes, arXiv
  preprint arXiv:2006.16431.

\bibitem{wan2021augmented}
X.~Wan, K.~Tang, Augmented {KR}net for density estimation and approximation,
  arXiv preprint arXiv:2105.12866.

\bibitem{santambrogio2015optimal}
F.~Santambrogio, Optimal transport for applied mathematicians, Birk{\"a}user,
  NY 55 (2015) 58--63.

\bibitem{dinh2016density}
L.~Dinh, J.~Sohl-Dickstein, S.~Bengio, Density estimation using real {NVP},
  arXiv preprint arXiv:1605.08803.

\bibitem{kingma2018glow}
D.~P. Kingma, P.~Dhariwal, Glow: Generative flow with invertible 1x1
  convolutions, in: Advances in Neural Information Processing Systems, 2018,
  pp. 10215--10224.

\bibitem{bottou2018optimization}
L.~Bottou, F.~E. Curtis, J.~Nocedal, Optimization methods for large-scale
  machine learning, SIAM Review 60~(2) (2018) 223--311.

\bibitem{kingma2017adam}
D.~P. Kingma, J.~Ba, Adam: A method for stochastic optimization, arXiv preprint
  arXiv:1412.6980.

\bibitem{bochev2016least}
P.~Bochev, M.~Gunzburger, Least-squares methods for hyperbolic problems, in:
  Handbook of Numerical Analysis, Vol.~17, Elsevier, 2016, pp. 289--317.

\bibitem{cybenko1989approximation}
G.~Cybenko, Approximation by superpositions of a sigmoidal function,
  Mathematics of Control, Signals and Systems 2~(4) (1989) 303--314.

\bibitem{hornik1989multilayer}
K.~Hornik, M.~Stinchcombe, H.~White, Multilayer feedforward networks are
  universal approximators, Neural Networks 2~(5) (1989) 359--366.

\bibitem{hornik1991approximation}
K.~Hornik, Approximation capabilities of multilayer feedforward networks,
  Neural {N}etworks 4~(2) (1991) 251--257.

\bibitem{leshno1993multilayer}
M.~Leshno, V.~Y. Lin, A.~Pinkus, S.~Schocken, Multilayer feedforward networks
  with a nonpolynomial activation function can approximate any function, Neural
  {N}etworks 6~(6) (1993) 861--867.

\bibitem{shin2020error}
Y.~Shin, Z.~Zhang, G.~E. Karniadakis, Error estimates of residual minimization
  using neural networks for linear {PDE}s, arXiv preprint arXiv:2010.08019.

\bibitem{lu2021priori}
J.~Lu, Y.~Lu, M.~Wang, A priori generalization analysis of the deep ritz method
  for solving high dimensional elliptic equations, arXiv preprint
  arXiv:2101.01708.

\bibitem{yu2022gradient}
J.~Yu, L.~Lu, X.~Meng, G.~E. Karniadakis, Gradient-enhanced physics-informed
  neural networks for forward and inverse {PDE} problems, Computer Methods in
  Applied Mechanics and Engineering 393 (2022) 114823.

\bibitem{gao2021active}
W.~Gao, C.~Wang, Active learning based sampling for high-dimensional nonlinear
  partial differential equations, arXiv preprint arXiv:2112.13988.

\bibitem{kobyzev2020normalizing}
I.~Kobyzev, S.~J. Prince, M.~A. Brubaker, Normalizing flows: An introduction
  and review of current methods, IEEE transactions on pattern analysis and
  machine intelligence 43~(11) (2020) 3964--3979.

\bibitem{goodfellow2014generative}
I.~Goodfellow, J.~Pouget-Abadie, M.~Mirza, B.~Xu, D.~Warde-Farley, S.~Ozair,
  A.~Courville, Y.~Bengio, Generative adversarial nets, in: Advances in
  {N}eural {I}nformation {P}rocessing {S}ystems, 2014, pp. 2672--2680.

\bibitem{arjovsky2017wasserstein}
M.~Arjovsky, S.~Chintala, L.~Bottou, Wasserstein generative adversarial
  networks, in: International {C}onference on {M}achine {L}earning, PMLR, 2017,
  pp. 214--223.

\bibitem{gulrajani2017improved}
I.~Gulrajani, F.~Ahmed, M.~Arjovsky, V.~Dumoulin, A.~C. Courville, Improved
  training of wasserstein gans, in: Advances in {N}eural {I}nformation
  {P}rocessing {S}ystems, Vol.~30, 2017.

\bibitem{kingma2014auto}
D.~P. Kingma, M.~Welling, Auto-{E}ncoding {V}ariational {B}ayes, stat 1050
  (2014) 1.

\bibitem{carlier2010knothe}
G.~Carlier, A.~Galichon, F.~Santambrogio, From {Knothe's transport to
  Brenier's} map and a continuation method for optimal transport, SIAM Journal
  on Mathematical Analysis 41~(6) (2010) 2554--2576.

\bibitem{de2005tutorial}
P.-T. De~Boer, D.~P. Kroese, S.~Mannor, R.~Y. Rubinstein, A tutorial on the
  cross-entropy method, Annals of {O}perations {R}esearch 134~(1) (2005)
  19--67.

\bibitem{rubinstein2013cross}
R.~Y. Rubinstein, D.~P. Kroese, The cross-entropy method: a unified approach to
  combinatorial optimization, Monte-Carlo simulation and machine learning,
  Springer Science \& Business Media, 2013.

\bibitem{lu2021deepxde}
L.~Lu, X.~Meng, Z.~Mao, G.~E. Karniadakis, Deepxde: A deep learning library for
  solving differential equations, SIAM Review 63~(1) (2021) 208--228.

\bibitem{mitchell2013collection}
W.~F. Mitchell, A collection of 2{D} elliptic problems for testing adaptive
  grid refinement algorithms, Applied {M}athematics and {C}omputation 220
  (2013) 350--364.

\end{thebibliography}

	\appendix
	\begin{appendices}
		\section{Proof of Lemma \ref{lem:fn_stat_err}}
		\begin{proof}
		We first introduce the following lemma:
		\begin{lem}\label{lem:I_delta_A}
			Consider a perturbed identity matrix $\mathbf{I}+\delta\mathbf{A}$ with  $\norm{\delta\mathbf{A}}{2} < 1$. We have
			\begin{equation}
				\norm{(\mathbf{I}+\delta\mathbf{A})^{-1}}{2} \leq\frac{1}{1- \norm{\delta\mathbf{A}}{2}}.
			\end{equation}
		\end{lem}
		\begin{proof}
			For any $\mb{x}\neq 0$, we have
			\[
			\norm{(\mathbf{I}+\delta\mathbf{A})\bx}{2} \geq \norm{\mb{x}}{2} - \norm{\delta\mathbf{A}}{2} \norm{\bx}{2} = (1-\norm{\delta\mathbf{A}}{2}) \norm{\mb{x}}{2} > 0,
			\]
			which implies that $\mathbf{I}+\delta\mathbf{A}$ is nonsingular. We then have
			\begin{align*}
				1 = \norm{(\mathbf{I}+\delta\mathbf{A})^{-1}(\mathbf{I}+\delta\mathbf{A})}{2} &= \norm{(\mathbf{I}+\delta\mathbf{A})^{-1}+(\mathbf{I}+\delta\mathbf{A})^{-1}\delta\mathbf{A}}{2}\\
				&\geq \norm{(\mathbf{I}+\delta\mathbf{A})^{-1}}{2} - \norm{(\mathbf{I}+\delta\mathbf{A})^{-1}}{2} \norm{\delta\mathbf{A}}{2},
			\end{align*}
			which yields the conclusion.
		\end{proof}
		
		Assume that $m_V^*(\mb{x})=(\mb{v}^*)^{\mathsf{T}}\mb{q}(\mb{x})$ and $m_{\hat{\mb{v}}^*}(\mb{x})=(\hat{\mb{v}}^*)^{\mathsf{T}}\mb{q}(\mb{x})$, where $\mb{q}(\mb{x})=[q_1(\mb{x}),\ldots,q_n(\mb{x})]^{\mathsf{T}}$ includes the basis functions in $V=\mathrm{span}\{q_i(\mb{x})\}_{i=1}^n$ and the vectors $\mb{v}^*$ and $\hat{\mb{v}}^*$ include the coefficients. It is easy to see that $\mb{v}^*$ and $\hat{\mb{v}}^*$ satisfy the following linear systems respectively
		\begin{equation}\label{eqn:fa_linear_sys}
			\mathbf{A}\mb{v}^*=\mb{b},\quad\hat{\mathbf{A}}\hat{\mb{v}}^*=\hat{\mb{b}},
		\end{equation}
		where
		\begin{align*}
			\hat{a}_{ij}&=\frac{1}{N}\sum_{k=1}^Nq_i(\bx^{(k)})q_j(\bx^{(k)})\approx\langle q_i,q_j\rangle_{\rho}=a_{ij},\\
			\hat{b}_i&=\frac{1}{N}\sum_{k=1}^Nq_i(\bx^{(k)}) h(\bx^{(k)})\approx\langle q_i, h\rangle_{\rho}=b_i,
		\end{align*}
		and $\langle q_i,q_j\rangle_\rho=\int_Dq_i(\mb{x})q_j(\mb{x})\rho(\mb{x})d\mb{x}$ indicates the inner product of $q_i$ and $q_j$. We rewrite the linear system for $\hat{\mb{v}}^*$ as
		\begin{equation}\label{eqn:fa_perturbed}
			(\mathbf{A}+\delta\mathbf{A})\hat{\mb{v}}^*=\mb{b}+\delta\mb{b},
		\end{equation}
		where $\delta{\mathbf{A}}=\hat{\mathbf{A}}-\mathbf{A}$ and $\delta\mb{b}=\hat{\mb{b}}-\mb{b}$. 
		Let $ \mathscr{B}=\{q_i(\mb{x})\}_{i=1}^n\cup\{h(\mb{x})\}$. Since both $\{q_{i}\}_{i=1}^n$ and $h(\mb{x})$ are continuous on a compact support, we may assume that $|m_1(\mb{x})m_2(\mb{x})|\leq M$ for any $m_1,m_2 \in \mathscr{B}$ and any $\mb{x}\in D$, where $0<M<\infty$ is a constant. Using the  Heoffding bound, we have for any $\delta>0$ with probability at least $1-2\delta$, 
		\begin{equation}
			\left|\frac{1}{N}\sum_{k=1}^N m_1(\bx^{(k)})m_2(\bx^{(k)})-\langle m_1,m_2\rangle_{\rho}\right|\leq \sqrt{\frac{2M^2\ln\delta^{-1}}{N}},
		\end{equation} 
		for any $m_1,m_2 \in \mathscr{B}$. This means with probability at least $1-2\delta$ we have
		\begin{equation}\label{eqn:fa_tail}
			\norm{\delta\mathbf{A}}{2} \leq \norm{\delta\mathbf{A}}{F} \leq n\sqrt{\frac{2M^2\ln\delta^{-1}}{N}}, \quad \norm{\delta\bb}{2} \leq \sqrt{n}\sqrt{\frac{2M^2\ln\delta^{-1}}{N}},
		\end{equation}
		where $\norm{\cdot}{F}$ indicates the Frobenius norm. Let $\delta\mb{v}^*=\hat{\mb{v}}^*-\mb{v}^*$. It is seen that $\norm{\delta\mathbf{A}}{2} \rightarrow 0$ as $N \rightarrow \infty$. Assume that $N$ is large enough such that $\norm{\delta\mathbf{A}}{2} \leq (1-r)$ with $0<r<1$, in other words, $(1 - \norm{\delta\mathbf{A}}{2} )^{-1}\leq r^{-1}$.  Since $q_i(\mb{x})$ are orthonormal, we have $\mathbf{A}=\mathbf{I}$. From equations \eqref{eqn:fa_linear_sys} and \eqref{eqn:fa_tail}, we have
		\[
		\delta\mb{v}^*=(\mathbf{I}+\delta\mathbf{A})^{-1}(\delta\mb{b}-\delta\mathbf{A}\mb{v}^*),
		\]
		to which we apply Lemma \ref{lem:I_delta_A} and the bounds in equation \eqref{eqn:fa_perturbed} and obtain
		\begin{equation}\label{eqn:fa_bound_dv}
			\norm{\delta\mb{v}^*}{2} \leq r^{-1}(\norm{\delta \mb{b}}{2} + \norm{\delta \mathbf{A}}{2} \norm{\mb{v}^*}{2}) \leq r^{-1}(\sqrt{n} + n \norm{\mb{v}^*}{2})\sqrt{\frac{2M^2\ln\delta^{-1}}{N}}.
		\end{equation}
		Using the Pythagorean theorem, we have
		\begin{align*}
			\|m_{\hat{\mb{v}}^*} - h\|_\rho^2 &= \|m_{\hat{\mb{v}}^*} - m^*_V\|_\rho^2+\|h - m^*_V\|_\rho^2\\
			&=\|\delta\mb{v}^* \|^2_2 + \|h - m^*\|_\rho^2\\
			&\leq (\|\delta\mb{v}^*\|_2+\|h - m^*_V\|_\rho)^2,
		\end{align*}
		which yields that
		\begin{equation}
			\label{eqn:fa_err_linear_sys}
			\|\hat{m}_{\hat{\mb{v}}^*} - h\|_\rho \leq \|\delta\mb{v}^*\|_2+\|h-m^*_V\|_\rho.
		\end{equation}
		Combing equations \eqref{eqn:fa_bound_dv} and \eqref{eqn:fa_err_linear_sys}, we reach the conclusion.
		
	 \end{proof}
 
		\section{Proof of Lemma \ref{lem:kl_is}}
		\begin{proof}
		Let 
		\[
		\hat{r}(\mb{x})=\left\{
		\begin{array}{rl}
			r(\mb{x}),&\textrm{ if }|r^2/p-\mu|\leq a;\\
			0,&\textrm{ otherwise},
		\end{array}
		\right.
		\]
		where $a>0$. We consider
		\begin{equation}
			\left|Q_p[r^2]-\mathbb{E}[r^2]\right|\leq \left|Q_p(r^2)-Q_p(\hat{r}^2)\right|+\left|Q_p(\hat{r}^2)-\mathbb{E}[\hat{r}^2]\right|+\left|\mathbb{E}[\hat{r}^2]-\mathbb{E}[r^2]\right|=I_1+I_2+I_3.
		\end{equation}
		The first term $I_1$ on the right-hand side is bounded as
		\begin{align}
			\mathbb{E}_p\left|Q_p(r^2)-Q_p(\hat{r}^2)\right|& = \mathbb{E}_p\left|r^2(\mb{X})p^{-1}(\mb{X})-\hat{r}^2(\mb{X})p^{-1}(\mb{X})\right|\nonumber\\
			&=\int_{|r^2/p-\mu|>a} r^2(\mb{x})d\mb{x}\nonumber\\
			&\leq\|r^2/p\|_{p} \sqrt{\mathbb{P}(|r^2/p-\mu|>a;p)},\label{eqn:01}
		\end{align}
		where the Cauchy-Schwarz inequality is used in the last step. 
		
		The second term $I_2$ on the right-hand side can be bounded as
		\begin{align}
			\mathbb{E}_p\left|Q_p(\hat{r}^2)-\mathbb{E}[\hat{r}^2]\right|&\leq \sqrt{\mathrm{Var}_p(Q_p(\hat{r}^2))}\nonumber\\
			&\leq N^{-1/2}\sqrt{\mathrm{Var}_p(\hat{r}^2(\mb{X})/p(\mb{X}))}\nonumber\\
			&\leq aN^{-1/2},\label{eqn:02}
		\end{align}
		where in the last step we used the fact that for any variable $\alpha\leq Y\leq \beta$, $\mathrm{Var}(Y)\leq\frac{(\alpha-\beta)^2}{4}$ with probability 1. The third term $I_3$ on the right-hand side can be bounded the same way as $I_1$. 
		
		We now estimate the tail probability $\mathbb{P}(|r^2/p-\mu|>a;p)$. Using the correspondence between $L_1$ norm and total variation distance for two probability measures as well as the Pinsker's inequality, we have
		\begin{equation}
			\|p-p^*\|_{L_1}=2\delta(p,p^*)\leq \sqrt{2D_{\mathsf{KL}}(p\|p^*)}\leq\sqrt{2\varepsilon},
		\end{equation}
		which yields that
		\begin{equation}
			\mathbb{E}_p\left[\left|r^2/p-\mu\right|\right]\leq \mu\sqrt{2\varepsilon}.
		\end{equation}
		From the Markov inequality, we have
		\begin{equation}\label{eqn:tail_prob}
			\mathbb{P}\left(\left|r^2/p-\mu\right|\geq a;p\right)\leq\frac{\mu\sqrt{2\varepsilon}}{a},
		\end{equation}
		where the probability is with respect to PDF $p(\mb{x})$. Combining the bounds for $I_i$, $i=1,2,3$, and equation \eqref{eqn:tail_prob}, we reach the conclusion.
	\end{proof}

		\section{Proof of Theorem \ref{thm_error_bound}}
		\begin{proof}
			By Assumption \ref{assump_norm_relations}, we have
			\begin{equation*}
				\begin{aligned}
					&\norm{u(\mb{x};\Theta_N^{*, (k)}) - u(\mb{x})}{2, \Omega}  \\
					\leq &  C_1^{-1} \left[ \norm{\mathcal{L}(u(\mb{x};\Theta_N^{*, (k)}) - u(\mb{x}))}{2, \Omega} + \norm{\mathfrak{b} (u(\mb{x};\Theta_N^{*, (k)}) - u(\mb{x}))}{2, \partial \Omega} \right]\\
					\leq & \sqrt{2} C_1^{-1} \left( \norm{\mathcal{L}(u(\mb{x};\Theta_N^{*, (k)}) - u(\mb{x}))}{2, \Omega}^2 + \norm{\mathfrak{b} (u(\mb{x};\Theta_N^{*, (k)}) - u(\mb{x}))}{2, \partial \Omega}^2 \right)^{\frac{1}{2}}.
				\end{aligned}
			\end{equation*}	
			Combining $\mathcal{L}u(\mb{x}) = s(\mb{x})$, $\mathfrak{b} u(\mb{x}) = g(\mb{x})$, $r(\mb{x};\Theta_N^{*, (k)}) = \mathcal{L} u(\mb{x};\Theta_N^{*, (k)}) - s(\mb{x})$ and  $b(\mb{x};\Theta_N^{*, (k)}) = \mathfrak{b} u(\mb{x};\Theta_N^{*, (k)}) - g(\mb{x})$ gives 
			\begin{equation} \label{eq_error_continuos}
				\norm{u(\mb{x};\Theta_N^{*, (k)}) - u(\mb{x})}{2, \Omega} \leq  \sqrt{2} C_1^{-1} \left( \norm{r(\mb{x};\Theta_N^{*, (k)})}{2, \Omega}^2 + \norm{b(\mb{x};\Theta_N^{*, (k)}) }{2, \partial \Omega}^2 \right)^{\frac{1}{2}}.
			\end{equation}
			Noting that $\mathbb{E}(R_k) = \norm{r(\mb{x};\Theta_N^{*, (k)})}{2, \Omega}^2$, and according to the Hoeffding inequality, we have
			\begin{equation} \label{eq_cheb}
				\mathbb{P} \left( R_k - \mathbb{E}(R_k)  \geq -\varepsilon  \right) \geq 1 - \mathrm{exp}\left( \frac{-2N_r \varepsilon^2} {(\tau_2-\tau_1)^2} \right).
			\end{equation}
			Combining \eqref{eq_error_continuos} and \eqref{eq_cheb} gives that
			\begin{equation*}
				\norm{u(\mb{x};\Theta_N^{*, (k)}) - u(\mb{x})}{2, \Omega} \leq \sqrt{2} C_1^{-1} \left( R_k + \varepsilon + \norm{b(\mb{x};\Theta_N^{*,(k)})}{2, \partial \Omega}^2 \right)^{\frac{1}{2}}
			\end{equation*}
			with probability at least $1 - \mathrm{exp}(-2N_r \varepsilon^2/(\tau_2-\tau_1)^2)$.
		\end{proof}

		\section{Proof of Corollary \ref{cor:R_k}}
		\begin{proof}
			Noting that 
			\begin{equation*}
				\Theta_N^{*, (k+1)} = \arg \min_{\Theta} \frac{1}{N_r} \sum_{i=1}^{N_r} \frac{r^2(\mb{x}_{\Omega}^{(i)};\Theta)}{\hat{p}_{\mathsf{KRnet}}(\mb{x}_{\Omega}^{(i)};\Theta_f^{*, (k)})}.
			\end{equation*} 
			Since $\Theta_N^{*, (k+1)}$ is the optimal solution at the $(k+1)$-th stage, we have 
			\begin{equation} \label{eq_discrete_rk}
				R_{k+1} = \frac{1}{N_r} \sum\limits_{i=1}^{N_r} \frac{r^2(\mb{x}_{\Omega}^{(i)};\Theta_N^{*, (k+1)})}{\hat{p}_{\mathsf{KRnet}}(\mb{x}_{\Omega}^{(i)};\Theta_f^{*, (k)})} \leq \frac{1}{N_r} \sum\limits_{i=1}^{N_r} \frac{r^2(\mb{x}_{\Omega}^{(i)};\Theta_N^{*, (k)})}{\hat{p}_{\mathsf{KRnet}}(\mb{x}_{\Omega}^{(i)};\Theta_f^{*, (k)})}.
			\end{equation}
			Plugging $\hat{p}_{\mathsf{KRnet}}(\mb{x};\Theta_f^{*, (k)}) = c_k r^2(\mb{x};\Theta_N^{*, (k)})$ into \eqref{eq_discrete_rk} gives that
			\begin{equation*}
				R_{k+1} \leq \frac{1}{c_k}.
			\end{equation*}	
			Noting that $R_{k+1}$ is a random variable and taking its expectation, it follows that
			\begin{equation*}
				\mathbb{E}(R_{k+1})  \leq \frac{1}{c_k} =  \int_{\Omega} r^2(\mb{x};\Theta_N^{*,(k)})d\mb{x} = \mathbb{E}(R_k),
			\end{equation*} 
			which completes the proof.
		\end{proof}
	\end{appendices}

\end{document}